\documentclass[11pt,onesided]{amsart}
\usepackage[utf8]{inputenc}
\usepackage[a4paper]{geometry}
\usepackage{amstext}
\usepackage{amsthm}
\usepackage{amssymb}
\usepackage[unicode=true,
 bookmarks=false,
 breaklinks=false,pdfborder={0 0 1},backref=section,colorlinks=false]
 {hyperref}

\makeatletter
\numberwithin{equation}{section}
\numberwithin{figure}{section}
\theoremstyle{plain}
\newtheorem{thm}{\protect\theoremname}
\theoremstyle{plain}
\newtheorem{prop}[thm]{\protect\propositionname}
\theoremstyle{remark}
\newtheorem{rem}[thm]{\protect\remarkname}
\theoremstyle{definition}
\newtheorem{defn}[thm]{\protect\definitionname}
\theoremstyle{plain}
\newtheorem{fact}[thm]{\protect\factname}
\theoremstyle{definition}
\newtheorem{example}[thm]{\protect\examplename}
\theoremstyle{plain}
\newtheorem{lem}[thm]{\protect\lemmaname}
\theoremstyle{remark}
\newtheorem{claim}[thm]{\protect\claimname}

\usepackage{latexsym}
\usepackage{amsfonts}\usepackage{enumerate}
\usepackage{bbm}
\usepackage{subfigure}
\usepackage[toc,page]{appendix}
\usepackage{multirow}
\usepackage{tikz}
\usepackage{tikz}
\usepackage{tikz-3dplot}

\usepackage{comment}

\usepackage{amsfonts}
\usepackage[all]{xy}
\usepackage{caption}
\usepackage{epstopdf}
\usepackage{color}
\usepackage{bbm}
\usepackage{dsfont}
\usepackage{wrapfig}\usepackage{dsfont}

\usepackage{tikz}
\usetikzlibrary{decorations.pathmorphing}
\tikzset{snake it/.style={decorate, decoration=snake}}
\usetikzlibrary{shapes.geometric,positioning,decorations.pathreplacing} 
\usepackage{pgfplots}

\newtheorem{theorem}{Theorem}

\def\beq{ \begin{equation} }
\def\eeq{ \end{equation} }

\def\square{\vcenter{\vbox{\hrule height .4pt
  \hbox{\vrule width .4pt height 5pt \kern 5pt
        \vrule width .4pt} \hrule height .4pt}}}

\newcommand{\bae}{\begin{equation}\begin{aligned}}
\newcommand{\eae}{\end{aligned}\end{equation}}

\DeclareFontFamily{OML}{rsfs}{\skewchar\font'177}
\DeclareFontShape{OML}{rsfs}{m}{n}{ <5> <6> rsfs5 <7> <8> <9>
rsfs7 <10> <10.95> <12> <14.4> <17.28> <20.74> <24.88> rsfs10 }{}
\DeclareMathAlphabet{\mathfs}{OML}{rsfs}{m}{n}

\usepackage[nomessages]{fp}

\providecommand{\claimname}{Claim}
\providecommand{\definitionname}{Definition}
\providecommand{\examplename}{Example}
\providecommand{\lemmaname}{Lemma}
\providecommand{\propositionname}{Proposition}
\providecommand{\remarkname}{Remark}
\providecommand{\theoremname}{Theorem}

\providecommand{\claimname}{Claim}
\providecommand{\definitionname}{Definition}
\providecommand{\examplename}{Example}
\providecommand{\lemmaname}{Lemma}

\providecommand{\propositionname}{Proposition}
\providecommand{\remarkname}{Remark}
\providecommand{\theoremname}{Theorem}

\makeatother

\providecommand{\claimname}{Claim}
\providecommand{\definitionname}{Definition}
\providecommand{\examplename}{Example}
\providecommand{\factname}{Fact}
\providecommand{\lemmaname}{Lemma}
\providecommand{\propositionname}{Proposition}
\providecommand{\remarkname}{Remark}
\providecommand{\theoremname}{Theorem}
\newtheorem{assumption}[theorem]{Assumption}
\newtheorem{notation}[theorem]{Notation}
\begin{document}
\title{Almost-Boundedness of Collatz-Type Maps on Number Rings via Archimedean
Methods}
\author{Maxwell Siegel}
\address{1626 Thayer Ave., Los Angeles, CA 90024}
\email{siegelmaxellc@ucla.edu}
\author{Rory O'Dwyer }
\address{382 Via Pueblo Mall, Stanford, CA 94305-4013}
\email{rodwyer@stanford.edu}
\begin{abstract}
Let $p$ be an integer $\geq2$, not necessarily prime, and let $K$
be a number field. In 2019, Tao showed that almost every positive
integer (in the sense of logarithmic density) had an ``almost bounded''
trajectory under the Collatz map. In this paper, we generalize this
result to a large class of Collatz-type maps (``$p$-Hydras'') acting
on the ring of integers $\mathcal{O}_{K}$ of a given number field
$K$ of dimension $d$ over $\mathbb{Q}$. Given an ideal $\Lambda\subseteq\mathcal{O}_{K}$
of index $p$, a $p$-Hydra map $H:\mathcal{O}_{K}\rightarrow\mathcal{O}_{K}$
is a collection of affine-linear maps $\left\{ H_{j}:K\rightarrow K\right\} _{j\in\mathcal{O}_{K}/\Lambda}$
of the form $z\mapsto r_{j}z+c_{j}$ for constants $c_{j}\in K$ and
invertible $\mathbb{Q}$-linear maps $r_{j}:K\rightarrow K$ so that
the map $1-r_{0}$ is invertible and so that $H\left(z\right)$ is
defined to be $H_{\left[z\right]_{\Lambda}}\left(z\right)$, where
$\left[z\right]_{\Lambda}$ is the projection of $z$ mod $\Lambda$.
We show that the direct analogue of Tao's almost-boundedness result
holds provided, in addition to several minor conditions, that the
``dimension constraint'' $\left(p-1\right)^{2}\rho_{H}^{d}<1$ is
satisfied, where $\rho_{H}=\prod_{j=0}^{p-1}\left\Vert r_{j}\right\Vert _{\textrm{Mink}}$
is the product of the operator norms of the $r_{j}$s with respect
to the Minkowski embedding of $K$. Unlike Tao's original argument,
which uses a $3$-adic Fourier decay result for the Syracuse Random
Variables, our approach employs analysis of an archimedean flavor,
and does not require any Fourier decay estimates for our generalized
Syracuse Random Variables. Our result includes Tao's original result
as a special case. 
\end{abstract}


\maketitle

\tableofcontents{}

\section{\label{section:intro}Introduction}

In 2019, Terence Tao published on \href{https://arxiv.org/abs/1909.03562}{arXiv}
a major result on the Collatz Conjecture, in which he established
that the trajectory of almost every positive integer under the \textbf{Collatz
Map}, $C:\mathbb{Z}\rightarrow\mathbb{Z}$, defined by: 
\begin{equation}
C\left(n\right)\overset{\textrm{def}}{=}\begin{cases}
\frac{n}{2} & \textrm{if }n\textrm{ is even}\\
3n+1 & \textrm{if }n\textrm{ is odd}
\end{cases}
\end{equation}
is \emph{almost} bounded. Specifically:

\vphantom{} 
\begin{thm}[Tao \cite{Tao Collatz}]
\label{thm:(Tao-2019)}Let $f:\mathbb{N}_{1}\rightarrow\mathbb{R}$
be any function with $f\left(n\right)\rightarrow+\infty$ as $n\rightarrow+\infty$.
Then, there is a set $S\subseteq\mathbb{N}_{1}$ of logarithmic density
$1$ so that, for any $n\in S$: 
\begin{equation}
f\left(n\right)>\min\left\{ n,C\left(n\right),C\left(C\left(n\right)\right),C\left(C\left(C\left(n\right)\right)\right),\ldots\right\} 
\end{equation}
\end{thm}
\vphantom{}

Tao produced this remarkable result by an ingenious application of
probabilistic reasoning to a family of random variables he called
the \textbf{Syracuse Random Variables (SRVs)}, which he denoted $\left\{ \mathbf{Syrac}\left(\mathbb{Z}/3^{n}\mathbb{Z}\right)\right\} _{n\geq0}$.
Here, for each $n$, $\mathbf{Syrac}\left(\mathbb{Z}/3^{n}\mathbb{Z}\right)$
takes values in the set $\left\{ 0,\ldots,3^{n}-1\right\} $. In \textbf{Remark
1.13 }at the bottom of page 10 of his paper, Tao pointed out that
one could view the SRVs as the projections mod $3^{n}$ of a single,
\emph{$3$-adic valued }random variable, $\mathbf{Syrac}\left(\mathbb{Z}_{3}\right)$,
however, he chose not to use that formalism for his work. In Tao's
own words, the most challenging part of his paper was in establishing
a decay estimate for the characteristic function of the SRVs:

\vphantom{} 
\begin{prop}[\textbf{Proposition 1.1.7 (Decay of characteristic function)}]
\label{prop:decay of characteristic function}Let $n\geq1$ and let
$\xi\in\mathbb{Z}/3^{n}\mathbb{Z}$ be co-prime to $3$. Then: 
\begin{equation}
\mathbb{E}\left(e^{-2\pi i\xi\mathbf{Syrac}\left(\mathbb{Z}/3^{n}\mathbb{Z}\right)/3^{n}}\right)\ll_{A}n^{-A}\label{eq:Tao characteristic function bound}
\end{equation}
for any fixed $A>0$.

\vphantom{} 
\end{prop}
\begin{rem}
Tao gives a heuristic argument in favor of the possibility that (\ref{eq:Tao characteristic function bound})
could be improved to $O\left(\exp\left(-cn\right)\right)$ for some
$c>0$. 
\end{rem}
In this paper, we show that the difficult Fourier decay estimate (\ref{eq:Tao characteristic function bound})
is not needed to establish the almost-boundedness result of \textbf{Theorem
\ref{thm:(Tao-2019)}}. This is done by transferring Tao's original
argument to the more general setting of Collatz-type maps defined
on the ring of integers of an algebraic number field. As per the work
in the first author's doctoral dissertation (\cite{my_disseration}),
we shall call these maps \textbf{Hydra maps}; these are defined in
detail in Section \ref{subsec:Hydra-Maps}, though the essence of
the idea is simple enough. First, however, let us recall the notion
of the \textbf{Minkowski embedding} of a number field, and the associated
\textbf{Minkowski norm}.
\begin{defn}
Given a number field $K$, recall that $K$ has a certain number $r_{K,\mathbb{R}}$
of embeddings into $\mathbb{R}$ and has a certain number of embeddings
into $\mathbb{C}$, and that one writes $r_{K,\mathbb{C}}$ to denote
the number of pairs of embedding which are equivalent up to an application
of complex conjugation. The degree $d$ of $K$ over $\mathbb{Q}$
is then equal to $r_{K,\mathbb{R}}+2r_{K,\mathbb{C}}$. Using all
these embeddings simultaneously defines the \textbf{Minkowski embedding}
$M:K\hookrightarrow\mathbb{R}^{r_{K,\mathbb{R}}}\oplus\mathbb{R}^{2r_{K,\mathbb{C}}}$.
The space $\mathbb{R}^{r_{K,\mathbb{R}}}\oplus\mathbb{R}^{2r_{K,\mathbb{C}}}$
is topologized by the $\ell^{\infty}$ norm, which induces the norm
$\left\Vert \cdot\right\Vert _{\textrm{Mink}}$ on $K$ defined by:
\begin{equation}
\left\Vert z\right\Vert _{\textrm{Mink}}\overset{\textrm{def}}{=}\max_{\iota:K\hookrightarrow\mathbb{C}}\left|\iota\left(z\right)\right|_{\infty},\textrm{ }\forall z\in K\label{eq:minkowski norm}
\end{equation}
where the maximum is taken over all field embeddings of $K$ into
$\mathbb{C}$, and where $\left|\cdot\right|_{\infty}$ is the absolute
value on $\mathbb{C}$. $\left\Vert \cdot\right\Vert _{\textrm{Mink}}$
is called the \textbf{Minkowski norm}. Note that if $K$ is a Galois
extension of $\mathbb{Q}$, we can write:
\begin{equation}
\left\Vert z\right\Vert _{\textrm{Mink}}=\max_{\sigma\in\textrm{Gal}\left(K/\mathbb{Q}\right)}\left|\iota\left(\sigma\left(z\right)\right)\right|_{\infty}
\end{equation}
for any embedding $\iota:K\hookrightarrow\mathbb{C}$.
\end{defn}
With this in mind, just as the Collatz map $C$ is constructed from
the two affine linear maps $n\mapsto n/2$ and $n\mapsto3n+1$, a
Hydra map $H:\mathcal{O}_{K}\rightarrow\mathcal{O}_{K}$ on the ring
of integers $\mathcal{O}_{K}$ of a number field $K$ consists of
$p\geq2$ (where $p$ is a positive integer, not necessarily prime)
affine linear maps $z\mapsto r_{j}z+c_{j}$ for $j\in\left\{ 0,\ldots,p-1\right\} $,
where the $r_{j}$s are invertible $\mathbb{Q}$-linear maps $K\rightarrow K$
and the $c_{j}$s are constants in $K$. 

In this set-up, the $r_{j}$s induce $\mathbb{Q}$-linear maps on
the image of $K$ under the Minkowski embedding. If we let $\rho_{H}$
be the product of the Minkowski norms of the $r_{j}$s under this
embedding, our argument shows that the condition:

\begin{equation}
\left(p-1\right)^{2}\rho_{H}^{d}<1\label{eq:dimension constraint first occurrence}
\end{equation}
is sufficient to establish an almost-boundedness result for $H$ in
the manner of \textbf{Theorem \ref{thm:(Tao-2019)}}, without the
need for difficult Fourier estimates. Here, $d$ is the degree of
the $K$ over $\mathbb{Q}$.

Our main result is the following. Note that the term ``proper, integral
$\Lambda$-Hydra map'' is defined in \textbf{Section \ref{subsec:Hydra-Maps}},
in \textbf{Definition \ref{def:definitions for Lambda hydras}} on
page \pageref{def:definitions for Lambda hydras}.
\begin{thm}[Almost-Boundedness]
\label{thm:almost boundedness}Let $K$ be a number field of degree
$d\geq1$ over $\mathbb{Q}$, let $\Lambda\subset\mathcal{O}_{K}$
be a non-trivial ideal of $\mathcal{O}_{K}$ with $\left[\mathcal{O}_{K}:\Lambda\right]=p$,
for some integer $p\geq2$, and let $H:\mathcal{O}_{K}\rightarrow\mathcal{O}_{K}$
be a proper, integral $\Lambda$-Hydra map.

Consider the hypotheses:

\vphantom{}\label{hyp:MH1}MH1: 
\begin{equation}
\rho_{H}<1\label{eq:geometric mean hypothesis}
\end{equation}
where: 
\begin{equation}
\rho_{H}\overset{\textrm{def}}{=}\prod_{j=0}^{p-1}\left\Vert r_{j}\right\Vert _{\textrm{Mink}}
\end{equation}
where: 
\begin{equation}
\left\Vert r_{j}\right\Vert _{\textrm{Mink}}\overset{\textrm{def}}{=}\sup_{z\in K}\frac{\left\Vert r_{j}z\right\Vert _{\textrm{Mink}}}{\left\Vert z\right\Vert _{\textrm{Mink}}}
\end{equation}
is the operator norm induced by the Minkowski norm.

\vphantom{}\label{hyp:MH2}MH2: There is a prime ideal $\mathfrak{P}$
in the factorization of $\Lambda$, with associated valuation $v_{\mathfrak{P}}$,
so that $\left\Vert r_{j}\right\Vert _{\mathfrak{P}}>1$ for all $j\in\left\{ 0,\ldots,p-1\right\} $,
where: 
\begin{equation}
\left\Vert r_{j}\right\Vert _{\mathfrak{P}}\overset{\textrm{def}}{=}\sup_{z\in K}\frac{\left|r_{j}z\right|_{\mathfrak{P}}}{\left|z\right|_{\mathfrak{P}}}
\end{equation}
where $\left|\cdot\right|_{\mathfrak{P}}$ is any non-archimedean
absolute value induced by the valuation $v_{\mathfrak{P}}$.

\vphantom{}\label{hyp:MH3}MH3: 
\begin{equation}
\left(p-1\right)^{2}\rho_{H}^{d}<1\label{eq:MH3}
\end{equation}

\vphantom{}\label{hyp:MH4}MH4: For each $j$, there exist scalars
$\mu_{j}\in K^{\times}$ with $v_{\Lambda}\left(\mu_{j}\right)\leq-1$
and an $\Lambda$-adic isometry $u_{j}:K\rightarrow K$ (i.e., $\left|u_{j}\left(z\right)\right|_{\Lambda}=\left|z\right|_{\Lambda}$,
for all $z\in K$) so that $r_{j}$ factors as $\mu_{j}u_{j}$ in
$\textrm{End}_{\mathbb{Q}}^{\times}\left(K\right)$, where $\left|\cdot\right|_{\Lambda}$
is a non-archimedean absolute value on $K$ induced by $\Lambda$.

If these four hypotheses are satisfied, then, for any function $f:\mathcal{O}_{K}\rightarrow\mathbb{R}$
with $f\left(z\right)\rightarrow+\infty$ as $\left|z\right|_{\infty}\rightarrow+\infty$,
there is a set $S\subseteq\mathcal{O}_{K}$ of logarithmic density
$1$ so that, for any $z\in S$: 
\begin{equation}
f\left(z\right)>\min\left\{ \left|z\right|_{\infty},\left|H\left(z\right)\right|_{\infty},\left|H\left(H\left(z\right)\right)\right|_{\infty},\left|H\left(H\left(H\left(z\right)\right)\right)\right|_{\infty},\ldots\right\} \label{eq:almost-boundedness result}
\end{equation}
where $\left|\cdot\right|_{\infty}$ is any archimedean absolute value
on $K$.

\vphantom{} 
\end{thm}
\begin{rem}
When $\Lambda$ is a prime ideal, we can and must take $\mathfrak{P}=\Lambda$,
and thus write $\left\Vert \cdot\right\Vert _{\Lambda}$ and $v_{\Lambda}$
to denote the absolute value and valuation associated to $\Lambda$,
respectively. In order to keep things simple, in an abuse of notation,
we shall write $\left\Vert \cdot\right\Vert _{\mathfrak{P}}$ and
$v_{\mathfrak{P}}$ as $\left\Vert \cdot\right\Vert _{\Lambda}$ and
$v_{\Lambda}$ even when $\Lambda$ is not prime. 
\end{rem}
\vphantom{} 
\begin{rem}
We call (\ref{eq:MH3}) the \textbf{Contraction/Dimension Constraint}
and thus characterize \textbf{Theorem \ref{thm:almost boundedness}}
as \textbf{contraction/dimension-driven almost boundedness}.

\vphantom{}
\end{rem}
Finally, a word on the structure of this paper. \textbf{Section \ref{sec:Background}
}gives an introduction to Hydra maps and the numen formalism. Though
not strictly necessary for the proof of \textbf{Theorem \ref{thm:almost boundedness}},
we include the material because of its major import to future work;
all probabilistic constructions in this paper can be reduced to finite-depth
symbolic truncations; no subtle issues regarding measurability, ergodicity,
or the like arise. An in-depth presentation of the Hydra and numen
material can be found in the reference document \cite{hydra_numen_formalism}
the first author has \href{https://arxiv.org/pdf/2601.17030}{posted to arXiv},
based on earlier work in \cite{my_disseration,my_first_blog_paper}.
\textbf{Section \ref{sec:Outline-of-the} }gives an outline of the
argument used to prove \textbf{Theorem \ref{thm:almost boundedness}};
the proof is given in \textbf{Sections \ref{sec:technical_machinery}}
through\textbf{ \ref{sec:Proof-of-Lemma-1}}.

\section{\label{sec:Background}Background}

\subsection{Background Material From Algebraic Number Theory and Lattice Theory}

The material in this section will be used primarily in \textbf{Section
\ref{sec:technical_machinery}} to establish asymptotics for certain
weighted lattice sums. Readers familiar with the subject matter can
skip this section.
\begin{defn}
Fix an integer $d\geq1$, let $K$ be a field, and let $V$ be a $d$-dimensional
vector space over $K$. Let $R$ be a ring contained within $K$,
and let $B=\left\{ \mathbf{v}_{1},\ldots,\mathbf{v}_{d}\right\} $
be a $K$-basis for $V$. Then, the \textbf{$R$-lattice in $V$ generated
by} $B$\textbf{ }is: 
\begin{equation}
\textrm{span}_{R}\left(B\right)=\left\{ \sum_{j=1}^{d}a_{j}\mathbf{v}_{j}:a_{1},\ldots,a_{d}\in R\right\} 
\end{equation}
More generally, a \textbf{lattice in $V$ }is an abelian group of
the form $\textrm{span}_{R}\left(B\right)$ for some $R$ and $B$.

Given $B$, let $\mathbf{B}$ denote the $d\times d$ matrix whose
colums are $\mathbf{v}_{1},\ldots,\mathbf{v}_{d}$. Then, the \textbf{Gram
matrix of $B$}, denoted $\textrm{Gram}\left(B\right)$ is defined
as the matrix product $\mathbf{B}^{T}\mathbf{B}$. Given a lattice
$\Lambda=\textrm{span}_{R}\left(B\right)$, if $V$ is a real vector
space, the quantity $\left|\det\left(\textrm{Gram}\left(B\right)\right)\right|$
is denoted $\textrm{vol}\left(\Lambda\right)$, and is called the
\textbf{volume }of $\Lambda$. By the standard geometric properties
of the determinant, this is the volume of the $d$-dimensional parallelepiped
generated by the elements of $B$.

If $V$ is an inner product space with inner product $\left\langle \cdot,\cdot\right\rangle $
the \textbf{dual lattice $\Lambda^{*}$ }of the lattice $\Lambda=\textrm{span}_{R}\left(B\right)$
is defined by:
\begin{equation}
\Lambda^{*}=\left\{ \mathbf{v}\in V:\left\langle \mathbf{x},\mathbf{v}\right\rangle =0,\textrm{ }\forall\mathbf{x}\in\Lambda\right\} 
\end{equation}
\end{defn}
\vphantom{}

Next, we recall some facts about the Minkowski norm.
\begin{fact}
The Minkowski norm is submultiplicative:
\begin{equation}
\left\Vert zw\right\Vert _{\textrm{Mink}}\leq\left\Vert z\right\Vert _{\textrm{Mink}}\left\Vert w\right\Vert _{\textrm{Mink}},\textrm{ }\forall z,w\in K
\end{equation}
\end{fact}
Embedded in $\mathbb{R}^{d}$ in this manner, $K$ satisfies several
useful properties.
\begin{fact}
For the Minkowski embedding $M\left(K\right)$ of $K$:

\vphantom{}I. The standard dot product on $M\left(K\right)$ is equal
to the field trace $\textrm{Tr}_{K/\mathbb{Q}}$. That is:

\begin{equation}
M\left(z\right)\cdot M\left(w\right)=\textrm{Tr}_{K/\mathbb{Q}}\left(zw\right),\textrm{ }\forall z,w\in K
\end{equation}

\vphantom{}II. The image of $\mathcal{O}_{K}$ under $M$ is a $d$-dimensional
lattice. Moreover, letting $\mathbf{B}$ be a $d\times d$ matrix
whose columns form a basis for $M\left(\mathcal{O}_{K}\right)$, one
has:
\begin{equation}
\det\left(\mathbf{B}^{T}\mathbf{B}\right)=\Delta_{K}
\end{equation}
where $\Delta_{K}$ is the \textbf{discriminant of $K$}. Furthermore,
$M\left(\mathcal{O}_{K}\right)$'s volume is $\sqrt{\left|\Delta_{K}\right|}$.

\vphantom{}III. For any non-zero ideal $J\subseteq\mathcal{O}_{K}$,
$M\left(J\right)$ is a $d$-dimensional lattice. Moreover, letting
$\mathbf{B}$ be a $d\times d$ matrix whose columns form a basis
for $M\left(J\right)$, one has: $\textrm{vol}\left(M\left(J\right)\right)=\sqrt{\left|\Delta_{K}\right|}\textrm{N}\left(J\right)$,
where $\textrm{N}\left(J\right)=\left[\mathcal{O}_{K}:J\right]$ is
the \textbf{ideal norm }of $J$. Furthermore, letting $M\left(J\right)^{*}$
denote the dual lattice of $M\left(J\right)$, we have that: 
\begin{equation}
\textrm{vol}\left(M\left(J\right)^{*}\right)=\frac{1}{\textrm{vol}\left(M\left(J\right)\right)}
\end{equation}
\end{fact}
\vphantom{}

Let us also recall:
\begin{defn}
The \textbf{different }of a number field $K$, denoted $\mathfrak{D}_{K}$,
is the ideal of $\mathcal{O}_{K}$ defined by:
\begin{equation}
\mathfrak{D}_{K}\overset{\textrm{def}}{=}\sqrt{\left|\Delta_{K}\right|}\mathcal{O}_{K}
\end{equation}
\end{defn}

\subsection{\label{subsec:Hydra-Maps}Hydra Maps}

Let $K$ be a number field and $\mathcal{O}_{K}$ the ring of $K$-integers.
It is a standard result of algebraic number theory that one can choose
a finite set $\left\{ b_{1},\ldots,b_{N}\right\} $ in $\mathcal{O}_{K}$
so that every element of $\mathcal{O}_{K}$ can be uniquely written
as: 
\begin{equation}
m_{1}b_{1}+\ldots+m_{N}b_{N}
\end{equation}
for some $m_{1},\ldots,m_{N}\in\mathbb{Z}$. That is, as an abelian
group under addition, $\mathcal{O}_{K}$ is a \textbf{lattice}. In
a similar way, every non-zero ideal of $\mathcal{O}_{K}$ can be realized
as lattice, as well.

We use the term \textbf{affine lattice }to refer to a coset of a lattice
in $\mathcal{O}_{K}$; thus, every affine lattice in $\mathcal{O}_{K}$
is of the form: 
\begin{equation}
c+\Lambda\overset{\textrm{def}}{=}\left\{ c+z:z\in\Lambda\right\} 
\end{equation}
for some ideal $\Lambda\subseteq\mathcal{O}_{K}$. We call $\Lambda$
the \textbf{underlying ideal }of $c+\Lambda$. Note that the underlying
ideal is necessarily unique, and is given by $\Lambda=\left\{ z-w:z,w\in c+\Lambda\right\} $.

In this context, letting $\Lambda$ be a lattice in $\mathcal{O}_{K}$
with $p$ distinct cosets, a $p$-Hydra map $H:\mathcal{O}_{K}\rightarrow\mathcal{O}_{K}$
is given by a collection of $p$ affine linear maps, called \textbf{branches},
which we associate to each of the cosets of $\Lambda$. In particular,
given $z\in\mathcal{O}_{K}$, $H\left(z\right)$ is obtained by applying
to $z$ the branch specified by the coset of $\Lambda$ that $z$
happens to lie in.

One difficulty that arises in generalizing maps of this shape to $\mathcal{O}_{K}$
is that while every map of the form: 
\begin{equation}
z\in K\mapsto az+b\in K
\end{equation}
for fixed constants $a,b\in K$ is affine linear, not every affine
linear map on $K$ will be of this form, as there exist linear maps
$K\rightarrow K$ which are not given by multiplication against some
fixed scalar. To deal with this, we introduce the notion of an \textbf{affine
lattice morphism}.

\vphantom{} 
\begin{defn}
We write $\textrm{Aff}_{\mathbb{Q}}K$ to denote the \textbf{group
of all} \textbf{invertible affine $\mathbb{Q}$-linear maps} $K\rightarrow K$;
these are maps of the form: 
\begin{equation}
z\mapsto rz+c\label{eq:affine map}
\end{equation}
where $r\in\textrm{End}_{\mathbb{Q}}^{\times}\left(K\right)$ and
$c\in K$; $\textrm{Aff}_{\mathbb{Q}}K$ is a non-abelian group under
map composition. We write $\left(r,c\right)$ as a shorthand for a
map of the form (\ref{eq:affine map}). 
\end{defn}
\vphantom{} 
\begin{rem}
It is important to distinguish between $\mathbb{Q}$-linear maps on
$K$ and $K$-linear maps on $K$. For the case $K=\mathbb{Q}\left(\sqrt{2}\right)$,
letting $a+b\sqrt{2}\in K$ be arbitrary (where $a,b\in\mathbb{Q}$),
consider the maps $m_{1},m_{2}:K\rightarrow K$ given by: 
\begin{align}
m_{1}\left(a+b\sqrt{2}\right)\overset{\textrm{def}}{=} & \sqrt{2}\left(a+b\sqrt{2}\right)\\
m_{2}\left(a+b\sqrt{2}\right)\overset{\textrm{def}}{=} & a+b
\end{align}
Both $m_{1}$ and $m_{2}$ are $\mathbb{Q}$-linear, but only $m_{1}$
is $K$-linear. Indeed: 
\begin{align}
m_{2}\left(\sqrt{2}\left(a+b\sqrt{2}\right)\right) & =m_{2}\left(2b+a\sqrt{2}\right)=a+2b\\
\sqrt{2}m_{2}\left(a+b\sqrt{2}\right) & =\sqrt{2}\left(a+b\right)\neq a+2b
\end{align}
\end{rem}
\vphantom{}

Now we can give the definition of a Hydra map.

\vphantom{} 
\begin{defn}
A pre-Hydra consists of the following data:

I. A number field $K$.

\vphantom{}

II. A proper, non-zero ideal $\Lambda\subset\mathcal{O}_{K}$.

\vphantom{}

III. For each element $j\in\mathcal{O}_{K}/\Lambda$, an affine linear
map $\left(r_{j},c_{j}\right)\in\textrm{Aff}_{\mathbb{Q}}K$. We often
denote $\left(r_{j},c_{j}\right)$ as $H_{j}$, and call $H_{j}$
the \textbf{$j$th branch} of the pre-Hydra. 
\end{defn}
\begin{rem}
These definitions can likely be modified to apply to general global
fields, rather than simply number fields, but we will not pursue this
in the present paper.
\end{rem}
\begin{defn}
\label{def:A--Hydra-(map)}A \textbf{$\Lambda$-Hydra (map)} on $K$
is a map $H:\mathcal{O}_{K}\rightarrow\mathcal{O}_{K}$ defined by:
\begin{equation}
H\left(z\right)\overset{\textrm{def}}{=}\sum_{j\in\mathcal{O}_{K}/\Lambda}\left[z\overset{\Lambda}{\equiv}j\right]\left(r_{j}z+c_{j}\right)\label{eq:Hydra}
\end{equation}
Here, $\left[z\overset{\Lambda}{\equiv}j\right]$ is the function
which is $1$ if $z$ is congruent to $j$ mod $\Lambda$, and is
$0$ otherwise, and $\left(K,\Lambda,\left\{ \left(r_{j},c_{j}\right)\right\} _{j\in\mathcal{O}_{K}/\Lambda}\right)$
is a pre-Hydra so that, for all $j\in\mathcal{O}_{K}/\Lambda$, we
have $r_{j}z+c_{j}\in\mathcal{O}_{K}$ for all $z\in\mathcal{O}_{K}$.

We call $\Lambda$ the \textbf{underlying ideal }of $H$, and write
$H_{j}$ to denote the affine linear map $\left(r_{j},c_{j}\right)$,
which we call the \textbf{$j$th branch of $H$}.

\vphantom{} 
\end{defn}
\begin{rem}
As defined, $H\left(z\right)$ applies $r_{j}z+c_{j}$, where $j$
is the representative of the coset of $\Lambda$ to which $z$ belongs.
(\ref{eq:Hydra}) can be written in a single line as: 
\begin{equation}
H\left(z\right)=r_{\left[z\right]_{\Lambda}}z+c_{\left[z\right]_{\Lambda}}
\end{equation}
where $\left[z\right]_{\Lambda}$ denotes the image of $z$ under
the canonical projection $\mathcal{O}_{K}\rightarrow\mathcal{O}_{K}/\Lambda$. 
\end{rem}
\vphantom{} 
\begin{rem}
As $\Lambda$ always has at least two cosets in $\mathcal{O}_{K}$,
we say a map $H:\mathcal{O}_{K}\rightarrow\mathcal{O}_{K}$ is a \textbf{$p$-Hydra
map} whenever $H$ is a $\Lambda$-Hydra map for some $\Lambda\subset\mathcal{O}_{K}$
so that $\left|\mathcal{O}_{K}/\Lambda\right|=p$.

\vphantom{} 
\end{rem}
Some notable properties of $\Lambda$-Hydras are as follows:

\vphantom{} 
\begin{defn}
\label{def:definitions for Lambda hydras}Let $H$ be a $\Lambda$-Hydra
on $K$. We say $H$ is:

I. \textbf{Integral}, if for all $z\in\mathcal{O}_{K}$ and all $j\in\mathcal{O}_{K}/\Lambda$,
$H_{j}\left(z\right)\in\mathcal{O}_{K}$ \emph{if and only if} $z\in j+\Lambda$.

\vphantom{}

II. \textbf{Proper}, if the map $1-r_{0}\in\textrm{End}_{\mathbb{Q}}\left(K\right)$
is invertible.

\vphantom{}

III. \textbf{Centered}, if $H_{0}\left(0\right)=0$. 
\end{defn}
\vphantom{} 
\begin{rem}
From the perspective of dynamical systems, the integrality property
of a Hydra is, by far, its most important property, and plays a critical
role in the proof of the \textbf{Correspondence Principle}. See \cite{hydra_numen_formalism}
for details.

As an example, the classic Collatz map $C:\mathbb{Z}\rightarrow\mathbb{Z}$
: 
\begin{equation}
C\left(n\right)=\begin{cases}
\frac{n}{2} & \textrm{if }n\overset{2}{\equiv}0\\
3n+1 & \textrm{if }n\overset{2}{\equiv}1
\end{cases}
\end{equation}
is a $2$-Hydra on $\mathbb{Z}$ which is non-integral, as the ``odd''
branch $H_{1}\left(x\right)=3x+1$ sends both even and odd integers
to integers. However, this is not a significant obstruction; the classic
``shortening'' of $C$ given by the aptly-named \textbf{Shortened
Collatz Map} $T_{3}:\mathbb{Z}\rightarrow\mathbb{Z}$: 
\begin{equation}
T_{3}\left(n\right)\overset{\textrm{def}}{=}\begin{cases}
\frac{n}{2} & \textrm{if }n\textrm{ is even}\\
\frac{3n+1}{2} & \textrm{if }n\textrm{ is odd}
\end{cases}
\end{equation}
\emph{is }integral, as its ``even branch'' $H_{0}$ is integer-valued
only for even integers, while its ``odd'' branch $H_{1}$ is integer-valued
only for odd integers. 
\end{rem}
\vphantom{} 
\begin{rem}
Letting $a\in\mathcal{O}_{K}$ be undetermined, one can often conjugate
an uncentered Hydra $H$ to a centered Hydra $H_{\textrm{center}}$
by writing: 
\begin{equation}
H_{\textrm{center}}\left(z\right)=H\left(z+a\right)-a
\end{equation}
for a judicious choice of $a\in\mathcal{O}_{K}$.

\vphantom{} 
\end{rem}

\subsection{\label{subsec:The-Numen-of}The Numen of a Hydra Map \& the $p$-Adic
Integers}

Throughout this section, let $H:\mathcal{O}_{K}\rightarrow\mathcal{O}_{K}$
be a $\Lambda$-Hydra map, not necessarily proper, integral, or centered.
Following \cite{my_disseration,my_first_blog_paper}, we construct
$X_{H}$, the numen of $H$, like so. 
\begin{defn}
Let $\Sigma_{\Lambda}^{*}$ denote the set of \textbf{strings} (finite
sequences) of elements of $\mathcal{O}_{K}/\Lambda$ including the
empty string, represented by the empty set $\varnothing$. given any
string $\mathbf{j}=\left(j_{1},\ldots,j_{N}\right)\in\Sigma_{\Lambda}^{*}$,
we call $N$ the \textbf{length }of $\mathbf{j}$, and denote it by
$\left|\mathbf{j}\right|$. The empty string is defined as the unique
string of length $0$. We write $\wedge:\Sigma_{\Lambda}^{*}\times\Sigma_{\Lambda}^{*}\rightarrow\Sigma_{\Lambda}^{*}$
to denote the concatenation operator: 
\begin{equation}
\mathbf{i}\wedge\mathbf{j}\overset{\textrm{def}}{=}\left(i_{1},\ldots,i_{\left|\mathbf{i}\right|},j_{1},\ldots,j_{\left|\mathbf{j}\right|}\right)
\end{equation}
We write $\Sigma_{\Lambda}^{\infty}$ to denote the set of strings
of finite or infinite length. Note that we can extend $\wedge$ to
a map $\wedge:\Sigma_{\Lambda}^{*}\times\Sigma_{\Lambda}^{\infty}\rightarrow\Sigma_{\Lambda}^{\infty}$. 
\end{defn}
\begin{defn}
Given any string $\mathbf{j}\in\Sigma_{\Lambda}^{*}$, we define the
\textbf{composition sequence }$H_{\mathbf{j}}\in\textrm{Aff}K$ by:
\begin{equation}
H_{\mathbf{j}}\left(z\right)\overset{\textrm{def}}{=}\left(H_{j_{1}}\circ\cdots\circ H_{j_{\left|\mathbf{j}\right|}}\right)\left(z\right),\textrm{ }\forall z\in K\label{eq:stringmap}
\end{equation}
When $\mathbf{j}$ is the empty string, we define $H_{\varnothing}$
to be the identity map on $K$.

As an element of $\textrm{Aff}K$, the map $H_{\mathbf{j}}$ can be
written in the form $\left(r,c\right)$ for some $r\in\textrm{End}_{\mathbb{Q}}^{\times}\left(K\right)$
and $c\in K$ depending on $\mathbf{j}$. We denote these by $M_{H}\left(\mathbf{j}\right)$
and $X_{H}\left(\mathbf{j}\right)$, respectively, so that: 
\begin{equation}
H_{\mathbf{j}}\left(z\right)=M_{H}\left(\mathbf{j}\right)z+X_{H}\left(\mathbf{j}\right),\textrm{ }\forall z\in K
\end{equation}
In this way, we can realize $M_{H}$ and $X_{H}$ as maps $\Sigma_{\Lambda}^{*}\rightarrow\textrm{End}_{\mathbb{Q}}^{\times}\left(K\right)$
and $\Sigma_{\Lambda}^{*}\rightarrow K$, respectively. $X_{H}$ is
called the \textbf{numen of $H$}, and is defined by: 
\begin{equation}
X_{H}\left(\mathbf{j}\right)\overset{\textrm{def}}{=}H_{\mathbf{j}}\left(0\right),\textrm{ }\forall\mathbf{j}\in\Sigma_{\Lambda}^{*}
\end{equation}
\end{defn}
The next two identities are trivial, and follow immediately from the
definitions: 
\begin{prop}
For any $\mathbf{i},\mathbf{j}\in\Sigma_{\Lambda}^{*}$: 
\begin{equation}
M_{H}\left(\mathbf{i}\wedge\mathbf{j}\right)=M_{H}\left(\mathbf{i}\right)M_{H}\left(\mathbf{j}\right)
\end{equation}
\begin{equation}
X_{H}\left(\mathbf{i}\wedge\mathbf{j}\right)=H_{\mathbf{i}}\left(X_{H}\left(\mathbf{j}\right)\right)
\end{equation}
\end{prop}
\begin{defn}
Letting $p$ denote $\left|\mathcal{O}_{K}/\Lambda\right|$, we fix
\emph{once and for all }choice of an enumeration $\Lambda_{0},\ldots,\Lambda_{p-1}$
of the cosets of $\Lambda$ in $\mathcal{O}_{K}$, with $\Lambda_{0}$
denoting $\Lambda$ itself and $\Lambda_{1},\ldots,\Lambda_{p-1}$
denoting its non-trivial cosets. We then let $\beta:\mathcal{O}_{K}/\Lambda\rightarrow\left\{ 0,\ldots,p-1\right\} $
be the map which sends $i\in\mathcal{O}_{K}/\Lambda$ to the unique
integer $\beta\left(i\right)\in\left\{ 0,\ldots,p-1\right\} $ so
that $i\in\Lambda_{\beta\left(i\right)}$. (Thus, if $i\in\Lambda_{0}$,
then $\beta\left(i\right)=0$; if $i\in\Lambda_{1}$, then $\beta\left(i\right)=1$,
etc.).

Because of this, we shall write $\Sigma_{p}^{*}$ instead of $\Sigma_{\Lambda}^{*}$,
with $\Sigma_{p}^{*}$ denoting the set of all finite strings of the
numbers $\left\{ 0,\ldots,p-1\right\} $, including the empty string.
$\beta$ then induces a bijection between $\Sigma_{\Lambda}^{*}$
and $\Sigma_{p}^{*}$.

Letting $\mathbb{N}_{0}=\left\{ 0,1,2,3,\ldots\right\} $ denote the
\textbf{set of non-negative integers}, we then write $\textrm{DigSum}_{p}:\Sigma_{p}^{*}\rightarrow\mathbb{N}_{0}$
to denote the function: 
\begin{equation}
\textrm{DigSum}_{p}\left(\mathbf{j}\right)=\sum_{n=1}^{\left|\mathbf{j}\right|}\beta\left(j_{n}\right)p^{n-1}\label{eq:DigSum}
\end{equation}
where $\left|\mathbf{j}\right|$ is the \textbf{length }of $\mathbf{j}$
(the number of entries of $\mathbf{j}$). 
\end{defn}
\begin{rem}
Like with writing $\Sigma_{p}^{*}$ instead of $\Sigma_{\Lambda}^{*}$,
going forward, in an very important\textbf{ }abuse of notation, we
will index indexing $r_{j}$ and $c_{j}$ using $j\in\left\{ 0,\ldots,p-1\right\} $
by way of $\beta$, instead of using $j\in\mathcal{O}_{K}/\Lambda$.
Thus, for example, by $r_{1}$, we actually mean $r_{\beta^{-1}\left(1\right)}$,
where $\beta^{-1}\left(1\right)$ is the unique $j\in\mathcal{O}_{K}/\Lambda$
that $\beta$ sends to $1$. 
\end{rem}
\begin{example}
If $\mathcal{O}_{K}=\mathbb{Z}$ and $\Lambda=3\mathbb{Z}$, then
upon identifying $\mathcal{O}_{K}/\Lambda=\mathbb{Z}/3\mathbb{Z}$
with $\left\{ 0,1,2\right\} $, we have that: 
\begin{equation}
\textrm{DigSum}_{3}\left(\mathbf{j}\right)=\sum_{n=1}^{\left|\mathbf{j}\right|}j_{n}3^{n-1}
\end{equation}
Thus, for example: 
\begin{equation}
\textrm{DigSum}_{3}\left(\left(1,2,0,1\right)\right)=1+2\cdot3+0\cdot3^{2}+1\cdot3^{3}
\end{equation}
\end{example}
Although $\textrm{DigSum}_{p}:\Sigma_{p}^{*}\rightarrow\mathbb{N}_{0}$
is surjective, note that it is \emph{not }injective. Indeed, given
any $\mathbf{i}\in\Sigma_{p}^{*}$ if we append finitely many $0$s
to the right of $\mathbf{i}$, we obtain a string $\mathbf{j}\neq\mathbf{i}$
for which $\textrm{DigSum}_{p}\left(\mathbf{i}\right)=\textrm{DigSum}_{p}\left(\mathbf{j}\right)$.
To that end: 
\begin{defn}
Let $\sim$ be the equivalence relation on $\Sigma_{p}^{*}$ defined
by $\mathbf{i}\sim\mathbf{j}$ if and only if $\textrm{DigSum}_{p}\left(\mathbf{i}\right)=\textrm{DigSum}_{p}\left(\mathbf{j}\right)$.
We then write $\Sigma_{p}^{*}/\sim$ to denote the set of equivalence
classes in $\Sigma_{p}^{*}$ under $\sim$. 
\end{defn}
This gives the easy result: 
\begin{prop}
\label{prop:equivalence-relation}The map $\Sigma_{p}^{*}/\sim\rightarrow\mathbb{N}_{0}$
induced by $\textrm{DigSum}_{p}$ is a bijection. 
\end{prop}
With this, we can realize $X_{H}$ as a function $\mathbb{N}_{0}\rightarrow K$. 
\begin{prop}
If $H$ is centered, then $X_{H}\left(\mathbf{i}\right)=X_{H}\left(\mathbf{j}\right)$
for all $\mathbf{i}$ and $\mathbf{j}$ belonging to the same equivalence
class in $\Sigma_{p}^{*}/\sim$. Hence, $X_{H}:\Sigma_{p}^{*}/\sim\rightarrow K$
is well-defined. We can then view $X_{H}$ as a function $\mathbb{N}_{0}\rightarrow K$
by using $\textrm{DigSum}_{p}$ to identify $\Sigma_{p}^{*}/\sim$
with $\mathbb{N}_{0}$. 
\end{prop}
\begin{rem}
If $H$ is \emph{not} centered, one can still formally define $X_{H}$
without reference to strings using the functional equations given
in \textbf{Proposition \ref{prop:functional equations}}, below, provided
they admit a solution. 
\end{rem}
Using the definition of $X_{H}\left(\mathbf{j}\right)$ as $H_{\mathbf{j}}\left(0\right)$,
one can easily give the following functional equation characterization
of $X_{H}:\mathbb{N}_{0}\rightarrow K$. 
\begin{prop}[Siegel (2022); see \cite{my_disseration}]
\label{prop:functional equations}The set of functions $X_{H}:\mathbb{N}_{0}\rightarrow K$
satisfying the functional equations: 
\begin{equation}
X_{H}\left(pn+j\right)=r_{j}X_{H}\left(n\right)+c_{j},\textrm{ }\forall n\in\mathbb{N}_{0},\textrm{ }\forall j\in\left\{ 0,\ldots,p-1\right\} \label{eq:X_H functional equation on N_0}
\end{equation}
(where $r_{j}X_{H}\left(n\right)$ is the image of $X_{H}\left(n\right)\in K$
under $r_{j}\in\textrm{End}_{\mathbb{Q}}^{\times}\left(K\right)$)
is in a bijective correspondence with the set of $z\in K$ for which
$\left(1-r_{0}\right)z=c_{0}$. In particular:

I. If $H$ is proper, then (\ref{eq:X_H functional equation on N_0})
has a unique solution, satisfying $X_{H}\left(0\right)=\left(1-r_{0}\right)^{-1}c_{0}$.

II. If $H$ is non-proper, then (\ref{eq:X_H functional equation on N_0})
has a solution $X_{H}$ if and only if there are $z\in K$ for which
$\left(1-r_{0}\right)z=c_{0}$. For each such $z$, one can set $X_{H}\left(0\right)=z$
and then obtain a distinct solution of (\ref{eq:X_H functional equation on N_0}).
In this way, we make think of $X_{H}\left(0\right)=z$ as an initial
condition subject to which (\ref{eq:X_H functional equation on N_0})
can be solved. 
\end{prop}
Just like with $\mathbb{N}_{0}$ and $\Sigma_{\Lambda}^{*}$, we can
identify $\mathbb{Z}_{p}$, the ring of $p$-adic integers, with a
set of strings. 
\begin{defn}
We write:

I. $\Sigma_{p}^{\infty}$, to denote the set of all strings of elements
of\emph{ }$\left\{ 0,\ldots,p-1\right\} $ of either finite or infinite
length. We also include the empty string as an element of $\Sigma_{p}^{\infty}$.

II. $\Sigma_{p}^{+\infty}$, to denote the set of all strings of elements
of\emph{ }$\left\{ 0,\ldots,p-1\right\} $ of infinite length. 
\end{defn}
\begin{rem}
We can and will extend $\textrm{DigSum}_{p}$ from a function $\Sigma_{p}^{*}\rightarrow\mathbb{N}_{0}$
to a function $\Sigma_{p}^{\infty}\rightarrow\mathbb{Z}_{p}$, where
$\Sigma_{p}^{\infty}$ is the set of finite- or infinite-length strings
in the symbols $\left\{ 0,\ldots,p-1\right\} $. Letting $\sim$ denote
the equivalence relation on $\Sigma_{p}^{\infty}$ with $\mathbf{i}\sim\mathbf{j}$
if and only if $\textrm{DigSum}_{p}\left(\mathbf{i}\right)=\textrm{DigSum}_{p}\left(\mathbf{j}\right)$,
we then have that $\textrm{DigSum}_{p}$ induces a bijection $\Sigma_{p}^{\infty}/\sim\rightarrow\mathbb{Z}_{p}$,
where $\Sigma_{p}^{\infty}/\sim$ is the set of equivalence classes
of $\Sigma_{p}^{\infty}$ under $\sim$. 
\end{rem}
In \cite{my_disseration}, it was shown that, given mild assumptions
on the $r_{j}$s, one can use the identification of $\Sigma_{p}^{\infty}/\sim$
with $\mathbb{Z}_{p}$ to interpolate $X_{H}:\mathbb{N}_{0}\rightarrow K$
to a function on $\mathbb{Z}_{p}$. The details for this are given
in \textbf{Lemma \ref{lem:extension and measurability of X_H}}, below.
The essential idea is that by making a judicious choice of absolute
values on $K$, one can realize $X_{H}$ as a function (or, in general,
as a distribution) out of $\mathbb{Z}_{p}$ whose regularity depends
on how close the $r_{j}$s are to behaving as contractions with respect
to $K$. The behavior of the $r_{j}$s is captured by their norm with
respect to a given absolute value of $K$. The result given below
is a slightly modified version of the versions originally stated in
\cite{my_first_blog_paper}. The reader should be aware that \textbf{Lemma
\ref{lem:extension and measurability of X_H} }treats $\mathbb{Z}_{p}$
as a measure space equipped with its \textbf{real-valued Haar probability
measure}, and treat $K_{\ell}$ (the completion of $K$ with respect
to the absolute value induced by a place $\ell$ of $K$) as a measure
space equipped with its real-valued Haar measure. We also use the
concept of a \textbf{locally constant function}. Readers unfamiliar
with either of these can find an expository treatment given of them
in \cite{hydra_numen_formalism}. 
\begin{lem}
\label{lem:extension and measurability of X_H}Suppose $H$ is proper.
For each non-trivial place $\ell$ of $K$, let:

\begin{equation}
\rho_{H,\ell}\overset{\textrm{def}}{=}\prod_{j=0}^{p-1}\left\Vert r_{j}\right\Vert _{\ell}
\end{equation}
be the product of the $\ell$-adic norms of the $r_{j}$s. (These
are the operator norms of the maps $K_{\ell}\rightarrow K_{\ell}$
induced by the $r_{j}$s).

Now, consider the extension $X_{H}:\mathbb{Z}_{p}\rightarrow K_{\ell}$
defined by the formula: 
\begin{align}
X_{H}\left(\mathfrak{z}\right) & \overset{K_{\ell}}{=}\lim_{n\rightarrow\infty}X_{H}\left(\left[\mathfrak{z}\right]_{p^{n}}\right)\label{eq:extension of Chi_H}
\end{align}
for all $\mathfrak{z}\in\mathbb{Z}_{p}$ for which $X_{H}\left(\mathfrak{z}\right)$
exists. Then:

I. If $\rho_{H,\ell}<1$, then (\ref{eq:extension of Chi_H}) defines
a measurable function $X_{H}:\mathbb{Z}_{p}\rightarrow K_{\ell}$,
with (\ref{eq:extension of Chi_H}) converging for almost every $\mathfrak{z}\in\mathbb{Z}_{p}$.
In particular, it converges at all $\mathfrak{z}\in\mathbb{Z}_{p}$
for which: 
\begin{equation}
\sum_{j=0}^{p-1}d_{H,\ell,j}\left(\mathfrak{z}\right)\ln\left\Vert r_{j}\right\Vert _{\ell}<0\label{eq:convergence condition}
\end{equation}
where: 
\begin{equation}
d_{H,\ell,j}\left(\mathfrak{z}\right)\overset{\textrm{def}}{=}\begin{cases}
\limsup_{m\rightarrow\infty}\frac{\#_{j}\left(\left[\mathfrak{z}\right]_{p^{m}}\right)}{m} & \textrm{if }\left\Vert r_{j}\right\Vert _{\ell}\leq1\\
\liminf_{m\rightarrow\infty}\frac{\#_{j}\left(\left[\mathfrak{z}\right]_{p^{m}}\right)}{m} & \textrm{if }\left\Vert r_{j}\right\Vert _{\ell}>1
\end{cases}\label{eq:upper density of j}
\end{equation}
is the \textbf{$\left(H,\ell\right)$-normalized upper density of
$j$ in the $p$-adic digits of $\mathfrak{z}$}.

II. If $\rho_{H,\ell}<1$ and $\max_{0\leq j<p}\left\Vert r_{j}\right\Vert _{\ell}<1$,
then (\ref{eq:extension of Chi_H}) converges uniformly with respect
to $\mathfrak{z}\in\mathbb{Z}_{p}$, making $X_{H}:\mathbb{Z}_{p}\rightarrow K_{\ell}$
into a continuous function.

Moreover, given either (I) or (II), we have that: 
\begin{equation}
X_{H}\left(p\mathfrak{z}+j\right)=r_{j}X_{H}\left(\mathfrak{z}\right)+c_{j}\label{eq:X_H functional equation}
\end{equation}
occurs for all $j\in\left\{ 0,\ldots,p-1\right\} $ and all $\mathfrak{z}\in\mathbb{Z}_{p}$
for which $X_{H}\left(\mathfrak{z}\right)$ converges in $K_{\ell}$,
and that the extensions of $X_{H}$ given in (I) and (II), respectively,
are the unique functions satisfying these functional equations. Furthermore,
if $\ell$ is non-archimedean and $\max_{0\leq j<p}\left\Vert r_{j}\right\Vert _{\ell}\leq1$,
then: 
\begin{equation}
\limsup_{N\rightarrow\infty}\sup_{\mathfrak{z}\in\mathbb{Z}_{p}}\left|X_{H}\left(\left[\mathfrak{z}\right]_{p^{N}}\right)\right|_{\ell}\leq\max_{0\leq j<p}\left|c_{j}\right|_{\ell}\label{eq:a.e.  ell-adic bound on X_H in NA case}
\end{equation}
In particular, $X_{H}:\mathbb{Z}_{p}\rightarrow K_{\ell}$ is $\ell$-adically
bounded by $\max_{0\leq j<p}\left|c_{j}\right|_{\ell}$ on its set
of convergence. Here, $\left|\cdot\right|_{\ell}$ is the absolute
value induced by the place $\ell$ on $K$, and hence by extension,
on $K_{\ell}$.
\end{lem}
Proof: See \cite{hydra_numen_formalism}.

Q.E.D.

\vphantom{} 
\begin{rem}
Like with \textbf{Proposition \ref{prop:functional equations}}, if
$H$ is not proper but there exist $z\in K$ for which $\left(1-r_{0}\right)z=c_{0}$,
one can recover the result of \textbf{Lemma \ref{lem:extension and measurability of X_H}
}subject to the condition that the uniqueness of $X_{H}$ is requires
not only the condition (\ref{eq:extension of Chi_H}) and the functional
equations (\ref{eq:X_H functional equation}), but also that $X_{H}$
satisfy the initial condition $X_{H}\left(0\right)=z$. 
\end{rem}
\vphantom{} 

\vphantom{} 
\begin{rem}
IFS studies traditionally use spaces of strings (usually called \textbf{words}
in that subject) where, following \cite{my_disseration}, we use the
$p$-adic integers. In passing from arbitrary strings $\mathbf{j}\in\Sigma_{p}^{\infty}$to
the $p$-adic integers using $\textrm{DigSum}_{p}$, we lose distinction
between those strings that end in infinitely many $0$s. Given any
iterated function system generated by maps $f_{1},\ldots,f_{N}$ on
a complete normed vector space $V$ we can avoid this lost information
by defining $p$ to be $N+1$, and then setting $H_{0}$ to be the
identity map of $V$ and setting $H_{j}=f_{j}$ for all $j\in\left\{ 1,\ldots,p-1\right\} $.
The numen: 
\begin{equation}
X_{H}\left(p\mathfrak{z}+j\right)=H_{j}\left(X_{H}\left(\mathfrak{z}\right)\right),\textrm{ }\forall j\in\left\{ 0,\ldots,p-1\right\} ,\textrm{ }\forall\mathfrak{z}\in\mathbb{Z}_{p}
\end{equation}
is then a surjective map from $\mathbb{Z}_{p}$ onto the fractal attractor
$F$ of the $f_{j}$s. If we set $q=N$ and defined 
\begin{equation}
Y_{H}\left(q\mathfrak{z}+j\right)=f_{j+1}\left(Y_{H}\left(\mathfrak{z}\right)\right),\textrm{ }\forall j\in\left\{ 0,\ldots,q-1\right\} ,\textrm{ }\forall\mathfrak{z}\in\mathbb{Z}_{q}
\end{equation}
the image of $Y_{H}$ would omit all the points of $F$ corresponding
to composition sequences: 
\begin{equation}
f_{j_{1}}\circ f_{j_{2}}\circ\cdots
\end{equation}
so that $j_{n}=1$ for all sufficiently large $n$. In this way, the
numen formalism can potentially be used to unify the study of IFS.
This, too, the first author hopes to explore in greater depth in future
papers. 
\end{rem}
\vphantom{}

\subsection{Probabilistic Formalism}

Throughout this paper, we will use the following construction to turn
a given finite set into a random variable:

\vphantom{} 
\begin{defn}
Let $\mathcal{U}$ be a finite, non-empty subset of $\mathcal{O}_{K}$.
We then define an $\mathcal{O}_{K}$-valued random variable, denoted
$\textrm{Log}\mathcal{U}$, by the probability distribution: 
\begin{equation}
\mathbb{P}\left(\textrm{Log}\mathcal{U}\in E\right)=\frac{\sum_{z\in E\cap\mathcal{U}}\left\Vert z\right\Vert _{\textrm{Mink}}^{-d}}{\sum_{z\in\mathcal{U}}\left\Vert z\right\Vert _{\textrm{Mink}}^{-d}}\label{eq:def of Log U}
\end{equation}
defined for all subsets $E\subseteq\mathcal{O}_{K}$.

With this, we then have the notion of the \textbf{logarithmic density}
of a set $A\subseteq\mathcal{O}_{K}$: 
\begin{equation}
d_{\textrm{log}}\left(A\right)\overset{\textrm{def}}{=}\lim_{x\rightarrow\infty}\mathbb{P}\left(\textrm{Log}\left(\left\{ z\in\mathcal{O}_{K}:0\leq\left\Vert z\right\Vert _{\textrm{Mink}}\leq x\right\} \right)\in A\right)\label{eq:def of logarithmic density}
\end{equation}
defined for all $A$ for which the limit exists. Given a statement
$S\left(z\right)$ depending on an element $z\in\mathcal{O}_{K}$,
we say $S\left(z\right)$ \textbf{holds for log-almost all $z\in\mathcal{O}_{K}$}
whenever the set: 
\begin{equation}
A=\left\{ z\in\mathcal{O}_{K}:S\left(z\right)=\textrm{True}\right\} 
\end{equation}
has logarithmic density $1$.

\vphantom{} 
\end{defn}
Next, we recall the definition of a geometric random variable. \vphantom{} 
\begin{defn}
A\textbf{ geometric random variable of mean $p$}; this is the positive-integer
valued random variable $\mathbf{Geom}\left(p\right)$ with the probability
mass function given by: 
\end{defn}
\begin{equation}
\mathbb{P}\left(\mathbf{Geom}\left(p\right)=m\right)=\frac{1}{p}\left(1-\frac{1}{p}\right)^{m-1},\textrm{ }\forall m\in\mathbb{N}_{1}\label{eq:def of geom}
\end{equation}
For an integer $n\geq1$, we then write $\mathbf{Geom}\left(p\right)^{n}$
to denote an $n$-tuple of independent, identically distributed random
variables of type $\mathbf{Geom}\left(p\right)$.

\vphantom{} 
\begin{defn}
Given a discrete space $R$, let $\mathbf{X}$ and $\mathbf{Y}$ be
RVs taking values in $R$. Then, the \textbf{total variation distance
}between $\mathbf{X}$ and $\mathbf{Y}$ is the metric defined by:
\begin{equation}
d_{\textrm{TV}}\left(\mathbf{X},\mathbf{Y}\right)\overset{\textrm{def}}{=}\sum_{r\in R}\left|\textrm{P}\left(\mathbf{X}=r\right)-\textrm{P}\left(\mathbf{Y}=r\right)\right|\label{eq:def of TV distance}
\end{equation}

\vphantom{} 
\end{defn}
Finally, we borrow the following Chernoff-type bound directly from
Tao's paper, where it occurs as \textbf{Lemma 2.2}:

\vphantom{} 
\begin{lem}[Chernoff Bound \cite{Tao Collatz}]
\label{lem:Chernoff Bound}

Let $d\in\mathbb{N}_{1}$, and suppose there is a constant $c_{0}>0$
and a $\mathbb{Z}^{d}$-valued random variable $\mathbf{v}$ so that:
\begin{equation}
\mathbb{P}\left(\Sigma\left(\mathbf{v}\right)\geq\lambda\right)\ll e^{-c_{0}\lambda},\textrm{ }\forall\lambda\geq0\label{eq:tail decay hypothesis for chernoff bound}
\end{equation}
Also, suppose that $\mathbf{v}$ is not almost-surely concentrated
on any coset of any proper subgroup of $\left(\mathbb{Z}^{d},+\right)$.
Let $\vec{\mu}$ denote the mean of $\mathbf{v}$; in what follows,
the positive constant $c$ and all constants of proportionality depend
on $d$, $c_{0}$, and the distribution of $\mathbf{v}$. Letting
$n\in\mathbb{N}_{0}$, let $\mathbf{v}_{1},\ldots,\mathbf{v}_{n}$
be i.i.d. copies of $\mathbf{v}$, and let $\mathbf{v}_{\left[1,n\right]}$
denote $\mathbf{v}_{1}+\cdots+\mathbf{v}_{n}$, and let $G_{n}:\mathbb{R}^{d}\rightarrow\mathbb{R}$
be defined by: 
\begin{equation}
G_{n}\left(\mathbf{x}\right)\overset{\textrm{def}}{=}e^{-\frac{1}{n}\left\Vert \mathbf{x}\right\Vert ^{2}}+e^{-\left\Vert \mathbf{x}\right\Vert }\label{eq:Gaussian-type weights}
\end{equation}
with $G_{0}\left(\mathbf{x}\right)\overset{\textrm{def}}{=}e^{-\left\Vert \mathbf{x}\right\Vert }$.
Then:

\vphantom{}

I. For all $\vec{L}\in\mathbb{Z}^{d}$: 
\begin{equation}
\mathbb{P}\left(\mathbf{v}_{\left[1,n\right]}=\vec{L}\right)\ll\frac{1}{\left(n+1\right)^{d/2}}G_{n}\left(c\left(\vec{L}-n\vec{\mu}\right)\right)\label{eq:Chernoff I}
\end{equation}

\vphantom{}

II. For any $\lambda\geq0$: 
\begin{equation}
\mathbb{P}\left(\Sigma\left(\mathbf{v}_{\left[1,n\right]}-n\vec{\mu}\right)\geq\lambda\right)\ll G_{n}\left(c\lambda\right)\label{eq:Chernoff II}
\end{equation}
\end{lem}

\section{\label{sec:Outline-of-the}Outline of the Main Argument}

Let everything be as given in \textbf{Theorem \ref{thm:almost boundedness}}.
As with Tao's work, the overarching idea is helpfully straightforward.
Formally, for our branches $H_{j}\left(z\right)=r_{j}z+c_{j}$, fixing
$n\geq1$ and given a string $\mathbf{j}=\left(j_{1},\ldots,j_{n}\right)$,
we have: 
\begin{equation}
H_{\mathbf{j}}\left(z\right)=r_{\mathbf{j}}z+X_{H}\left(\mathbf{j}\right)=\left(\prod_{m=1}^{n}r_{j_{m}}\right)z+\underbrace{\sum_{m=1}^{n}\left(\prod_{k=1}^{m-1}r_{j_{k}}\right)c_{j_{m}}}_{X_{H}\left(\mathbf{j}\right)}
\end{equation}
for all $z\in\mathcal{O}_{K}$, where: 
\begin{equation}
r_{\mathbf{j}}\overset{\textrm{def}}{=}\prod_{m=1}^{n}r_{j_{m}}\overset{\textrm{def}}{=}r_{j_{1}}r_{j_{2}}\cdots r_{j_{n}}
\end{equation}
Note that for any given $z$, there is a unique string $\mathbf{j}\left(z\right)=\left(j_{1}\left(z\right),\ldots,j_{n}\left(z\right)\right)$
of length $n$ so that $H_{\mathbf{j}\left(z\right)}\left(z\right)=H^{\circ n}\left(z\right)$,
where $H^{\circ n}\left(z\right)$ is the result of $n$ iterations
of $z$ under $H$. As such, choosing $\mathbf{j}\left(z\right)$
for $\mathbf{j}$ and taking Minkowski norms, we get: 
\begin{equation}
\left\Vert H_{\mathbf{j}\left(z\right)}\left(z\right)\right\Vert _{\textrm{Mink}}\leq\left(\prod_{k=0}^{p-1}\left\Vert r_{k}\right\Vert _{\textrm{Mink}}^{\#_{k}\left(\mathbf{j}\left(z\right)\right)}\right)\left\Vert z\right\Vert _{\textrm{Mink}}+\sum_{m=0}^{n-1}\left(\prod_{h=0}^{p-1}\left\Vert r_{h}\right\Vert _{\textrm{Mink}}^{\#_{h}\left(\mathbf{j}_{m}\left(z\right)\right)}\right)\left\Vert c_{j_{m}\left(z\right)}\right\Vert _{\textrm{Mink}}\label{eq:H_bold j upper estimate for section 3 intro}
\end{equation}
where $\mathbf{j}_{m}\left(z\right)=\left(j_{1}\left(z\right),\ldots,j_{m}\left(z\right)\right)$,
with $\mathbf{j}_{0}\left(z\right)=\varnothing$. As in Tao's case,
we will study the relation between $z$ and $\mathbf{j}\left(z\right)$
by replacing $z$ with a random variable $\mathcal{N}$ that takes
values in $\mathcal{O}_{K}$.

Now, suppose we can show that there is a choice of probability distribution
for $\mathcal{N}$ so that $\mathcal{N}$'s trajectory under $H$
(i.e., $\mathcal{N},H\left(\mathcal{N}\right),H^{\circ2}\left(\mathcal{N}\right)$)
is uniformly distributed mod $\Lambda$. This would then mean that
the entries of sequence $\mathbf{j}_{m}\left(\mathcal{N}\right)=\left(j_{1}\left(\mathcal{N}\right),\ldots,j_{m}\left(\mathcal{N}\right)\right)$
would be chosen from $\left\{ 0,\ldots,p-1\right\} $ according to
a uniform distribution. As such, by the \textbf{Law of Large Numbers},
we would expect that, as $m\rightarrow\infty$, the number of $k$s
that occur in $\mathbf{j}_{m}\left(\mathcal{N}\right)$ (which, recall,
is the to the number of times $H^{\circ h}\left(\mathcal{N}\right)$
lies in $\Lambda_{k}$ for $h\in\left\{ 0,\ldots,m-1\right\} $) would
satisfy: 
\begin{equation}
\#_{k}\left(\mathbf{j}_{m}\left(\mathcal{N}\right)\right)\approx\frac{m}{p},\textrm{ }\forall k\in\left\{ 0,\ldots,p-1\right\} \label{eq:geometric RV heuristic}
\end{equation}
for all sufficiently large $m$. If we fudge things a little and assume
that (\ref{eq:geometric RV heuristic}) holds true for all $m$, we
would get: 
\begin{align*}
\left\Vert H^{\circ n}\left(\mathcal{N}\right)\right\Vert _{\textrm{Mink}} & \leq\left(\prod_{k=0}^{p-1}\left\Vert r_{k}\right\Vert _{\textrm{Mink}}^{\#_{k}\left(\mathbf{j}\left(\mathcal{N}\right)\right)}\right)\left\Vert \mathcal{N}\right\Vert _{\textrm{Mink}}+\sum_{m=0}^{n-1}\left(\prod_{h=0}^{p-1}\left\Vert r_{h}\right\Vert _{\infty}^{\#_{h}\left(\mathbf{j}_{m}\left(\mathcal{N}\right)\right)}\right)\left\Vert c_{j_{m}\left(\mathcal{N}\right)}\right\Vert _{\textrm{Mink}}\\
 & \approx\left(\prod_{k=0}^{p-1}\left\Vert r_{k}\right\Vert _{\textrm{Mink}}\right)^{n/p}\left\Vert \mathcal{N}\right\Vert _{\textrm{Mink}}+\sum_{m=0}^{n-1}\left(\prod_{h=0}^{p-1}\left\Vert r_{h}\right\Vert _{\textrm{Mink}}\right)^{m/p}\left\Vert c_{j_{m}\left(\mathcal{N}\right)}\right\Vert _{\textrm{Mink}}\\
\left(\max_{0\leq j<p}\left\Vert c_{j}\right\Vert _{\textrm{Mink}}<\infty\right); & \ll\left(\prod_{k=0}^{p-1}\left\Vert r_{k}\right\Vert _{\textrm{Mink}}\right)^{n/p}\left\Vert \mathcal{N}\right\Vert _{\textrm{Mink}}+\underbrace{\sum_{m=0}^{n-1}\left(\prod_{h=0}^{p-1}\left\Vert r_{h}\right\Vert _{\textrm{Mink}}\right)^{m/p}}_{\textrm{sum this}}\\
\left(\prod_{k=0}^{p-1}\left\Vert r_{k}\right\Vert _{\textrm{Mink}}=\rho_{H}\right); & \ll\rho_{H}^{n/p}\left\Vert \mathcal{N}\right\Vert _{\textrm{Mink}}
\end{align*}
Thus, we are led to the prediction that: 
\begin{equation}
\rho_{H}<1\label{eq:geometric mean heuristic}
\end{equation}
implies: 
\begin{equation}
\left\Vert H^{\circ n}\left(\mathcal{N}\right)\right\Vert _{\textrm{Mink}}\ll\rho_{H}^{n/p}\left\Vert \mathcal{N}\right\Vert _{\textrm{Mink}}\label{eq:almost boundedness heuristic}
\end{equation}
In particular, we then have that $\left\Vert H^{\circ n}\left(\mathcal{N}\right)\right\Vert _{\textrm{Mink}}<\left\Vert \mathcal{N}\right\Vert _{\textrm{Mink}}$
should be bounded, provided that $n$ is much smaller than $\ln\left\Vert \mathcal{N}\right\Vert _{\textrm{Mink}}$.
The work before us is simply a matter of making the above heuristics
rigorous.

\subsection{Comparisons to Tuples of Geometric Random Variables}

The first and most important technical difference separating our work
from Tao's is the greater complexity of keeping track of $\mathcal{N}$'s
trajectory under $H$. For Tao's case, given any integers $N\in\mathbb{Z}$
and $n\geq0$, applying the Collatz map $C$ to $N$ a total of $n$
times will result in a length-$n$ sequence of applications of the
even branch ($x\mapsto x/2$) and odd branch ($x\mapsto3x+1$) of
the Collatz map. Because of this, Tao centered his analysis on the
following accelerated version of $C$, the so-called \textbf{Syracuse
map} $\textrm{Syr}:2\mathbb{Z}+1\rightarrow2\mathbb{Z}+1$ defined
by: 
\begin{equation}
\textrm{Syr}\left(N\right)\overset{\textrm{def}}{=}\frac{3N+1}{2^{v_{2}\left(3N+1\right)}},\textrm{ }\forall N\in2\mathbb{Z}+1
\end{equation}
where $v_{2}$ is the $2$-adic valuation. As constructed, when $N$
is odd, $\textrm{Syr}\left(N\right)$ will be the next odd integer
in $N$'s trajectory under $C$. As such, given $n\geq1$, the total
number of times Collatz's even branch would be applied in order to
send $N$ to $\textrm{Syr}^{\circ n}\left(N\right)$ can made arbitrarily
large by making an appropriate choice of $N$. Tao's \textbf{valuation
heuristic} (\textbf{Heuristic 1.8} from his paper) asserted that when
$N$ is very large and $n$ is much smaller than $\ln N$, the number
of applications of the even branch needed for $C$ to map $N$ to
$\textrm{Syr}^{\circ n}\left(N\right)$ behaves like a geometric random
variable. The analogue of this heuristic in our argument is precisely
the approximation (\ref{eq:geometric RV heuristic}) used above. To
see how, let us introduce some simple terminology.

\vphantom{} 
\begin{defn}
Given our $\Lambda$-Hydra $H$ and a $z\in\mathcal{O}_{K}$, let
$\mathbf{j}\left(z\right)=\left(j_{1}\left(z\right),j_{2}\left(z\right),\ldots\right)\in\Sigma_{p}^{\infty}$
be the $p$-ity vector of $z$, so that: 
\[
H^{\circ m}\left(z\right)=\left(H_{j_{1}}\circ H_{j_{2}}\circ\cdots\circ H_{j_{m}}\right)\left(z\right),\textrm{ }\forall m\geq1
\]
where $H^{\circ m}$ is the composition of $m$ copies of $H$. Given
$n\geq1$ and $j\in\left\{ 0,\ldots,p-1\right\} $, we say $z$'s
trajectory has a \textbf{length $n$ run of $j$s} whenever there
is a $k\geq1$ so that $j_{m}\left(z\right)=j$ for all $m\in\left\{ k,k+1,\ldots,k+n-1\right\} $.

More generally, by a \textbf{run}, we mean a consecutive sequence
of applications of a single branch of $H$. 
\end{defn}
\vphantom{}

Tao's definition of $\textrm{Syr}$ is best understood as being built
around runs. Given our random variable $\mathcal{N}$, applying $\textrm{Syr}$
to $\mathcal{N}$ means considering a parity vector for $N$ consisting
of a run of $0$s followed by a $1$. It is a basic fact of probability
that given a sequence $a_{1},a_{2},\ldots$ whose entries are chosen
randomly from a finite set $A$, the lengths of the runs in the sequence
will behave as a geometric random variable. Thus, we see that if $\mathcal{N}$'s
trajectory under the Collatz map is uniformly distributed mod $2$,
the lengths of the runs of $0$s that occur in applying $\textrm{Syr}$
should obey a geometric distribution, precisely as asserted by Tao's
valuation heuristic.

Thus, it is no surprise that Tao's rigorous formulation of his valuation
heuristic in \textbf{Proposition 1.9} of his paper involves comparing
$\mathcal{N}$ to a random variable which is uniformly distributed
modulo $2$; specifically, modulo a sufficiently large power of $2$.

\vphantom{} 
\begin{prop}[\textbf{\textit{Tao - Proposition 1.9}}\textit{ \cite{Tao Collatz}}]
\label{prop:Tao geometric comparison}Let $n\in\mathbb{N}_{0}$,
and let $\mathcal{N}$ be a random variable taking values in $2\mathbb{N}_{0}+1$
. Suppose there exist an absolute constant $c_{0}>0$ and some natural
number $n^{\prime}\geq\left(2+c_{0}\right)n$ such that $\mathcal{N}\textrm{ mod }2^{n^{\prime}}$
is approximately uniformly distributed in the odd residue classes
$\left(2\mathbb{Z}+1\right)/2^{n^{\prime}}\mathbb{Z}$ of $\mathbb{Z}/2^{n^{\prime}}\mathbb{Z}$,
in the sense that: 
\begin{equation}
d_{\textrm{TV}}\left(\mathcal{N}\textrm{ mod }2^{n^{\prime}},\mathbf{Unif}\left(\left(2\mathbb{Z}+1\right)/2^{n^{\prime}}\mathbb{Z}\right)\right)\ll2^{-n^{\prime}}\label{eq:Tao uniform condition}
\end{equation}
Then: 
\begin{equation}
d_{\textrm{TV}}\left(\vec{a}_{\textrm{Tao}}^{\left(n\right)}\left(\mathcal{N}\right),\mathbf{Geom}\left(2\right)^{n}\right)\ll2^{-c_{1}n}\label{eq:Tao local normality lemma estimate}
\end{equation}
for some absolute constant $c_{1}>0$ that depends on $c_{0}$. The
implied constants in the asymptotic notation are also allowed to depend
on $c_{0}$.

\vphantom{} 
\end{prop}
Here, $\vec{a}_{\textrm{Tao}}^{\left(n\right)}\left(\mathcal{N}\right)$,
Tao's version of the \textbf{$n$-path} of $\mathcal{N}$, is the
integer-tuple-valued random variable defined by: 
\begin{equation}
\vec{a}_{\textrm{Tao}}^{\left(n\right)}\left(\mathcal{N}\right)\overset{\textrm{def}}{=}\left(v_{2}\left(3\mathcal{N}+1\right),v_{2}\left(3\textrm{Syr}\left(\mathcal{N}\right)+1\right),\ldots,v_{2}\left(3\textrm{Syr}^{\circ n-1}\left(\mathcal{N}\right)+1\right)\right)\label{eq:Tao's n-path}
\end{equation}
while $\mathbf{Geom}\left(2\right)^{n}$ is an $n$-tuple of i.i.d.
geometric random variables of mean $2$. As defined, (\ref{eq:Tao's n-path})
counts the number of times we divide by $2$ in sending an odd integer
$\mathcal{N}$ to the next odd integer in \emph{$\mathcal{N}$'s}
Collatz trajectory. Our analogue of \textbf{Proposition \ref{prop:Tao geometric comparison}}
is \textbf{Lemma \ref{lem:Geometric comparison}}; the statement of
the Lemma is given below, after some motivating examples and the necessary
collection of definitions. The proof of our Lemma is effectively the
same as Tao's; the differences lie in how our Lemma is formulated.

Tao's set up constructs $\vec{a}^{\left(n\right)}\left(\mathcal{N}\right)$
to keep track of the lengths of the consecutive runs of applications
of the even branch of Collatz, and compares it to a tuple of geometric
random variables. We do much the same, except than instead of singling
out a single branch as Tao does---counting only runs of $H_{0}$---we
will keep track of the lengths of the consecutive runs of \emph{any}
given branch. As the definition is a bit of a mouthful, it helps to
motivate it with an example.

\vphantom{} 
\begin{example}
Let $z\in\mathcal{O}_{K}$, and let $p=3$. Using our abuse of notation
of writing $j\in\left\{ 0,\ldots,p-1\right\} $ to denote the element
of $\mathcal{O}_{K}/\Lambda$ corresponding to the coset $\Lambda_{j}$,
suppose the sequence $\left[z\right]_{\Lambda},\left[H\left(z\right)\right]_{\Lambda},\ldots,\left[H^{\circ11}\left(z\right)\right]_{\Lambda}$
is: 
\begin{equation}
0,0,1,2,2,1,1,1,0,1,1,2\label{eq:sample path vector}
\end{equation}
read from left-to-right. In order to keep track of the runs of consecutive
branches, let $\left\{ \mathbf{e}_{0},\mathbf{e}_{1},\mathbf{e}_{2}\right\} $
be the standard basis for $\mathbb{R}^{3}$ ($\mathbf{e}_{0}=\left(1,0,0\right)$,
$\mathbf{e}_{1}=\left(0,1,0\right)$, $\mathbf{e}_{2}=\left(0,0,1\right)$).
Then, we can annotate (\ref{eq:sample path vector}) by writing: 
\begin{equation}
2\mathbf{e}_{0},\mathbf{e}_{1},2\mathbf{e}_{2},3\mathbf{e}_{1},\mathbf{e}_{0},2\mathbf{e}_{1},\mathbf{e}_{2}\label{eq:sample run length vector}
\end{equation}
Read left-to-right, the subscripts of the vectors in this sequence
list the branch being used in a given run, while the coefficients
of the vectors give the number of branches in a given run. Thus, the
\textbf{first run} of (\ref{eq:sample path vector}) is two instances
of $0$; the \textbf{second run} is one instance of $1$, the \textbf{third
run }is two instances of $2$, and so on. We then call the sequence
(\ref{eq:sample run length vector}) the \textbf{seventh run vector
of} $z$, as it has seven terms.

\vphantom{} 
\end{example}
In this manner, we will keep track of the behavior of $\mathcal{N}$'s
trajectory under $H$ mod $\Lambda$ by using sequences of positive
integer multiples of vectors from the standard basis of $\mathbb{R}^{p}$. 

\vphantom{} 
\begin{defn}
\label{def:path and run vectors}Let $m,n\in\mathbb{N}_{1}$, and
let $z\in\mathcal{O}_{K}$. The \textbf{length $m$} \textbf{path
vector} \textbf{of $z$} is the sequence $\left[z\right]_{\Lambda},\left[H\left(z\right)\right]_{\Lambda},\ldots,\left[H^{\circ m-1}\left(z\right)\right]_{\Lambda}$
in $\mathcal{O}_{K}/\Lambda$. As per our use of the indices $j\in\left\{ 0,\ldots,p-1\right\} $
to label the elements of $\mathcal{O}_{K}/\Lambda$, we shall abuse
notation and treat the path vector of $z$ as a sequence of numbers
in $\left\{ 0,\ldots,p-1\right\} $. Notation-wise, we denote the
length $m$ path vector by: 
\begin{equation}
\mathbf{j}_{m}\left(z\right)\overset{\textrm{def}}{=}\left(\left[z\right]_{\Lambda},\left[H\left(z\right)\right]_{\Lambda},\ldots,\left[H^{\circ m-1}\left(z\right)\right]_{\Lambda}\right)\label{eq:path vector}
\end{equation}
Note the order of the indices in (\ref{eq:path vector}), which is
done to be compatible with our notation for composition sequences.
For brevity, we write: 
\begin{equation}
j_{k}\left(z\right)\overset{\textrm{def}}{=}\left[H^{\circ k}\left(z\right)\right]_{\Lambda}
\end{equation}

Next, note there exist unique positive integers $\nu_{1}\left(z\right),\ldots,\nu_{n}\left(z\right)$
and $\textrm{Stay}_{1}\left(z\right),\ldots,\textrm{Stay}_{n}\left(z\right)\in\left\{ 0,\ldots,p-1\right\} $
that: 
\begin{eqnarray*}
 & z\overset{\Lambda}{\equiv}H\left(z\right)\overset{\Lambda}{\equiv}\cdots\overset{\Lambda}{\equiv}H^{\circ\nu_{1}\left(z\right)-1}\left(z\right)\overset{\Lambda}{\equiv}\Lambda_{\textrm{Stay}_{1}\left(z\right)}\\
 & H^{\circ\nu_{1}\left(z\right)}\left(z\right)\overset{\Lambda}{\equiv}H\left(z\right)\overset{\Lambda}{\equiv}\cdots\overset{\Lambda}{\equiv}H^{\circ\nu_{2}\left(z\right)-1}\left(z\right)\overset{\Lambda}{\equiv}\Lambda_{\textrm{Stay}_{2}\left(z\right)}\\
 & \vdots\\
 & H^{\circ\nu_{n-1}\left(z\right)}\left(z\right)\overset{\Lambda}{\equiv}H\left(z\right)\overset{\Lambda}{\equiv}\cdots\overset{\Lambda}{\equiv}H^{\circ\nu_{n}\left(z\right)-1}\left(z\right)\overset{\Lambda}{\equiv}\Lambda_{\textrm{Stay}_{n}\left(z\right)}
\end{eqnarray*}
Here, $\nu_{k}$ is the lengths of the $k$th run in $z$'s trajectory,
while $\textrm{Stay}_{k}\left(z\right)$ is the index $j$ so that
the $k$th run of $z$ has $z$ staying in the coset $\Lambda_{j}$s
of $\Lambda$. Letting $B=\left\{ \mathbf{e}_{0},\ldots,\mathbf{e}_{p-1}\right\} $
be the standard basis $\mathbb{Z}^{p}$, we define the \textbf{$n$th
run vector of $z$} as: 
\begin{equation}
\vec{a}^{\left(n\right)}\left(z\right)\overset{\textrm{def}}{=}\left(\nu_{1}\left(z\right)\mathbf{e}_{\textrm{Stay}_{1}\left(z\right)},\nu_{n}\left(z\right)\mathbf{e}_{\textrm{Stay}_{n}\left(z\right)}\right)\label{eq:nth run vector def}
\end{equation}
Again, note the order of the indices. The $k$th component of the
$n$th run vector is a vector in $\mathbb{R}^{p}$ whose direction
encodes the branch applied to $z$ during the $k$th run, and whose
magnitude is the length of the $k$th run of $z$. Finally, we write
$\vec{\#}^{\left(n\right)}\left(z\right)$ to denote the element of
$\mathbb{N}_{1}^{p}$ defined by: 
\begin{equation}
\vec{\#}^{\left(n\right)}\left(z\right)\overset{\textrm{def}}{=}\sum_{m=1}^{n}\nu_{m}\left(z\right)\mathbf{e}_{\textrm{Stay}_{m}\left(z\right)}=\Sigma\left(\vec{a}^{\left(n\right)}\left(z\right)\right)\label{eq:j branch count}
\end{equation}
We can also write this as: 
\begin{equation}
\vec{\#}^{\left(n\right)}\left(z\right)=\left(\#_{0}^{\left(n\right)}\left(z\right),\ldots,\#_{p-1}^{\left(n\right)}\left(z\right)\right)
\end{equation}
where $\#_{j}^{\left(n\right)}\left(z\right)$ is the total number
of times $H_{j}$ was applied in sending $z$ to $H^{\circ\nu_{n}\left(z\right)}\left(z\right)$.

As defined, $\vec{a}^{\left(n\right)}\left(z\right)$ is a potentially
very long vector whose entires encode both the length and branch for
each of the first $n$ runs in $z$'s trajectory. In doing so, it
keeps track of the order in which the branches are applied. Meanwhile,
the $j$th entry of the $p$-tuple $\vec{\#}^{\left(n\right)}\left(z\right)$
is the total number of times the $j$th branch of $H$ was applied
during the first $n$ runs of $z$.

Finally, given a vector $\mathbf{v}=\left(v_{0},\ldots,v_{p-1}\right)\in\mathbb{N}_{1}^{p}$,
we write $\vec{\#}^{\left(n\right)}\left(z\right)\geq\mathbf{v}$
to mean that $\#_{j}^{\left(n\right)}\left(z\right)\geq v_{j}$ for
all $j\in\left\{ 0,\ldots,p-1\right\} $; the notation $\vec{\#}^{\left(n\right)}\left(z\right)\leq\mathbf{v}$
is defined similarly. 
\end{defn}
We also have the following $\Sigma$ notation that will be used throughout
the paper: 

\vphantom{}
\begin{defn}
Let $G$ be an abelian group, let $n\geq1$, and let $\mathbf{v}=\left(v_{1},\ldots,v_{n}\right)\in G^{n}$.
Then, we write $\Sigma\left(\mathbf{v}\right)$ to denote the sum
$v_{1}+\cdots+v_{n}$ in $G$.

A specific case of this that will occur in this paper is the following:
$\vec{\mathbf{v}}=\left(\mathbf{v}_{1},\ldots,\mathbf{v}_{n}\right)$,
where $\mathbf{v}_{1},\ldots,\mathbf{v}_{n}$ are elements of $\mathbb{R}^{p}$,
where $p\geq2$. Then, $\vec{\mathbf{v}}\in\left(\mathbb{R}^{p}\right)^{n}$
and $\Sigma\left(\vec{\mathbf{v}}\right)=\mathbf{v}_{1}+\cdots+\mathbf{v}_{n}\in\mathbb{R}^{p}$.
As such, we can apply $\Sigma$ again to get a real number $\Sigma\left(\Sigma\left(\vec{\mathbf{v}}\right)\right)=\Sigma\left(\mathbf{v}_{1}\right)+\cdots+\Sigma\left(\mathbf{v}_{n}\right)$.
We also write $\Sigma_{k}\left(\vec{\mathbf{v}}\right)$ to denote
$\mathbf{v}_{1}+\cdots+\mathbf{v}_{k}$.
\end{defn}
\begin{rem}
Morally, $\Sigma$ occurs when we are summing the entires of a tuple.
The subscript indicates how many entries we are summing. 
\end{rem}
Where Tao compares $\vec{a}_{\textrm{Tao}}^{\left(n\right)}\left(\mathcal{N}\right)$
to an $n$-tuple of geometric random variables of mean $2$, we will
need to construct a hybrid geometric-uniform random variable in order
to use to compare to $\vec{a}^{\left(n\right)}\left(\mathcal{N}\right)$.

\vphantom{} 
\begin{defn}
Let everything be as in \textbf{Definition \ref{def:path and run vectors}}. 

I. $B^{n}$, to denote the set of length $n$ sequences $\vec{\mathbf{v}}=\left(\mathbf{v}_{1},\ldots,\mathbf{v}_{n}\right)$
of vectors $\mathbf{v}_{1},\ldots,\mathbf{v}_{n}\in B$.

II. $B^{n*}$, to denote the subset of of tuples $\vec{\mathbf{v}}\in B^{n}$
so that $j_{k+1}\neq j_{k}$ for any $k\in\left\{ 1,\ldots,n-1\right\} $.

III. $\left(\mathbb{N}_{1}B\right)^{n}$, to denote the set of $n$-tuples
of the form $\left(m_{1}\mathbf{v}_{1},\ldots,m_{n}\mathbf{v}_{n}\right)$,
where $\vec{\mathbf{v}}=\left(\mathbf{v}_{1},\ldots,\mathbf{v}_{n}\right)\in B^{n}$
and where $\mathbf{m}=\left(m_{1},\ldots,m_{n}\right)\in\mathbb{N}_{1}^{n}$
is an $n$-tuple of positive integers. We write elements of $\left(\mathbb{N}_{1}B\right)^{n}$
as either $\vec{\mathbf{V}}=\left(\mathbf{V}_{1},\ldots,\mathbf{V}_{n}\right)$
or $\mathbf{m}\vec{\mathbf{v}}$, where $\mathbf{V}_{k}=m_{k}\mathbf{v}_{k}$
for all $k$. We will use the $\vec{\mathbf{V}}$ notation most often
in \textbf{Section \ref{sec:Proof-of-Lemma-1}}.

IV. $\left(\mathbb{N}_{1}B\right)^{n*}$, to denote the set of $n$-tuples
of the form $\left(m_{1}\mathbf{v}_{1},\ldots,m_{n}\mathbf{v}_{n}\right)$,
where $\vec{\mathbf{v}}=\left(\mathbf{v}_{1},\ldots,\mathbf{v}_{n}\right)\in B^{n*}$
and where $\mathbf{m}=\left(m_{1},\ldots,m_{n}\right)\in\mathbb{N}_{1}^{n}$
is an $n$-tuple of positive integers. 

V. $\mathbf{Path}^{n}\left(B\right)$, to denote the uniformly distributed
random variable taking values in $B^{n*}$.

VI. $\mathbf{Geom}^{n}\left(p\right)$, to denote the $\mathbb{N}_{1}^{n}$-valued
random variable whose entries are i.i.d. RVs of type $\mathbf{Geom}\left(p\right)$.

VII. We write $\mathbf{Run}\left(p\right)^{n}$ to denote the $\left(\mathbb{N}_{1}B\right)^{n*}$-valued
random variable given by $\mathbf{Geom}^{n}\left(p\right)\mathbf{Path}^{n}\left(B\right)$.
This outpus $n$-tuples of the form $\vec{\mathbf{V}}=\left(\mathbf{V}_{1},\ldots,\mathbf{V}_{n}\right)=\mathbf{m}\vec{\mathbf{v}}$,
where $\mathbf{m}=\left(m_{1},\ldots,m_{n}\right)$ is an outcome
of $\mathbf{Geom}^{n}\left(p\right)$ and where $\vec{\mathbf{v}}=\left(\mathbf{v}_{1},\ldots,\mathbf{v}_{n}\right)$
is an outcome of $\mathbf{Path}^{n}\left(B\right)$. Here, $\mathbf{Geom}^{n}\left(p\right)$
and $\mathbf{Path}^{n}\left(B\right)$ are defined to be independent
of one another.

\vphantom{} 
\end{defn}
\begin{rem}
\label{rem:PMF of run}Elementary computations give:

\begin{align}
\mathbb{P}\left(\mathbf{Run}\left(p\right)^{n}=\mathbf{m}\vec{\mathbf{v}}\right) & =\frac{p-1}{p}\frac{\left(1-\frac{1}{p}\right)^{\Sigma\left(\mathbf{m}\right)}}{\left(p-1\right)^{2n}}\prod_{k=2}^{n}\left[\mathbf{v}_{k}\neq\mathbf{v}_{k-1}\right]\label{eq:PMF of many Runs}
\end{align}
for all $n\geq1$, $\mathbf{m}=\left(m_{1},\ldots,m_{n}\right)\in\mathbb{N}_{1}^{n}$,
and $\vec{\mathbf{v}}=\left(\mathbf{v}_{1},\ldots,\mathbf{v}_{n}\right)\in B^{n}$.
The Iverson brackets in the product on the far right is there to remind
us that the probability will be $0$ whenever $\vec{\mathbf{v}}\notin B^{n*}$;
that is, whenever two consecutive entries of $\vec{\mathbf{v}}$ are
equal. 

\vphantom{} 
\end{rem}
\textbf{Lemma \ref{lem:Geometric comparison}}, stated below, generalizes
Tao's \textbf{Proposition \ref{prop:Tao geometric comparison}} by
demonstrating that $\mathcal{N}$ being approximately uniformly distributed
modulo a large power of $\Lambda$ implies that the $n$th run vector
of $\mathcal{N}$ behaves like $\mathbf{Run}\left(p\right)^{n}$ as
$n\rightarrow\infty$.

\vphantom{} 
\begin{lem}[Geometric Comparison Lemma]
\label{lem:Geometric comparison}Let $n\in\mathbb{N}_{0}$, and let
$\mathcal{N}$ be a random variable taking values in $\mathcal{O}_{K}$.
Then, there exists a constant $c_{0}>0$ (independent of $n$ and
$\mathcal{N}$) and an integer $n^{\prime}\geq\left(2+c_{0}\right)n$
so that the coset $\mathcal{N}+\Lambda^{n^{\prime}}$ is approximately
uniformly distributed in $\mathcal{O}_{K}/\Lambda^{n^{\prime}}$,
in the sense that: 
\begin{equation}
d_{\textrm{TV}}\left(\mathcal{N}_{y}\textrm{ mod }\Lambda^{n^{\prime}},\textrm{Unif}\left(\mathcal{O}_{K}/\Lambda^{n^{\prime}}\right)\right)\ll p^{-n^{\prime}}\label{eq:uniformity hypothesis}
\end{equation}
Then, there is a constant $c>0$ depending on $c_{0}$, so that: 
\begin{equation}
d_{\textrm{TV}}\left(\vec{a}^{\left(n\right)}\left(\mathcal{N}\right),\mathbf{Run}\left(p\right)^{n}\right)\ll p^{-cn}\label{eq:geometric comparison}
\end{equation}
The implied constants in the asymptotic notation are also allowed
to depend on $c_{0}$.

\vphantom{} 
\end{lem}
Like in Tao's paper, the specific $\mathcal{N}$ we will use in our
application of \textbf{Lemma \ref{lem:Geometric comparison} }will
be a logarithmically distributed random variable of the form: 
\begin{equation}
\mathcal{N}_{y}=\textrm{Log}U_{y}
\end{equation}
where $y$ is a large positive real number, and where: 
\begin{equation}
U_{y}\overset{\textrm{def}}{=}\left\{ z\in\mathcal{O}_{K}:y\leq\left\Vert z\right\Vert _{\textrm{Mink}}\leq y^{\alpha}\right\} \label{eq:definition of U_y}
\end{equation}
where $1<\alpha<2$ is a constant chosen for the occasion. 

\subsection{Stabilization Property of First Passage Locations}

With our definitions in hand, we can state our analogue of Tao's $\textrm{Syr}$,
and his notions of the first passage time and first passage location.

\vphantom{} 
\begin{defn}
Define $\textrm{Syr}_{H}:\mathcal{O}_{K}\rightarrow\mathcal{O}_{K}$
given by: 
\begin{equation}
\textrm{Syr}_{H}\left(z\right)\overset{\textrm{def}}{=}H^{\circ\Sigma\left(\Sigma\left(\vec{a}^{\left(1\right)}\left(z\right)\right)\right)}\left(z\right),\textrm{ }\forall z\in\mathcal{O}_{K}\label{eq:def of Syr H}
\end{equation}
We call $\textrm{Syr}_{H}$ the \textbf{Syracuse function of $H$}.
As defined $\textrm{Syr}_{H}\left(z\right)$ applies $H$ to $z$\textbf{
}precisely the number of times needed to advance $z$ to the end of
its first run of consecutive branches under $H$.

\vphantom{} 
\end{defn}
\begin{rem}
As defined, we have that: 
\begin{equation}
\textrm{Syr}_{H}^{\circ n}\left(z\right)=H^{\circ\Sigma\left(\Sigma\left(\vec{a}^{\left(n\right)}\left(z\right)\right)\right)}\left(z\right),\textrm{ }\forall n\geq1\label{eq:defining formula for nth iterate of Syr_H}
\end{equation}
That is, $\textrm{Syr}_{H}^{\circ n}\left(z\right)$ entails applying
$H$ to $z$ a number of times so as to advance $z$ to the end of
its $n$th run of consecutive branches under $H$. 
\end{rem}
\vphantom{} 
\begin{defn}
Given any real number $x>0$, the \textbf{first $x$-passage time}
is the function $T_{x}:\mathcal{O}_{K}\rightarrow\mathbb{N}_{0}\cup\left\{ +\infty\right\} $
defined by: 
\begin{equation}
T_{x}\left(z\right)\overset{\textrm{def}}{=}\inf\left\{ n\in\mathbb{N}_{0}:\left\Vert \textrm{Syr}_{H}^{\circ n}\left(z\right)\right\Vert _{\textrm{Mink}}\le x\right\} \label{eq:def of firstpasstime}
\end{equation}
The \textbf{first $x$-passage location}, $\textrm{Pass}_{x}:\mathcal{O}_{K}\rightarrow\mathcal{O}_{K}$,
is then defined by: 
\begin{equation}
\textrm{Pass}_{x}\left(z\right)\overset{\textrm{def}}{=}H^{\circ T_{x}\left(z\right)}\left(z\right)\label{eq:def of firstpassloc}
\end{equation}
Following Tao, in the case $T_{x}\left(z\right)=+\infty$, we simply
define $\textrm{Pass}_{x}\left(z\right)$ to be $1$.
\end{defn}
\vphantom{}

As in Tao's paper, the key proposition is: \vphantom{} 
\begin{lem}[Stabilization of first passage]
\label{lem:stabilization of first passage}Let $H$ be a $\Lambda$-Hydra
map satisfying MH1, MH2, \& MH3, as in \textbf{Theorem \ref{thm:almost boundedness}}.
For each sufficiently large $y>0$, let $\mathcal{N}_{y}$ be a random
variable with distribution: 
\begin{equation}
\textrm{Log}\left(\left\{ z\in\mathcal{O}_{K}:y\leq\left\Vert z\right\Vert _{\textrm{Mink}}\leq y^{\alpha}\right\} \right)
\end{equation}
With all of the above conditions satisfied, there is are absolute
constants $c>0$ and $\alpha\in\left(1,2\right)$ so that, for either
$y=x^{\alpha}$ or $y=x^{\alpha^{2}}$:

\vphantom{}

I. (\textbf{Passage Time Estimate}): For all sufficiently large $x>0$:
\begin{equation}
\mathbb{P}\left(T_{x}\left(\mathcal{N}_{y}\right)=+\infty\right)\ll x^{-c}\label{eq:passing time estimate}
\end{equation}

\vphantom{}

II. (\textbf{Passage Location Comparison}): For all sufficiently large
$x>0$: 
\begin{equation}
d_{\textrm{TV}}\left(\textrm{Pass}_{x}\left(\mathcal{N}_{x^{\alpha}}\right),\textrm{Pass}_{x}\left(\mathcal{N}_{x^{\alpha^{2}}}\right)\right)\ll x^{-\left(\alpha-1\right)c}\label{eq:passage location comparison}
\end{equation}

\vphantom{} 
\end{lem}
\begin{rem}
(\ref{eq:passing time estimate}) requires only MH1, MH2, \& MH4.
(\ref{eq:passage location comparison}) requires MH1, MH2, MH3, and
MH4.

\vphantom{} 
\end{rem}
In \textbf{Section \ref{subsec:First-Passage-Stabilization-impl}},
we prove that \textbf{Lemma \ref{lem:stabilization of first passage}}
implies our almost-boundedness result, \textbf{Theorem \ref{thm:almost boundedness}}
While the heart of Tao's paper is his proof of \textbf{Proposition
\ref{prop:decay of characteristic function}}, the most difficult
component of his argument overall, the proof of \textbf{Lemma \ref{lem:stabilization of first passage}}
will be the heart of our paper. Overall, our proof of \textbf{Lemma
\ref{lem:stabilization of first passage}} will follow the same path
as Tao's proof of the corresponding proposition from his paper, albeit
with three key differences. 
\begin{enumerate}
\item Our arguments will need to generalize Tao's approach to Hydra maps
with $3$ or more branches acting on multidimensional spaces. 
\item We will need to convert Tao's $3$-adic machinery into its natural
archimedean analogues. 
\item Most significantly, \emph{unlike }Tao's proof, \emph{our }proof of
(\ref{eq:passage location comparison}) DOES NOT require a Fourier
decay result like \textbf{Proposition \ref{prop:decay of characteristic function}}
(Tao's \textbf{Proposition 1.17}). Rather, the driving force behind
(\ref{eq:passage location comparison}) is the dimension constraint
(\ref{eq:MH3}). 
\end{enumerate}
(3) came as a great surprise to us; we only embarked on writing this
paper once the first author obtained a Fourier decay result for the
characteristic function of $\chi_{3}$, when $\chi_{3}$ is viewed
as a \emph{real}-valued random variable, and we naturally assumed
this decay result would play a key role in our proof, just as it did
in Tao's. And yet, it did not.

The proof of (\ref{eq:passing time estimate}), given in \textbf{Section
\ref{sec:-adic-distribution-of}}, is in a bijective correspondence
with Section 4 of Tao's paper. On the other hand, the parts of our
argument corresponding to the material in Section 5 of Tao's paper
will be ordered differently compared to Tao's original presentation.
First, we have set aside \textbf{Section \ref{sec:technical_machinery}}
to establish the necessary relations between our analogues of the
constants Tao calls $\alpha$, $m_{0}$, and $n_{0}$, as well as
to gather together various technical results, notation, and estimates
that will be used in the arguments going forward. This section also
gives the definition of $\mathbf{Syrac}_{H}\left(n\right)$ (also
written $\mathbf{S}_{H}\left(n\right)$), our analogue of Tao's Syracuse
random variables.

As stated above, \textbf{Section \ref{subsec:First-Passage-Stabilization-impl}}
proves our almost-boundedness result, assuming \textbf{Lemma \ref{lem:stabilization of first passage}}.
\textbf{Section \ref{sec:-adic-distribution-of} }is a 1-to-1 rewrite
of Section 4 of Tao's paper, where we prove \textbf{Lemma \ref{lem:Geometric comparison}}
after some initial legwork. Finally, \textbf{Section \ref{sec:Proof-of-Lemma-1}}
gives the proof of \textbf{Lemma \ref{lem:stabilization of first passage}}.
This is done in two subsections, one for each of the Lemma's two claims.

\section{\label{sec:technical_machinery}Technical Machinery}

We begin with some constants.

\vphantom{} 
\begin{defn}
Set: 
\begin{equation}
\rho_{H}\overset{\textrm{def}}{=}\prod_{j=0}^{p-1}\left\Vert r_{j}\right\Vert _{\textrm{Mink}}
\end{equation}
\begin{equation}
R_{H}\overset{\textrm{def}}{=}\max_{0\leq j<p}\left\Vert r_{j}\right\Vert _{\textrm{Mink}}
\end{equation}
\begin{equation}
\overline{\rho}_{H}\overset{\textrm{def}}{=}\prod_{j=0}^{p-1}\max\left\{ \left\Vert r_{j}\right\Vert _{\textrm{Mink}},\left\Vert r_{j}\right\Vert _{\textrm{Mink}}^{-1}\right\} 
\end{equation}
\end{defn}
In order to replicate Tao's scaling arguments, we need to take some
care in setting up our parameters.

\vphantom{}\begin{assumption}[Parameters] \label{assu:initial parameter assumptions}Let
$x$ be an arbitrarily large positive real number. Let $A_{0},A_{1}>1$
be constants to be determined. Set: 
\begin{equation}
n_{x}\overset{\textrm{def}}{=}\left\lfloor \frac{\ln x}{\ln A_{0}}\right\rfloor \label{eq:def of n_0}
\end{equation}
\begin{equation}
m_{x}\overset{\textrm{def}}{=}\left\lfloor \left(\alpha-1\right)\frac{\ln x}{\ln A_{1}}\right\rfloor \label{eq:def of m_0}
\end{equation}

Next, let $\epsilon_{0},\delta_{0}>0$ be small, and pick $C_{0}$
with: 
\[
p-\epsilon_{0}<C_{0}<p
\]
and choose $\alpha>1$ so that: 
\begin{equation}
\max\left\{ \alpha^{3}-C_{0}\frac{\ln\rho_{H}^{-1}}{\ln A_{0}},C_{0}\frac{p\ln R_{H}}{\ln A_{0}}\right\} <1\label{eq:C_0 A_0 inequality}
\end{equation}
\begin{equation}
\alpha>\frac{\ln\left(R_{H}^{p}\rho_{H}^{-1}\right)}{\ln A_{0}}\label{eq:alpha A_0 inequality}
\end{equation}
\end{assumption}

\vphantom{}Our next asymptotic (\ref{eq:lattice sum lower bound})
is a classical result of algebraic number theory and lattice theory.
\begin{lem}
\label{lem:lattice sum asymptotic}Let $f,g:\left(0,\infty\right)\rightarrow\mathbb{R}$
be functions with $f\left(r\right)<g\left(r\right)$ for all sufficiently
large $r$, with $f\left(r\right),g\left(r\right)\rightarrow\infty$
as $r\rightarrow\infty$, so that:
\[
\lim_{r\rightarrow\infty}\ln\left(\frac{g\left(r\right)}{f\left(r\right)}\right)=\infty
\]
Then, setting:
\begin{equation}
S_{r}=\left\{ z\in\mathcal{O}_{K}:f\left(r\right)\leq\left\Vert z\right\Vert _{\textrm{Mink}}\leq g\left(r\right)\right\} 
\end{equation}
and letting $m\geq d$ (where, recall, $d$ is the dimension of $K$
over $\mathbb{Q}$) we have:
\begin{equation}
\sum_{z\in S_{r}}\left\Vert z\right\Vert _{\textrm{Mink}}^{-m}\sim\begin{cases}
C_{K}d\ln\left(\frac{g\left(r\right)}{f\left(r\right)}\right) & \textrm{if }m=d\\
C_{K}d\frac{\left(f\left(r\right)\right)^{d-m}-\left(g\left(r\right)\right)^{d-m}}{m-d} & \textrm{if }m>d
\end{cases}\textrm{ as }r\rightarrow\infty\label{eq:lattice sum asymptotic}
\end{equation}
where $C_{K}$ is an absolute constant depending only on $K$.
\end{lem}
Proof Let $N_{K}\left(R\right)$ denote the set of all $z\in\mathcal{O}_{K}\backslash\left\{ 0\right\} $
for which $\left\Vert z\right\Vert _{\textrm{Mink}}\leq R$. Since
the Minkowski embedding identifies $\mathcal{O}_{K}$ with a full-rank
lattice in a real vector space of dimension $d$, the standard lattice-point
counting estimate gives:
\begin{equation}
N_{K}\left(R\right)=C_{K}R^{d}+O\left(R^{d-1}\right)\textrm{ as }R\rightarrow\infty
\end{equation}
where $C_{K}$ is a constant depending solely on $K$. 

Next, for each sufficiently large real number $r>0$, set $a=f\left(r\right)$
and $b=g\left(r\right)$, so that $0<a<b$ and that $a$ and $b$
tend to $\infty$ as $r\rightarrow\infty$. Using Abel's Summation
Formula, we have:
\[
\sum_{\begin{array}{c}
z\in\mathcal{O}_{K}\\
f\left(r\right)\leq\left\Vert z\right\Vert _{\textrm{Mink}}\leq g\left(r\right)
\end{array}}\left\Vert z\right\Vert _{\textrm{Mink}}^{-m}=\frac{N_{K}\left(b\right)}{b^{m}}-\frac{N_{K}\left(a^{-}\right)}{a^{m}}+m\int_{a}^{b}\frac{N_{K}\left(t\right)}{t^{m+1}}dt
\]
where $N_{K}\left(a^{-}\right)$ denotes $\lim_{t\uparrow a}N_{K}\left(t\right)$.
The lattice-point estimate from above gives us:
\begin{align*}
N_{K}\left(b\right) & =C_{K}b^{d}+O\left(b^{d-1}\right)\\
N_{K}\left(a^{-}\right) & =C_{K}a^{d}+O\left(a^{d-1}\right)
\end{align*}
and so:
\begin{align*}
\sum_{\begin{array}{c}
z\in\mathcal{O}_{K}\\
f\left(r\right)\leq\left\Vert z\right\Vert _{\textrm{Mink}}\leq g\left(r\right)
\end{array}}\left\Vert z\right\Vert _{\textrm{Mink}}^{-m} & =C_{K}\left(b^{d-m}-a^{d-m}\right)+C_{K}m\int_{a}^{b}t^{d-m-1}dt\\
 & +O_{K}\left(b^{d-m-1}\right)+O_{K}\left(a^{d-m-1}\right)+O_{K,m}\left(\int_{a}^{b}t^{d-m-2}dt\right)
\end{align*}
where the subscripts of the $O$s indicates the objects upon which
the implied constants of proportionality depend.

Since $m\geq d$, we have:
\[
d-m-1\leq-1
\]
and:
\[
d-m-2\leq-2
\]
Since $b\geq a$, we have:
\[
b^{d-m-1}\leq a^{d-m-1}
\]
Furthermore:
\[
\int_{a}^{b}t^{d-m-2}\leq\int_{a}^{\infty}t^{d-m-2}dt\ll_{m}a^{d-m-1}
\]
Hence:
\begin{equation}
\sum_{\begin{array}{c}
z\in\mathcal{O}_{K}\\
f\left(r\right)\leq\left\Vert z\right\Vert _{\textrm{Mink}}\leq g\left(r\right)
\end{array}}\left\Vert z\right\Vert _{\textrm{Mink}}^{-m}=C_{K}\left(b^{d-m}-a^{d-m}\right)+C_{K}m\int_{a}^{b}t^{d-m-1}dt+O_{K,m}\left(a^{d-m-1}\right)\label{eq:1}
\end{equation}

We now consider cases based on the value of $m$. If $m=d$, then:
\[
b^{d-m}-a^{d-m}=1-1=0
\]
while:
\[
m\int_{a}^{b}t^{d-m-1}dt=d\int_{a}^{b}\frac{dt}{t}=d\ln\left(\frac{b}{a}\right)
\]
and so:

\[
\sum_{\begin{array}{c}
z\in\mathcal{O}_{K}\\
f\left(r\right)\leq\left\Vert z\right\Vert _{\textrm{Mink}}\leq g\left(r\right)
\end{array}}\left\Vert z\right\Vert _{\textrm{Mink}}^{-m}=C_{K}d\ln\left(\frac{b}{a}\right)+O_{K}\left(\frac{1}{a}\right)
\]
By our hypotheses, $\ln\left(\frac{b}{a}\right)=\ln\left(\frac{g\left(r\right)}{f\left(r\right)}\right)$
and $1/a$ tend to $\infty$ and $0$, respectively, as $r\rightarrow\infty$,
leaving us with:
\[
\sum_{\begin{array}{c}
z\in\mathcal{O}_{K}\\
f\left(r\right)\leq\left\Vert z\right\Vert _{\textrm{Mink}}\leq g\left(r\right)
\end{array}}\left\Vert z\right\Vert _{\textrm{Mink}}^{-d}\sim C_{K}d\ln\left(\frac{g\left(r\right)}{f\left(r\right)}\right)\textrm{ as }r\rightarrow\infty,\textrm{ when }m=d
\]

Next, if $m>d$, we have:
\[
\int_{a}^{b}t^{d-m-1}dt=\frac{b^{d-m}-a^{d-m}}{d-m}=\frac{a^{d-m}-b^{d-m}}{m-d}
\]
Substituting this into (\ref{eq:1}), we have:
\begin{align*}
\sum_{\begin{array}{c}
z\in\mathcal{O}_{K}\\
f\left(r\right)\leq\left\Vert z\right\Vert _{\textrm{Mink}}\leq g\left(r\right)
\end{array}}\left\Vert z\right\Vert _{\textrm{Mink}}^{-m} & =C_{K}\left(b^{d-m}-a^{d-m}\right)+C_{K}m\frac{a^{d-m}-b^{d-m}}{m-d}+O_{K,m}\left(a^{d-m-1}\right)\\
\left(\frac{m}{m-d}=1+\frac{d}{m-d}\right); & =\frac{C_{K}d}{m-d}\left(a^{d-m}-b^{d-m}\right)+O_{K,m}\left(a^{d-m-1}\right)
\end{align*}
Since $m-d>0$, it follows that:
\[
O_{K,m}\left(a^{d-m-1}\right)=o\left(a^{d-m}-b^{d-m}\right)\rightarrow0\textrm{ as }r\rightarrow\infty
\]
and we are left with:
\[
\sum_{\begin{array}{c}
z\in\mathcal{O}_{K}\\
f\left(r\right)\leq\left\Vert z\right\Vert _{\textrm{Mink}}\leq g\left(r\right)
\end{array}}\left\Vert z\right\Vert _{\textrm{Mink}}^{-m}\sim\frac{C_{K}d}{m-d}\left(a^{d-m}-b^{d-m}\right)\textrm{ when }m>d
\]
Since $a=f\left(r\right)$ and $b=g\left(r\right)$, we get the desired
result.

Q.E.D.

\vphantom{}

Tao's use of the map $\textrm{Aff}_{\vec{a}}$ is specific to the
two-branch nature of the Collatz map. In order to generalize this,
we need to pay attention to the intuition behind it. For positive
integer values of $a$, the map: 
\begin{equation}
\textrm{Aff}_{a}\left(x\right)=\frac{3x+1}{2^{a}}
\end{equation}
controls the length of consecutive applications of Collatz's odd branch.

As such, to generalize it, we need to keep track of the consecutive
applications of any given branch of $H$. We do this with the following
notation.

\vphantom{}\begin{notation} Recall that we write $\left(\mathbb{N}_{1}B\right)^{n}$
to denote the set of $n$-tuples $\vec{\mathbf{V}}=\mathbf{m}\vec{\mathbf{v}}$
where $\mathbf{m}\in\mathbb{N}_{1}^{n}$, for each $k\in\left\{ 1,\ldots,n\right\} $,
there are integers $m_{k}\geq1$ and $\vec{\mathbf{v}}=\left(\mathbf{v}_{1},\ldots,\mathbf{v}_{n}\right)\in B^{n}$. 

Additionally, for any $k\in\left\{ 1,\ldots,n\right\} $, we write:
\begin{equation}
\Sigma_{k}\left(\mathbf{m}\vec{\mathbf{v}}\right)\overset{\textrm{def}}{=}\sum_{i=1}^{k}m_{i}\mathbf{v}_{i}
\end{equation}
Note that $\Sigma_{k}\left(\mathbf{m}\vec{\mathbf{v}}\right)$ will
be an element of $\mathbb{R}^{p}$. For each $j\in\left\{ 0,\ldots,p-1\right\} $,
the quantity $\Sigma_{k}\left(\mathbf{m}\vec{\mathbf{v}}\right)\cdot\mathbf{e}_{j}$
is the $j$th entry of $\Sigma_{k}\left(\mathbf{m}\vec{\mathbf{v}}\right)$,
and consists of the sum of $m_{i}$ for all $i\in\left\{ 1,\ldots,k\right\} $
for which $\mathbf{v}_{i}=\mathbf{e}_{j}$.

\vphantom{} \end{notation}
\begin{defn}
\label{def:v-arrow formalism}For any $n\geq1$, any $\vec{\mathbf{v}}=\left(\mathbf{v}_{1},\ldots,\mathbf{v}_{n}\right)\in B^{n}$
will be of the form $\mathbf{v}_{k}=\mathbf{e}_{j_{k}}$ for some
$j_{k}\in\left\{ 0,\ldots,p-1\right\} $. To that end, we write $H_{\mathbf{e}_{j}}$
to denote $H_{j}$, the $j$th branch of $H$. To that end, we write:
\begin{equation}
H_{\mathbf{m}\vec{\mathbf{v}}}\overset{\textrm{def}}{=}H_{\mathbf{v}_{n}}^{\circ m_{n}}\circ\cdots\circ H_{\mathbf{v}_{1}}^{\circ m_{1}}\label{eq:def of H_v-arrow}
\end{equation}
That is, $\mathbf{m}$'s entries dictate the length of the consecutive
runs of branches, while $\vec{\mathbf{v}}$'s entries dictate the
specific branches used in each run. We then define: 
\begin{equation}
X_{H}\left(\mathbf{m}\vec{\mathbf{v}}\right)\overset{\textrm{def}}{=}H_{\mathbf{m}\vec{\mathbf{v}}}\left(0\right),\textrm{ }\forall\mathbf{m}\vec{\mathbf{v}}\in\bigcup_{n=0}^{\infty}\left(\mathbb{N}_{1}B\right)^{n}\label{eq:def of X_v-arrow}
\end{equation}

\vphantom{} 
\end{defn}
\begin{rem}
Combining this notation with our earlier one, we have that: 
\begin{equation}
H^{\circ\Sigma\left(\vec{\#}^{\left(n\right)}\left(z\right)\right)}\left(z\right)=M_{H}\left(\vec{a}^{\left(n\right)}\left(z\right)\right)z+X_{H}\left(\vec{a}^{\left(n\right)}\left(z\right)\right)
\end{equation}
for all $z\in\mathcal{O}_{K}$ and all $n\in\mathbb{N}_{0}$.

\vphantom{} 
\end{rem}
Using this notation, we can define our analogue of Tao's Syracuse
Random Variables.

\vphantom{} 
\begin{defn}
Let $n\in\mathbb{N}_{1}$. Then, we define $H$'s $n$th \textbf{Syracuse
random variable (SRV)}, denoted $\mathbf{Syrac}_{H}\left(n\right)$
or $\mathbf{S}_{H}\left(n\right)$, for short, as the \emph{discrete}
$K$-valued random variable: 
\begin{equation}
\mathbf{Syrac}_{H}\left(n\right)\overset{\textrm{def}}{=}\mathbf{S}_{H}\left(n\right)\overset{\textrm{def}}{=}X_{H}\mid_{\mathbb{N}_{0}}\left(\mathbf{Run}\left(p\right)^{n}\right)\label{eq:our syracuse random variables}
\end{equation}
where, as indicated, we treat $X_{H}\mid_{\mathbb{N}_{0}}\left(\mathbf{Run}\left(p\right)^{n}\right)$
as the image of $\mathbf{Run}\left(p\right)^{n}$ under $X_{H}:\mathbb{N}_{0}\rightarrow K$,
so that: 
\[
\mathbb{P}\left(\mathbf{S}_{H}\left(n\right)\in E\right)=\mathbb{P}\left(X_{H}\left(\mathbf{Run}\left(p\right)^{n}\right)\in E\right)=\sum_{\begin{array}{c}
\mathbf{m}\vec{\mathbf{v}}\in\left(\mathbb{N}_{1}B\right)^{n}\\
X_{H}\left(\mathbf{m}\vec{\mathbf{v}}\right)\in E
\end{array}}\mathbb{P}\left(\mathbf{Run}\left(p\right)^{n}=\mathbf{m}\vec{\mathbf{v}}\right)
\]
for all measurable $E\subseteq K$, where we use our convention for
making sense of $X_{H}\left(\mathbf{m}\vec{\mathbf{v}}\right)$ to
make sense of $X_{H}\left(\mathbf{Run}\left(p\right)^{n}\right)$.

\vphantom{} 
\end{defn}
\begin{rem}
This is the natural generalization of Tao's $\mathbf{Syrac}\left(\mathbb{Z}/3^{n}\mathbb{Z}\right)$.
For $H=T_{3}$, $\mathbf{S}_{H}\left(n\right)$ is equal to $\mathbf{Syrac}\left(\mathbb{Z}/3^{n}\mathbb{Z}\right)$.

\vphantom{} 
\end{rem}
These next three notations will be used intensively in all that follows:

\vphantom{} 
\begin{defn}
For any $x>0$ and any $L,L^{\prime}\in\mathbb{N}_{0}$ with $L^{\prime}\leq L$,
we define: 
\begin{equation}
\mathcal{B}_{x}^{\left(L\right)}\overset{\textrm{def}}{=}\left\{ \mathbf{n}\in\mathbb{N}_{0}^{p}:\Sigma\left(\mathbf{n}\right)=L\textrm{ \& }\max_{0\leq j<p}\left|\frac{L}{p}-\mathbf{n}\cdot\mathbf{e}_{j}\right|\leq\ln^{0.6}x\right\} \label{eq:definition of B_x L}
\end{equation}
and write $\mathcal{A}^{\left(L^{\prime}\right)}$ to denote the set
of all $\mathbf{m}\vec{\mathbf{v}}=\left(m_{1}\mathbf{v}_{1},\ldots,m_{L^{\prime}}\mathbf{v}_{L^{\prime}}\right)\in\left(\mathbb{N}_{1}B\right)^{L^{\prime}}$
so that: 
\begin{equation}
\max_{0\leq j<p}\left|n-\Sigma_{n}\left(\mathbf{m}\vec{\mathbf{v}}\right)\cdot\mathbf{e}_{j}\right|\leq\frac{1}{p}\ln^{0.6}x,\textrm{ }\forall n\in\left\{ 1,\ldots,L^{\prime}\right\} \label{eq:approximate uniform behavior}
\end{equation}
These are the tuples that represent the behavior of ``generic''
trajectories of $H$. 

Lastly, we write: 
\begin{equation}
I_{y}\overset{\textrm{def}}{=}\left[\frac{\ln\left(y/x\right)}{\ln\rho_{H}^{-1}}+\ln^{0.8}x,\frac{\ln\left(y^{\alpha}/x\right)}{\ln\rho_{H}^{-1}}-\ln^{0.8}x\right]\label{eq:def of I_y}
\end{equation}
where, as per our convention, we either choose $y=x^{\alpha}$ or
$y=x^{\alpha^{2}}$. 

\vphantom{} 
\end{defn}
\begin{rem}
Note that, by the triangle inequality (\ref{eq:approximate uniform behavior})
implies: 
\begin{equation}
\ln^{0.6}x\geq\sum_{j=0}^{p-1}\left|u_{j}pn-\Sigma_{n}\left(\mathbf{m}\vec{\mathbf{v}}\right)\cdot\mathbf{e}_{j}\right|\geq\left|pn\sum_{j=0}^{p-1}d_{j}-\sum_{j=0}^{p-1}\Sigma_{n}\left(\mathbf{m}\vec{\mathbf{v}}\right)\cdot\mathbf{e}_{j}\right|=\left|pn-\Sigma\left(\Sigma_{n}\left(\vec{\mathbf{v}}\right)\right)\right|\label{eq:condition on sum of n-sum of v-arrow}
\end{equation}
for all $n\in\left\{ 1,\ldots,L\right\} $.

\vphantom{} 
\end{rem}
We will need the following inequality relating the condition $I_{y}\subset\left[m_{x},n_{x}\right]$
to the values of $\gamma$, $\alpha$, $A_{0}$, and $A_{1}$. \vphantom{} 
\begin{prop}
\label{prop:gamma inequality proposition}$I_{y}\subset\left[m_{x},n_{x}\right]$
whenever: 
\begin{align}
A_{0}^{\alpha^{3}-1} & <\rho_{H}^{-1}<A_{1}\label{eq:A_0 A_1 A_2 inequalities}\\
A_{0}^{\alpha-1} & <A_{1}\nonumber 
\end{align}
\end{prop}
Proof: We need: 
\[
\left(\alpha-1\right)\frac{\ln x}{\ln A_{1}}\asymp m_{x}<\frac{\ln\left(y/x\right)}{\ln\rho_{H}^{-1}}+\ln^{0.8}x
\]
\[
m_{x}\asymp\left(\alpha-1\right)\frac{\ln x}{\ln A_{1}}<\frac{\ln x}{\ln A_{0}}\asymp n_{x}
\]
\[
\frac{\ln\left(y^{\alpha}/x\right)}{\ln\rho_{H}^{-1}}-\ln^{0.8}x<n_{x}\asymp\frac{\ln x}{\ln A_{0}}
\]
as $x\rightarrow\infty$.

In the first inequality, the upper bound is minimized when $y=x^{\alpha}$,
while in the third inequality, the lower bound is maximized when $y=x^{\alpha^{2}}$.
Using this and dividing out by $\ln x$ in all three inequalities
gives: 
\[
\frac{\alpha-1}{\ln A_{1}}<\frac{\alpha-1}{\ln\rho_{H}^{-1}}+\ln^{-0.2}x
\]
\[
\frac{\alpha-1}{\ln A_{1}}<\frac{1}{\ln A_{0}}
\]
\[
\frac{\alpha^{3}-1}{\ln\rho_{H}^{-1}}-\ln^{-0.2}x<\frac{1}{\ln A_{0}}
\]
These three inequalities will all be true whenever: 
\begin{align*}
\rho_{H}^{-1} & <A_{1}\\
A_{0}^{\alpha-1} & <A_{1}\\
A_{0}^{\alpha^{3}-1} & <\rho_{H}^{-1}
\end{align*}

Q.E.D. 

\vphantom{}We conclude with the principal estimates we will use for
$X_{H}$ and $M_{H}$. We give the archimedean and non-archimedean
estimates separately, as the non-archimedean estimates come with an
attached probabilistic formalism that does not arise in the archimedean
case.

\vphantom{} 
\begin{prop}
\label{prop:elementary properties of X_H of v-arrow}We have the following
explicit formula for $X_{H}\left(\vec{\mathbf{v}}\right)$: 
\begin{equation}
X_{H}\left(\mathbf{m}\vec{\mathbf{v}}\right)=\sum_{k=1}^{n}\left(\prod_{i=0}^{p-1}r_{i}^{\Sigma_{k-1}\left(\mathbf{m}\vec{\mathbf{v}}\right)\cdot\mathbf{e}_{i}}\right)\prod_{j=0}^{p-1}\left(\sum_{h=0}^{m_{k}\left(\mathbf{v}_{k}\cdot\mathbf{e}_{j}\right)-1}r_{j}^{h}c_{j}\right),\textrm{ }\forall\mathbf{m}\vec{\mathbf{v}}\in\left(\mathbb{N}_{1}B\right)^{n},\textrm{ }\forall n\geq1\label{eq:explicit formula for X_H of v-arrow}
\end{equation}
as well as the estimate: 
\begin{equation}
\left\Vert X_{H}\left(\mathbf{m}\vec{\mathbf{v}}\right)\right\Vert _{\textrm{Mink}}<\frac{nR_{H}^{2\Sigma\left(\mathbf{m}\right)}}{R_{H}-1}=O\left(nR_{H}^{2\Sigma\left(\mathbf{m}\right)}\right)\textrm{ as }n\rightarrow\infty,\textrm{ }\forall\mathbf{m}\vec{\mathbf{v}}\in\left(\mathbb{N}_{1}B\right)^{n}\label{eq:estimate for X_H-arrow}
\end{equation}
where the constant of proportionality depends only on $H$.

Furthermore, given $L\in\mathbb{N}_{1}$ and $\mathbf{m}\vec{\mathbf{v}}\in\mathcal{A}^{\left(L\right)}$
we have: 
\begin{align}
\left\Vert X_{H}\left(\vec{\mathbf{v}}\right)\right\Vert _{\textrm{Mink}} & <\frac{1-\rho_{H}^{L}}{1-\rho_{H}}\frac{\overline{\rho}_{H}^{\frac{1}{p}\ln^{0.6}x}}{R_{H}-1}R_{H}^{p+2\ln^{0.6}x}\label{eq:selective estimate for X_H of v arrow}\\
 & \ll\left(R_{H}^{2}\overline{\rho}_{H}^{1/p}\right)^{\ln^{0.6}x}\nonumber 
\end{align}
and: 
\begin{equation}
\left\Vert M_{H}\left(\mathbf{m}\vec{\mathbf{v}}\right)\right\Vert _{\textrm{Mink}}\leq\overline{\rho}_{H}^{\frac{1}{p}\ln^{0.6}x}\rho_{H}^{L}\label{eq:selective estimate for M_H of v arrow}
\end{equation}
as $x\rightarrow\infty$. 
\end{prop}
\begin{rem}
Let's briefly unpack the notation in (\ref{eq:explicit formula for X_H of v-arrow})
before proceeding. Recall that $\mathbf{m}\vec{\mathbf{v}}$ is the
``generic'' version of the $n$th run vector of an algebraic integer
$z$ under $H$; the $k$th entry of $\mathbf{m}$ represents the
length of a $k$th run; the unit vector in the $k$th entry of $\vec{\mathbf{v}}$
indicates which branch was applied in the $k$th run. $\Sigma\left(\mathbf{m}\right)$
is then sum total of the lengths of the first $n$ runs; $\Sigma_{k-1}\left(\mathbf{m}\vec{\mathbf{v}}\right)$,
meanwhile, is a $p$-tuple whose $j$th entry represents the number
of times the $j$th branch was applied in the first $k$ entries of
$\vec{\mathbf{v}}$. Thus, $\Sigma_{k-1}\left(\mathbf{m}\vec{\mathbf{v}}\right)\cdot\mathbf{e}_{i}$
gives us the sum of the coefficients of the first $k-1$ entries of
$\mathbf{m}\vec{\mathbf{v}}$ whose $\vec{\mathbf{v}}$ components
were $\mathbf{e}_{i}$. $m_{k}\left(\mathbf{v}_{k}\cdot\mathbf{e}_{j}\right)$,
meanwhile, is the coefficient of $\mathbf{v}_{k}$ if $\mathbf{v}_{k}=\mathbf{e}_{j}$
and is $0$.
\end{rem}
Proof: The proof of (\ref{eq:explicit formula for X_H of v-arrow})
is by direct computation using the definitions (\ref{eq:def of X_v-arrow})
and (\ref{eq:def of H_v-arrow}). Taking absolute values of (\ref{eq:explicit formula for X_H of v-arrow})
yields: 
\begin{align*}
\left\Vert X_{H}\left(\mathbf{m}\vec{\mathbf{v}}\right)\right\Vert _{\textrm{Mink}} & \leq\left(\max_{0\leq j<p}\left\Vert c_{j}\right\Vert _{\textrm{Mink}}\right)\sum_{k=1}^{n}\left(\prod_{\ell=1}^{k-1}\left\Vert r_{j_{\ell}}\right\Vert _{\textrm{Mink}}^{m_{\ell}}\right)\sum_{h=0}^{m_{k}-1}\left\Vert r_{j_{k}}\right\Vert _{\textrm{Mink}}^{h}\\
\left(R_{H}=\max_{0\leq j<p}\left\Vert r_{j}\right\Vert _{\infty}\right); & \leq\left(\max_{0\leq j<p}\left\Vert c_{j}\right\Vert _{\textrm{Mink}}\right)\sum_{k=1}^{n}\left(\prod_{\ell=1}^{k-1}R_{H}^{m_{\ell}}\right)\sum_{h=0}^{m_{k}-1}R_{H}^{h}=\sum_{k=1}^{n}R_{H}^{\sum_{\ell=1}^{k-1}m_{\ell}}\frac{R_{H}^{m_{k}}-1}{R_{H}-1}\\
 & \leq\left(\max_{0\leq j<p}\left\Vert c_{j}\right\Vert _{\textrm{Mink}}\right)\sum_{k=1}^{n}R_{H}^{\Sigma\left(\mathbf{m}\right)}\frac{R_{H}^{\Sigma\left(\mathbf{m}\right)}-1}{R_{H}-1}\\
 & <n\frac{R_{H}^{2\Sigma\left(\mathbf{m}\right)}}{R_{H}-1}
\end{align*}

As for the more refined estimates, let $\mathbf{m}\vec{\mathbf{v}}\in\mathcal{A}^{\left(L\right)}$.
Then: 
\begin{equation}
\left\Vert M_{H}\left(\mathbf{m}\vec{\mathbf{v}}\right)\right\Vert _{\textrm{Mink}}\leq\prod_{j=0}^{p-1}\left\Vert r_{j}\right\Vert _{\textrm{Mink}}^{\Sigma\left(\mathbf{m}\vec{\mathbf{v}}\right)\cdot\mathbf{e}_{j}}
\end{equation}
because the total number of $r_{j}$s that occur in $M_{H}\left(\mathbf{m}\vec{\mathbf{v}}\right)$
is the $\mathbf{e}_{j}$-coordinate of $\Sigma\left(\mathbf{m}\vec{\mathbf{v}}\right)$.
Our choice of $\mathbf{m}\vec{\mathbf{v}}$ gives: 
\begin{equation}
L-\frac{1}{p}\ln^{0.6}x<\Sigma\left(\mathbf{m}\vec{\mathbf{v}}\right)\cdot\mathbf{e}_{j}<L+\frac{1}{p}\ln^{0.6}x,\textrm{ }\forall j\in\left\{ 0,\ldots,p-1\right\} 
\end{equation}
As such: 
\begin{equation}
\left\Vert r_{j}\right\Vert _{\textrm{Mink}}^{\Sigma\left(\mathbf{m}\vec{\mathbf{v}}\right)\cdot\mathbf{e}_{j}}\leq\left\Vert r_{j}\right\Vert _{\textrm{Mink}}^{L}\left(\max\left\{ \left\Vert r_{j}\right\Vert _{\textrm{Mink}},\left\Vert r_{j}\right\Vert _{\textrm{Mink}}^{-1}\right\} \right)^{\frac{1}{p}\ln^{0.6}x},\textrm{ }\forall j\in\left\{ 0,\ldots,p-1\right\} 
\end{equation}
This gives us: 
\begin{align*}
\left\Vert M_{H}\left(\mathbf{m}\vec{\mathbf{v}}\right)\right\Vert _{\textrm{Mink}} & \leq\prod_{j=0}^{p-1}\left(\left\Vert r_{j}\right\Vert _{\textrm{Mink}}^{L}\underbrace{\left(\max\left\{ \left\Vert r_{j}\right\Vert _{\textrm{Mink}},\left\Vert r_{j}\right\Vert _{\textrm{Mink}}^{-1}\right\} \right)^{\frac{1}{p}\ln^{0.6}x}}_{\prod_{j=0}^{p-1}\textrm{ of this }=\overline{\rho}_{H}^{\ln^{0.6}x}}\right)\\
\left(\times\frac{\left\Vert r_{j}\right\Vert _{\textrm{Mink}}^{L}}{\left\Vert r_{j}\right\Vert _{\textrm{Mink}}^{L}}\right); & \leq\overline{\rho}_{H}^{\frac{1}{p}\ln^{0.6}x}\prod_{j=0}^{p-1}\left\Vert r_{j}\right\Vert _{\textrm{Mink}}^{L}\\
\left(\left\Vert r_{j}\right\Vert _{\textrm{Mink}}\leq\overline{\rho}_{H},\textrm{ }\forall j\right); & \leq\overline{\rho}_{H}^{\frac{1}{p}\ln^{0.6}x}\rho_{H}^{L}
\end{align*}
Note that this proves (\ref{eq:selective estimate for M_H of v arrow}).
Now, turning to $X_{H}\left(\mathbf{m}\vec{\mathbf{v}}\right)$, note
that: 
\begin{equation}
\prod_{i=0}^{p-1}r_{i}^{\Sigma_{k-1}\left(\mathbf{m}\vec{\mathbf{v}}\right)\cdot\mathbf{e}_{i}}=M_{H}\left(m_{1}\mathbf{v}_{1},\ldots,m_{k}\mathbf{v}_{k-1}\right)
\end{equation}
 for $k\in\left\{ 1,\ldots,L\right\} $. As such, (\ref{eq:selective estimate for M_H of v arrow})
yields: 
\begin{equation}
\left\Vert \prod_{i=0}^{p-1}r_{i}^{\Sigma_{k-1}\left(\mathbf{m}\vec{\mathbf{v}}\right)\cdot\mathbf{e}_{i}}\right\Vert _{\textrm{Mink}}=\left\Vert M_{H}\left(m_{1}\mathbf{v}_{1},\ldots,m_{k}\mathbf{v}_{k-1}\right)\right\Vert _{\textrm{Mink}}\leq\overline{\rho}_{H}^{\frac{1}{p}\ln^{0.6}x}\rho_{H}^{k-1}
\end{equation}
which gives: 
\begin{align*}
\left\Vert X_{H}\left(\vec{\mathbf{v}}\right)\right\Vert _{\textrm{Mink}} & \leq\sum_{k=1}^{L}\overline{\rho}_{H}^{\frac{1}{p}\ln^{0.6}x}\rho_{H}^{k-1}\prod_{j=0}^{p-1}\underbrace{\left(\sum_{h=0}^{m_{k}\left(\mathbf{v}_{k}\cdot\mathbf{e}_{j}\right)-1-1}\left\Vert r_{j}\right\Vert _{\textrm{Mink}}^{h}\left\Vert c_{j}\right\Vert _{\textrm{Mink}}\right)}_{0\textrm{ for all}j\neq j_{k}}\\
\left(\mathbf{v}_{k}=\mathbf{e}_{j_{k}}\textrm{ for some }j_{k}\in\left\{ 0,\ldots,p-1\right\} \right); & \leq\overline{\rho}_{H}^{\frac{1}{p}\ln^{0.6}x}\sum_{k=1}^{L}\rho_{H}^{k-1}\sum_{h=0}^{m_{k}-1}R_{H}^{h}\\
 & \leq\overline{\rho}_{H}^{\frac{1}{p}\ln^{0.6}x}\sum_{k=0}^{L-1}\rho_{H}^{k}\frac{R_{H}^{m_{k+1}}-1}{R_{H}-1}\\
 & =\frac{\overline{\rho}_{H}^{\frac{1}{p}\ln^{0.6}x}}{R_{H}-1}\sum_{k=0}^{L-1}\rho_{H}^{k}R_{H}^{m_{k+1}}-\frac{1}{R_{H}-1}\frac{1-\rho_{H}^{L}}{1-\rho_{H}}\overline{\rho}_{H}^{\frac{1}{p}\ln^{0.6}x}
\end{align*}
Since $\sum_{h=1}^{k}m_{h}=\Sigma\left(\Sigma_{k}\left(\mathbf{m}\vec{\mathbf{v}}\right)\right)$,
adding the inequalities: 
\begin{equation}
k-\frac{1}{p}\ln^{0.6}x<\Sigma_{k}\left(\mathbf{m}\vec{\mathbf{v}}\right)\cdot\mathbf{e}_{j}<k+\frac{1}{p}\ln^{0.6}x,\textrm{ }\forall k\in\left\{ 1,\ldots,L\right\} ,\forall j\in\left\{ 0,\ldots,p-1\right\} 
\end{equation}
over all $j$ yields: 
\begin{equation}
pk-\ln^{0.6}x<\sum_{h=1}^{k}m_{h}<pk+\ln^{0.6}x\label{eq:sum of m_h inequality}
\end{equation}
Thus: 
\begin{equation}
\left\Vert X_{H}\left(\mathbf{m}\vec{\mathbf{v}}\right)\right\Vert _{\textrm{Mink}}\ll\overline{\rho}_{H}^{\frac{1}{p}\ln^{0.6}x}R_{H}^{pL+\ln^{0.6}x}=\left(R_{H}\overline{\rho}_{H}^{1/p}\right)^{\ln^{0.6}x}R_{H}^{pL}
\end{equation}
Note also that (\ref{eq:sum of m_h inequality}) implies: 
\begin{equation}
-p\left(k-1\right)-\ln^{0.6}x<-\sum_{h=1}^{k-1}m_{h}<-p\left(k-1\right)+\ln^{0.6}x
\end{equation}
which if we add to (\ref{eq:sum of m_h inequality}) gives: 
\begin{equation}
p-2\ln^{0.6}x<m_{k}<p+2\ln^{0.6}x,\textrm{ }\forall k\in\left\{ 1,\ldots,L\right\} 
\end{equation}
Hence, we can also write: 
\begin{align*}
\left\Vert X_{H}\left(\mathbf{m}\vec{\mathbf{v}}\right)\right\Vert _{\infty} & \leq\frac{\overline{\rho}_{H}^{\frac{1}{p}\ln^{0.6}x}}{R_{H}-1}\sum_{k=0}^{L-1}\rho_{H}^{k}R_{H}^{m_{k+1}}-\frac{1}{R_{H}-1}\frac{1-\rho_{H}^{L}}{1-\rho_{H}}\overline{\rho}_{H}^{\frac{1}{p}\ln^{0.6}x}\\
 & <\frac{\overline{\rho}_{H}^{\frac{1}{p}\ln^{0.6}x}}{R_{H}-1}\sum_{k=0}^{L-1}\rho_{H}^{k}R_{H}^{p+2\ln^{0.6}x}-\frac{1}{R_{H}-1}\frac{1-\rho_{H}^{L}}{1-\rho_{H}}\overline{\rho}_{H}^{\frac{1}{p}\ln^{0.6}x}\\
 & =\frac{1-\rho_{H}^{L}}{1-\rho_{H}}\frac{\overline{\rho}_{H}^{\frac{1}{p}\ln^{0.6}x}}{R_{H}-1}\left(R_{H}^{p+2\ln^{0.6}x}-1\right)\\
 & <\frac{1-\rho_{H}^{L}}{1-\rho_{H}}\frac{\overline{\rho}_{H}^{\frac{1}{p}\ln^{0.6}x}}{R_{H}-1}R_{H}^{p+2\ln^{0.6}x}\\
 & \ll\left(R_{H}^{2}\overline{\rho}_{H}^{1/p}\right)^{\ln^{0.6}x}
\end{align*}

Q.E.D.

\section{\label{subsec:First-Passage-Stabilization-impl}First-Passage Stabilization
implies Almost-Boundedness}

In this section, we will prove \textbf{Theorem \ref{thm:almost boundedness}}
assuming \textbf{Lemma \ref{lem:stabilization of first passage}}.

\vphantom{} 
\begin{thm}[Alternate form of \textbf{Theorem \ref{thm:almost boundedness}}]
\label{thm:alt form of main theorem}Let all parameters be as given
in \textbf{Assumption \ref{assu:initial parameter assumptions}} and
\textbf{Proposition \ref{prop:gamma inequality proposition}}.

If \textbf{Lemma \ref{lem:stabilization of first passage} }is true,
then: 
\begin{equation}
\mathbb{P}\left(\inf_{n\geq1}\left\Vert \textrm{Syr}_{H}^{\circ n}\left(\mathcal{N}_{x}\right)\right\Vert _{\textrm{Mink}}\le N_{0}\right)\ge1-O\left(\left(\ln N_{0}\right)^{-c}\right)\label{eq:alt form of main theorem}
\end{equation}
for all $x,N_{0}\geq2$, for some absolute constant $c>0$.

Moreover particular, (\ref{eq:alt form of main theorem}) implies
\textbf{Theorem \ref{thm:almost boundedness}}. 
\end{thm}
Proof: We may assume that $N_{0}$ is larger than any given absolute
constant, or else (\ref{eq:alt form of main theorem}) is trivial.
Now, let: 
\begin{equation}
\textrm{Syr}_{H,\textrm{min}}\left(z\right)\overset{\textrm{def}}{=}\inf_{n\geq1}\left\Vert \textrm{Syr}_{H}^{\circ n}\left(z\right)\right\Vert _{\textrm{Mink}}
\end{equation}
and let: 
\begin{equation}
E_{N_{0}}\overset{\textrm{def}}{=}\left\{ z\in\mathcal{O}_{K}:\textrm{Syr}_{H,\textrm{min}}\left(z\right)\leq N_{0}\right\} 
\end{equation}
denote the set of all starting positions $z$ of orbits whose norms
reach $N_{0}$ or below. Let $\alpha=1.001$, let $x\ge2$, and let
$\mathcal{N}_{y}$ be the random variable: 
\begin{equation}
\mathcal{N}_{y}=\textrm{Log}\left(\left\{ z\in\mathcal{O}_{K}:y\leq\left\Vert z\right\Vert _{\textrm{Mink}}\leq y^{\alpha}\right\} \right).
\end{equation}
Next, let $B_{x}$ denote the event: 
\begin{equation}
B_{x}\overset{\textrm{def}}{=}\left\{ T_{x}\left(\mathcal{N}_{x^{\alpha}}\right)<+\infty\right\} \cap\left\{ \ensuremath{\textrm{Pass}_{x}\left(\mathcal{N}_{x^{\alpha}}\right)\in E_{N_{0}}}\right\} .
\end{equation}
This is the event that the trajectory of $z$ under $H$ reaches a
magnitude of $x$ or less and then later reaches a magnitude of $N_{0}$
or less.

If we also have $T_{x}\left(\mathcal{N}_{x^{\alpha^{2}}}\right)<+\infty$
and $\textrm{Pass}_{x}\left(\mathcal{N}_{x^{\alpha}}\right)\in E_{N_{0}}$
then, because $x<x^{\alpha}$, and because any walk that starts at
$\mathcal{N}_{x^{\alpha^{2}}}$ that goes to a point with absolute
value $x$ must first pass through a point with absolute value $x^{\alpha}$,
we have:

\begin{equation}
T_{x^{\alpha}}\left(\mathcal{N}_{\alpha^{2}}\right)\le T_{x}\left(\mathcal{N}_{x^{\alpha^{2}}}\right)<+\infty
\end{equation}
and, because it is the same trajectory, just for possibly longer steps:

\begin{equation}
\left\{ \textrm{Syr}_{H}^{\circ n}\left(\textrm{Pass}_{x}\left(\mathcal{N}_{x^{\alpha^{2}}}\right)\right):n\geq1\right\} \subseteq\left\{ \textrm{Syr}_{H}^{\circ n}\left(\textrm{Pass}_{x^{\alpha}}\left(\mathcal{N}_{x^{\alpha^{2}}}\right)\right):n\geq1\right\} 
\end{equation}
These imply:

$\textrm{Syr}_{H,\min}\left(\textrm{Pass}_{x}\left(\mathcal{N}_{x^{\alpha^{2}}}\right)\right)\le\textrm{Syr}_{H,\min}\left(\textrm{Pass}_{x^{\alpha}}\left(\mathcal{N}_{x^{\alpha^{2}}}\right)\right)\le N_{0}$
so that $B_{x^{\alpha}}$. holds. This implies the event $B_{x^{\alpha}}$
is more likely than: 
\begin{equation}
\left\{ \textrm{Pass}_{x}\left(\mathcal{N}_{x^{\alpha^{2}}}\right)\in E_{N_{0}}\right\} \cap\left\{ T_{x}\left(\mathcal{N}_{x^{\alpha^{2}}}\right)<+\infty\right\} 
\end{equation}
which gives:

\begin{equation}
\mathbb{P}\left(B_{x^{\alpha}}\right)\ge\mathbb{P}\left(\left\{ \textrm{Pass}_{x}\left(\mathcal{N}_{x^{\alpha^{2}}}\right)\in E_{N_{0}}\right\} \cap\left\{ \textrm{Pass}_{x}\left(\mathcal{N}_{x^{\alpha^{2}}}\right)\in E_{N_{0}}\right\} \cap\left\{ T_{x}\left(\mathcal{N}_{x^{\alpha^{2}}}\right)<+\infty\right\} \right)
\end{equation}
Because of (\ref{eq:passing time estimate}), the difference between
these and the simple passing time probability is $x^{-c}$ for some
constant $c>0$, which leaves us with: 
\begin{equation}
\mathbb{P}\left(B_{x^{\alpha}}\right)\ge\mathbb{P}\left(\textrm{Pass}_{x}\left(\mathcal{N}_{x^{\alpha^{2}}}\right)\in E_{N_{0}}\right)-\mathcal{O}\left(x^{-c}\right)
\end{equation}
Applying ($\ref{eq:passage location comparison}$), this lower bound
becomes:

\begin{align*}
\mathbb{P}\left(B_{x^{\alpha}}\right) & \ge\mathbb{P}\left(\textrm{Pass}_{x}\left(\mathcal{N}_{x^{\alpha}}\right)\in E_{N_{0}}\right)-\mathcal{O}\left(x^{-\left(\alpha-1\right)c^{\prime}}\right)\\
 & \geq\mathbb{P}\left(B_{x}\right)-\mathcal{O}\left(x^{-\left(\alpha-1\right)c^{\prime}}\right)
\end{align*}

Now, let $J$ be some number depending on $x$ and $N_{0}$ such that
$y=x^{\alpha^{-J}}$ is less than $N_{0}^{1/\alpha}$. Assuming $N_{0}$
is large, we can substitute $y^{\alpha^{j-2}}$ for $x$ in the above
estimates to get

\begin{equation}
\mathbb{P}\left(B_{y^{\alpha^{j-1}}}\right)\ge\mathbb{P}\left(B_{y^{\alpha^{j-2}}}\right)-\mathcal{O}\left(y^{-\left(\alpha^{j}-1\right)c}\right),\textrm{ }\forall j\in\left\{ 1,\ldots,J\right\} .
\end{equation}
Because $\left|\mathcal{N}_{y}\right|\le y^{\alpha}\le N_{0}$, the
event $B_{y^{\alpha^{-1}}}$ is equivalent to $\left\{ T_{y^{\alpha^{-1}}}\left(\mathcal{N}_{y}\right)<+\infty\right\} $,
as the passage part of the definition of the set $B_{y^{\alpha^{-1}}}$
is satisfied by design. By ($\ref{eq:passing time estimate}$), this
implies that: 
\begin{equation}
\mathbb{P}\left(B_{y^{\alpha^{-1}}}\right)\geq1-\mathcal{O}\left(y^{-c}\right).
\end{equation}
Notice the telescoping sum over $y^{\alpha^{j}}$ in the indices of
the $B$-sets; summing over them, we have

\begin{equation}
\mathbb{P}\left(B_{y^{\alpha^{J-1}}}\right)\ge1-\mathcal{O}\left(y^{-c}+\sum_{j=1}^{J}y^{-\left(\alpha^{j}-1\right)c}\right)
\end{equation}

Ultimately, this is bounded by the decay of $y^{-\left(\alpha^{J}-1\right)c}$.
By construction, $y\ge N_{0}^{1/\alpha^{2}}$ and $y^{\alpha^{J}}=x$,
so:

\begin{equation}
P\left(B_{x^{1/\alpha}}\right)=P\left(B_{y^{\alpha^{J-1}}}\right)\ge1-\mathcal{O}\left(y^{-\left(\alpha^{J}-1\right)c}\right)\ge1-\mathcal{O}\left(N_{0}^{-\left(\alpha^{J}-1\right)c/\alpha^{2}}\right)\label{eq:fin}
\end{equation}

Now, by our construction, $J$ is ultimately order $-\ln\left(\frac{\ln N_{0}}{\ln x}\right)$
as $x\rightarrow\infty$. It obtains larger values when $N_{0}$ is
smaller with respect to $x$. We let $N_{0}>2$ (if we let $N_{0}=1$
then $J$ is identically $0$), so that $J=\mathcal{O}\left(\ln\left(\ln x\right)\right)$
as $x\rightarrow\infty$. From the asymptotic growth rate of J, we
have that the error term in Equation~\ref{eq:fin} decays by at least
$N_{0}^{-\left(\alpha^{\ln\ln x}-1\right)c}$. This expression is
still polynomial, though with a coefficient that depends very slowly
on $x$. When $B_{x^{1/\alpha}}$ holds, we have that $\textrm{Pass}_{x^{1/\alpha}}\left(\mathcal{N}_{x}\right)$
must ultimately lie in the orbit of $H$ and: 
\begin{equation}
\textrm{Syr}_{H,\min}\left(x\right)\le\textrm{Syr}_{H,\min}\left(\textrm{Pass}_{x^{1/\alpha}}\left(\mathcal{N}_{x}\right)\right)\le N_{0}.
\end{equation}
The set of events where the iterates of $H$ on $\mathcal{N}_{x}$
that stay larger than $N_{0}$ are in the complement of $B_{x^{1/\alpha}}$
and we have:

\begin{equation}
\mathbb{P}\left(\textrm{Syr}_{H,\min}\left(\mathcal{N}_{x}\right)>N_{0}\right)<CN_{0}^{-c\left(\alpha^{\ln\ln x}-1\right)}
\end{equation}
where $C$ and $c$ are independent of $N_{0}$ and $x$. This rate
decays super-logarithmically, so logarithmically, thus giving us (\ref{eq:alt form of main theorem}).

This then implies \textbf{Theorem \ref{thm:almost boundedness}} by
the same argument Tao uses to show his analogous result (\textbf{Theorem
3.1}) implies \textbf{Theorem \ref{thm:(Tao-2019)}}.

Q.E.D.

\vphantom{}Thus, we have reduced the problem to proving \textbf{Lemma
\ref{lem:stabilization of first passage}}. This will be done over
the next two sections.

\section{\label{sec:-adic-distribution-of}$\Lambda$-adic distribution of
iterates \& Geometric Comparisons}

In this section, we prove our \textbf{Geometric Comparison Lemma}
(\textbf{Lemma \ref{lem:Geometric comparison}}). We prove this result
in three steps. The first is establishing a tail bound on the total
number of applications of branches to $\mathcal{N}$'s trajectory.
\vphantom{} 
\begin{prop}
\label{prop:tail bound}Let everything be as given in \textbf{Lemma
\ref{lem:Geometric comparison}}. Then, there is a constant $c>0$
depending only on $c_{0}$ so that: 
\begin{equation}
\max_{0\leq j<p}\mathbb{P}\left(\#_{j}^{\left(n\right)}\left(\mathcal{N}\right)\geq n^{\prime}\right)\ll p^{-cn}\label{eq:tail bound}
\end{equation}
\end{prop}
Proof: Fix $j\in\left\{ 0,\ldots,p-1\right\} $. To begin with, we
use hypothesis \textbf{MH4}: the endomorphisms $r_{0},\ldots,r_{p-1}$
that occur in the branches of $H$ can all be factored as: 
\begin{equation}
r_{j}=\mu_{j}u_{j},\textrm{ }\forall j\in\left\{ 0,\ldots,p-1\right\} 
\end{equation}
where the $\mu_{j}$s are elements of $K^{\times}$ and where the
$u_{j}$s are elements of $\textrm{End}_{\mathbb{Q}}^{\times}K$ which
are $\Lambda$-adic isometries, so that: 
\begin{equation}
v_{\Lambda}\left(r_{j}\left(z\right)\right)=v_{\Lambda}\left(\mu_{j}\right)+v_{\Lambda}\left(z\right),\textrm{ }\forall z\in K,\textrm{ }\forall j\in\left\{ 0,\ldots,p-1\right\} 
\end{equation}
Also, note that while the $u_{j}$s need not commute with one another,
the $\mu_{j}$s do commute, both among themselves, and with the $u_{j}$s.
Given any string $\mathbf{j}=\left(j_{1},\ldots,j_{\left|\mathbf{j}\right|}\right)$
of the integers $\left\{ 0,\ldots,p-1\right\} $, we then have that,
for: 
\begin{equation}
r_{\mathbf{j}}=r_{j_{1}}r_{j_{2}}\cdots r_{j_{m}}=\prod_{h=1}^{m}r_{j_{h}}
\end{equation}
we can write: 
\begin{equation}
r_{\mathbf{j}}=\left(\prod_{h=1}^{m}\mu_{j_{h}}\right)\times\underbrace{\left(u_{j_{1}}\cdots u_{j_{m}}\right)}_{u_{\mathbf{j}}}=\left(\prod_{k=0}^{p-1}\mu_{k}^{\#_{k}\left(\mathbf{j}\right)}\right)\times u_{\mathbf{j}}
\end{equation}
where the exponent of $\mu_{k}$ in the product is the number of $k$s
in $\mathbf{j}$.

Following Tao, we write:

\begin{equation}
\mathbb{P}\left(\#_{j}^{\left(n\right)}\left(\mathcal{N}\right)\geq n^{\prime}\right)=\sum_{k=1}^{n}\mathbb{P}\left(\#_{j}^{\left(k-1\right)}\left(\mathcal{N}\right)<n^{\prime}\le\#_{j}^{\left(k\right)}\left(\mathcal{N}\right)\right)
\end{equation}
from which it follows that our proof will be complete provided we
can demonstrate that: 
\begin{equation}
\max_{1\leq k\leq n}\mathbb{P}\left(\#_{j}^{\left(k-1\right)}\left(\mathcal{N}\right)<n^{\prime}\le\#_{j}^{\left(k\right)}\left(\mathcal{N}\right)\right)\ll p^{-cn}\label{eq:tail bound goal}
\end{equation}
for some constant $c$. We can then take the largest of these upper
bounds with respect to $i\in J^{c}$.

So, let $k\in\left\{ 1,\ldots,n\right\} $ be arbitrary, and let:
\begin{equation}
N_{k}\overset{\textrm{def}}{=}\sum_{h=1}^{k}\nu_{h}\left(\mathcal{N}\right)
\end{equation}
denote the lengths of the $n$-path of $\mathcal{N}$. By a straightforward
computation, we have that: 
\begin{equation}
H^{\circ N_{k}}\left(\mathcal{N}\right)=\underbrace{\left(\prod_{h=0}^{p-1}\mu_{h}^{\#_{h}\left(\mathbf{j}_{N_{k}}\left(\mathcal{N}\right)\right)}\right)}_{\textrm{call this }\Pi_{N_{k}}}u_{\mathbf{j}_{N_{k}}\left(\mathcal{N}\right)}\mathcal{N}+\sum_{m=0}^{N_{k}-1}\underbrace{\left(\prod_{h=0}^{p-1}\mu_{h}^{\#_{h}\left(\mathbf{j}_{m}\left(\mathcal{N}\right)\right)}\right)}_{\textrm{call this }\Pi_{m}}u_{\mathbf{j}_{m}\left(\mathcal{N}\right)}c_{j_{m}\left(\mathcal{N}\right)}\label{eq:rough}
\end{equation}
where, recall: 
\[
\mathbf{j}_{m}\left(\mathcal{N}\right)=\left(\left[\mathcal{N}\right]_{\Lambda},\left[H\left(\mathcal{N}\right)\right]_{\Lambda},\ldots,\left[H^{\circ m-1}\left(\mathcal{N}\right)\right]_{\Lambda}\right)
\]
and: 
\[
j_{m}\left(\mathcal{N}\right)=\left[H^{\circ m}\left(\mathcal{N}\right)\right]_{\Lambda}
\]
Since right-hand side of (\ref{eq:rough}) is an element of $\mathcal{O}_{K}$,
there is an $i\in\left\{ 0,\ldots,p-1\right\} $ so that $H^{\circ N_{k}}\left(\mathcal{N}\right)$
is congruent to $\lambda_{i}$ mod $\Lambda$, where $\lambda_{i}$
is any representative element of the coset $\Lambda_{i}$ that we
have chosen for the occasion: 
\begin{equation}
\Pi_{N_{k}}u_{\mathbf{j}_{N_{k}}\left(\mathcal{N}\right)}\mathcal{N}+\sum_{m=0}^{N_{k}-1}\Pi_{m}u_{\mathbf{j}_{m}\left(\mathcal{N}\right)}c_{j_{m+1}\left(\mathcal{N}\right)}\overset{\Lambda}{\equiv}\lambda_{i}\label{eq:local normality lemma congruence}
\end{equation}

By \textbf{MH4}'s hypothesis, we have: 
\[
v_{\Lambda}\left(\mu_{h}\right)\leq-1,\textrm{ }\forall h\in\left\{ 0,\ldots,p-1\right\} 
\]
As such: 
\begin{equation}
v_{\Lambda}\left(\Pi_{m}\right)\leq-\left|\mathbf{j}_{m}\left(\mathcal{N}\right)\right|,\textrm{ }\forall m\label{eq:lambda valuation of the mu h product}
\end{equation}
That is, the endomorphisms $r_{i}$ of each branch catalogued in the
string $\mathbf{j}_{m}\left(\mathcal{N}\right)$ has the effect of
decreasing their inputs' $\Lambda$-adic valuation by at least $1$.
So, $\Pi_{m}$ has a denominator which is ``divisible'' by $\Lambda^{m}$,
and we could therefore clear denominators on both sides of (\ref{eq:local normality lemma congruence})
by multiplying by an appropriately chosen number. This would be a
so-called \textbf{uniformizer of $\Lambda$ over $\mathcal{O}_{K}$}:
an element $\pi_{\Lambda}\in\mathcal{O}_{K}$ so that $z/\pi_{\Lambda}\in\mathcal{O}_{K}$
and $v_{\Lambda}\left(z/\pi_{\Lambda}\right)=0$ for all $z\in\mathcal{O}_{K}$
with $v_{\Lambda}\left(z\right)=1$. This uniformizer will exist in
$\mathcal{O}_{K}$ only if $\Lambda$ is a principal ideal, however,
this is no trouble; as it is a standard fact of algebraic number theory
that $\mathcal{O}_{K_{\Lambda}}$, the ring of $\Lambda$-adic integers
of $K$ obtained by completing $\mathcal{O}_{K}$ with respect to
$\left|\cdot\right|_{\Lambda}$, always contains a uniformizer $\pi_{\Lambda}$
for $\Lambda$. Because of this, we can embed (\ref{eq:local normality lemma congruence})
in $\mathcal{O}_{K_{\Lambda}}$, which has the effect of changing
the congruence mod $\Lambda$ to one mod the ideal $\pi_{\Lambda}\mathcal{O}_{K_{\Lambda}}$,
which we denote with an $\pi_{\Lambda}$. Multiplying both sides of
(\ref{eq:local normality lemma congruence}) by $\pi_{\Lambda}^{N_{k}}$
and rearranging terms gives us:

\begin{equation}
\pi_{\Lambda}^{N_{k}}\Pi_{N_{k}}u_{\mathbf{j}_{N_{k}}\left(\mathcal{N}\right)}\mathcal{N}\overset{\pi_{\Lambda}^{N_{k}+1}}{\equiv}\pi_{\Lambda}^{N_{k}}\lambda_{j}-\sum_{m=0}^{N_{k}-1}\pi_{\Lambda}^{N_{k}}\Pi_{m}u_{\mathbf{j}_{m}\left(\mathcal{N}\right)}c_{j_{m}\left(\mathcal{N}\right)}\label{eq:rough2}
\end{equation}
The $m=N_{k}$ case of (\ref{eq:lambda valuation of the mu h product})
tells us that $v_{\Lambda}\left(\pi_{\Lambda}^{N_{k}}\Pi_{N_{k}}\right)=0$,
and hence, that $\pi_{\Lambda}^{N_{k}}\Pi_{N_{k}}$ is a unit of $\mathcal{O}_{K_{\Lambda}}$,
which means that multiplying by $\pi_{\Lambda}^{N_{k}}\Pi_{N_{k}}$
is a $\Lambda$-adic isometry, that is, an isometry of $\mathcal{O}_{K_{\Lambda}}$.
Since $u_{\mathbf{j}_{N_{k}}\left(\mathcal{N}\right)}$ is also an
isometry of $\mathcal{O}_{K_{\Lambda}}$, so is $\pi_{\Lambda}^{N_{k}}\Pi_{N_{k}}u_{\mathbf{j}_{N_{k}}\left(\mathcal{N}\right)}$,
and we can apply the inverse of this map to both sides of (\ref{eq:rough2}),
which then shows that $\mathcal{N}$ belongs to a single equivalence
class of $\mathcal{O}_{K_{\Lambda}}$ mod $\pi_{\Lambda}^{N_{k}+1}$.
Since $\mathcal{N}$ is in $\mathcal{O}_{K}$, we see that $\mathcal{N}$
therefore belongs to a single equivalence class of $\Lambda^{N_{k}+1}$.

Suppose that the event $\#_{j}^{\left(k-1\right)}\left(\mathcal{N}\right)<n^{\prime}\le\#_{j}^{\left(k\right)}\left(\mathcal{N}\right)$
occurs. Then $N_{k}\geq\#_{j}^{\left(k\right)}\left(\mathcal{N}\right)$
tells us that we can project down from mod $\Lambda^{N_{k}+1}$ to
mod $\Lambda^{n^{\prime}}$ , the unique coset in $\mathcal{O}_{K}/\Lambda^{N_{k}+1}$
containing $\mathcal{N}$ projects down to a unique coset in $\mathcal{O}_{K}/\Lambda^{n^{\prime}}$.
By the approximate uniform distribution of $\mathcal{N}$ (assumption
(\ref{eq:uniformity hypothesis})), we then have that the probability
that $\mathcal{N}$ is congruent to any given $\lambda\in\mathcal{O}_{K}$
mod $\Lambda^{n^{\prime}}$ satisfies: 
\begin{equation}
\mathbb{P}\left(\mathcal{N}\overset{\Lambda^{n^{\prime}}}{\equiv}\lambda\right)\ll p^{-n^{\prime}}\label{eq:fun-1}
\end{equation}
because the index of $\Lambda$ in $\mathcal{O}_{K}$ being $p$ then
guarantees that the index of $\Lambda^{n^{\prime}}$ in $\mathcal{O}_{K}$
is $p^{n^{\prime}}$, and thus, that the probability of selecting
a given equivalence class of $\mathcal{O}_{K}$ mod $\Lambda^{n^{\prime}}$
is $p^{-n^{\prime}}$. In particular, (\ref{eq:fun-1}) gives this
probability subject to a given value of $k$, and shows that the distribution
of $\mathcal{N}$ mod $\Lambda^{n^{\prime}}$ is independent of $k$.

Finally, in order to get an upper bound for $\mathbb{P}\left(\#_{j}^{\left(k-1\right)}\left(\mathcal{N}\right)<n^{\prime}\le\#_{j}^{\left(k\right)}\left(\mathcal{N}\right)\right)$,
we sum the upper bound (\ref{eq:fun-1}) over all choices of $\mathbf{v}=\left(v_{1},\ldots,v_{k}\right)\in\mathbb{N}_{1}^{k}$
so that the event $\#_{j}^{\left(m\right)}\left(\mathcal{N}\right)=v_{m}$
occurs for all $m\in\left\{ 1,\ldots,k\right\} $. This gives: 
\begin{equation}
\mathbb{P}\left(\#_{j}^{\left(k-1\right)}\left(\mathcal{N}\right)<n^{\prime}\le\#_{j}^{\left(k\right)}\left(\mathcal{N}\right)\right)\ll\sum_{v_{1},\ldots,v_{k}\in\mathbb{N}_{1}:\Sigma\left(\mathbf{v}\right)<n^{\prime}}p^{-n^{\prime}}=\binom{n^{\prime}-1}{k}p^{-n^{\prime}}
\end{equation}
where $\Sigma\left(\mathbf{v}\right)$ is the sum of the entries of
$\mathbf{v}$. Applying Stirling's approximation to the binomial coefficient
yields: 
\begin{equation}
\binom{n^{\prime}-1}{k}p^{-n^{\prime}}\ll p^{-cn}
\end{equation}
for some constant $c>0$. Hence: 
\begin{align*}
\mathbb{P}\left(\#_{j}^{\left(n\right)}\left(\mathcal{N}\right)\geq n^{\prime}\right) & =\sum_{k=1}^{n}\mathbb{P}\left(\#_{j}^{\left(k-1\right)}\left(\mathcal{N}\right)<n^{\prime}\le\#_{j}^{\left(k\right)}\left(\mathcal{N}\right)\right)\\
 & \ll\sum_{k=1}^{n}p^{-cn}\\
 & \ll p^{-cn}
\end{align*}
Taking maximums over $j\in\left\{ 0,\ldots,p-1\right\} $ completes
the proof.

Q.E.D.

\vphantom{}

\textbf{Proof of \ref{lem:Geometric comparison}}: By \textbf{Lemma
\ref{lem:Chernoff Bound}} (the Chernoff-type bound from Tao's paper)
we have that: 
\begin{equation}
\mathbb{P}\left(\mathbf{Run}\left(p\right)^{n}\geq n^{\prime}\right)\ll p^{-cn}
\end{equation}
where $\mathbf{Run}\left(p\right)^{n}\geq n^{\prime}$ means that
the scalar multipliers in each entry of $\mathbf{Run}\left(p\right)^{n}$
is $\geq n^{\prime}$.

Now, as constructed, we have that both $\vec{a}^{\left(n\right)}\left(\mathcal{N}\right)$
and $\mathbf{Run}\left(p\right)^{n}$ takes values in $\left(\mathbb{N}_{1}B\right)^{n}$.
By definition (\ref{eq:def of TV distance}) of the total variation
distance, we then get: 
\begin{equation}
d_{\textrm{TV}}\left(\vec{a}^{\left(n\right)}\left(\mathcal{N}\right),\mathbf{Run}\left(p\right)^{n}\right)=\sum_{\vec{\mathbf{v}}\in B^{n}}\sum_{\mathbf{m}\in\mathbb{N}_{1}^{n}}\left|\mathbb{P}\left(\vec{a}^{\left(n\right)}\left(\mathcal{N}\right)=\mathbf{m}\vec{\mathbf{v}}\right)-\mathbb{P}\left(\mathbf{Run}\left(p\right)^{n}=\mathbf{m}\vec{\mathbf{v}}\right)\right|
\end{equation}
where $\vec{\mathbf{v}}=\left(\mathbf{v}_{1},\ldots,\mathbf{v}_{n}\right)$,
$\mathbf{m}=\left(m_{1},\ldots,m_{n}\right)$, and $\mathbf{m}\vec{\mathbf{v}}=\left(m_{1}\mathbf{v}_{1},\ldots,m_{n}\mathbf{v}_{n}\right)$
where the $m_{k}$s are positive integers and the $\mathbf{v}_{k}$s
are vectors in $B$. By (\ref{eq:PMF of many Runs}), we have: 

\begin{equation}
\mathbb{P}\left(\mathbf{Run}\left(p\right)^{n}=\left(m_{1}\mathbf{v}_{1},\ldots,m_{n}\mathbf{v}_{n}\right)\right)=\frac{p-1}{p}\frac{\left(1-\frac{1}{p}\right)^{\Sigma\left(\mathbf{m}\right)}}{\left(p-1\right)^{2n}}\prod_{k=2}^{n}\left[\mathbf{v}_{k}\neq\mathbf{v}_{k-1}\right]
\end{equation}
where $\Sigma\left(\mathbf{m}\right)$ is the sum of the entries of
$\mathbf{m}$. Hence: 
\begin{align}
d_{\textrm{TV}}\left(\vec{a}^{\left(n\right)}\left(\mathcal{N}\right),\mathbf{Run}\left(p\right)^{n}\right) & =\sum_{\vec{\mathbf{v}}\in B^{n*}}\sum_{\mathbf{m}\in\mathbb{N}_{1}^{n}}\left|\mathbb{P}\left(\vec{a}^{\left(n\right)}\left(\mathcal{N}\right)=\mathbf{m}\vec{\mathbf{v}}\right)-\frac{p-1}{p}\frac{\left(1-\frac{1}{p}\right)^{\Sigma\left(\mathbf{m}\right)}}{\left(p-1\right)^{2n}}\right|\label{eq:second step of final part of comparison lemma proof}
\end{align}
Where: 
\[
B^{n*}\overset{\textrm{def}}{=}\left\{ \vec{\mathbf{v}}=\left(\mathbf{v}_{1},\ldots,\mathbf{v}_{n}\right)\in B^{n}:\mathbf{v}_{1}\neq\mathbf{v}_{2},\mathbf{v}_{2}\neq\mathbf{v}_{3},\ldots,\mathbf{v}_{n-1}\neq\mathbf{v}_{n}\right\} 
\]
Letting $L$ be a large positive integer, we divide the domain of
the $\mathbf{m}$-sum based on the size of $\Sigma\left(\mathbf{m}\right)$:
\begin{equation}
\sum_{\mathbf{m}\in\mathbb{N}_{1}^{n}}=\sum_{\begin{array}{c}
\mathbf{m}\in\mathbb{N}_{1}^{n}\\
\Sigma\left(\mathbf{m}\right)<L
\end{array}}+\sum_{\begin{array}{c}
\mathbf{m}\in\mathbb{N}_{1}^{n}\\
\Sigma\left(\mathbf{m}\right)\geq L
\end{array}}\label{eq:m sum decomposition using L}
\end{equation}
The part of (\ref{eq:second step of final part of comparison lemma proof})
corresponding to the right-most sum in (\ref{eq:m sum decomposition using L})
is: 
\begin{equation}
\sum_{\vec{\mathbf{v}}\in B^{n}}\sum_{\begin{array}{c}
\mathbf{m}\in\mathbb{N}_{1}^{n}\\
\Sigma\left(\mathbf{m}\right)\geq L
\end{array}}\left|\mathbb{P}\left(\vec{a}^{\left(n\right)}\left(\mathcal{N}\right)=\mathbf{m}\vec{\mathbf{v}}\right)-\frac{p-1}{p}\frac{\left(1-\frac{1}{p}\right)^{\Sigma\left(\mathbf{m}\right)}}{\left(p-1\right)^{2n}}\right|\label{eq:ready for break}
\end{equation}
Here, the event $\vec{a}^{\left(n\right)}\left(\mathcal{N}\right)=\mathbf{m}\vec{\mathbf{v}}$
is strictly contained in the event $\vec{\#}^{\left(n\right)}\left(\mathcal{N}\right)=\Sigma\left(\mathbf{m}\vec{\mathbf{v}}\right)$,
as the latter is the image of the former under the $\Sigma$ map,
which is surjective, but injective. As such, by the triangle inequality,
(\ref{eq:ready for break}) is bounded from above by: 
\begin{align}
 & \overbrace{\sum_{\vec{\mathbf{v}}\in B^{n*}}\sum_{\begin{array}{c}
\mathbf{m}\in\mathbb{N}_{1}^{n}\\
\Sigma\left(\mathbf{m}\right)\geq L
\end{array}}\mathbb{P}\left(\vec{\#}^{\left(n\right)}\left(\mathcal{N}\right)=\Sigma\left(\mathbf{m}\vec{\mathbf{v}}\right)\right)}^{\textrm{I}}\label{eq:break into I and II}\\
 & +\frac{p-1}{p}\underbrace{\sum_{\vec{\mathbf{v}}\in B^{n*}}\sum_{\begin{array}{c}
\mathbf{m}\in\mathbb{N}_{1}^{n}\\
\Sigma\left(\mathbf{m}\right)\geq L
\end{array}}\frac{\left(1-\frac{1}{p}\right)^{\Sigma\left(\mathbf{m}\right)}}{\left(p-1\right)^{2n}}}_{\textrm{II}}\nonumber 
\end{align}

For (II), we have: 
\begin{align*}
\textrm{II} & =\sum_{k=L}^{\infty}\sum_{\vec{\mathbf{v}}\in B^{*n}}\sum_{\begin{array}{c}
\mathbf{m}\in\mathbb{N}_{1}^{n}\\
\Sigma\left(\mathbf{m}\right)=k
\end{array}}\frac{\left(1-\frac{1}{p}\right)^{k}}{\left(p-1\right)^{2n}}\\
\left(\left|B^{n*}\right|=p\left(p-1\right)^{n}\right); & =p\sum_{k=L}^{\infty}\sum_{\begin{array}{c}
\mathbf{m}\in\mathbb{N}_{1}^{n}\\
\Sigma\left(\mathbf{m}\right)=k
\end{array}}\frac{\left(1-\frac{1}{p}\right)^{k}}{\left(p-1\right)^{n}}\\
 & =\frac{p}{\left(p-1\right)^{n}}\sum_{k=L}^{\infty}\binom{k-1}{n-1}\left(1-\frac{1}{p}\right)^{k}\\
\left(\textrm{as }L,n\rightarrow+\infty\right); & =O\left(\left(1-\frac{1}{p}\right)^{L-n}\right)
\end{align*}
As for (I), observe that we can write: 
\begin{equation}
\textrm{I}=\sum_{\begin{array}{c}
\mathbf{w}\in\mathbb{N}_{1}^{p}\\
\Sigma\left(\mathbf{w}\right)\geq L
\end{array}}\mathbb{P}\left(\vec{\#}^{\left(n\right)}\left(\mathcal{N}\right)=\mathbf{w}\right)=\sum_{k=L}^{\infty}\sum_{\begin{array}{c}
\mathbf{w}\in\mathbb{N}_{1}^{p}\\
\Sigma\left(\mathbf{w}\right)=k
\end{array}}\mathbb{P}\left(\vec{\#}^{\left(n\right)}\left(\mathcal{N}\right)=\mathbf{w}\right)
\end{equation}
This follows from the fact that, as $\mathbf{m}$ and $\vec{\mathbf{v}}$
vary, $\Sigma\left(\mathbf{m}\vec{\mathbf{v}}\right)$ will take values
in the set of $p$-tuples of non-negative integers, the sum of the
entries of which will be equal to $\Sigma\left(\mathbf{m}\right)$.
Now, given $\Sigma\left(\mathbf{w}\right)=k$, observe that there
must be at least one $j\in\left\{ 0,\ldots,p-1\right\} $ for which
$\#_{j}^{\left(n\right)}\left(\mathcal{N}\right)\geq\left\lfloor k/p\right\rfloor $.
As such: 
\begin{align*}
\textrm{I} & =\sum_{k=L}^{\infty}\sum_{\begin{array}{c}
\mathbf{w}\in\mathbb{N}_{1}^{p}\\
\Sigma\left(\mathbf{w}\right)=k
\end{array}}\mathbb{P}\left(\vec{\#}^{\left(n\right)}\left(\mathcal{N}\right)=\mathbf{w}\right)\leq\sum_{k=L}^{\infty}\sum_{\begin{array}{c}
\mathbf{w}\in\mathbb{N}_{1}^{p}\\
\Sigma\left(\mathbf{w}\right)=k
\end{array}}\max_{0\leq j<p}\mathbb{P}\left(\#_{j}^{\left(n\right)}\left(\mathcal{N}\right)\geq\left\lfloor k/p\right\rfloor \right)\\
\left(\mathbf{Prop.}\textrm{ }\mathbf{\ref{prop:tail bound}}\right); & \ll\sum_{k=L}^{\infty}\sum_{\begin{array}{c}
\mathbf{w}\in\mathbb{N}_{1}^{p}\\
\Sigma\left(\mathbf{w}\right)=k
\end{array}}p^{-c\left\lfloor k/p\right\rfloor }\\
 & =\sum_{k=L}^{\infty}\binom{k-1}{n-1}p^{-c\left\lfloor k/p\right\rfloor }\\
\left(\textrm{as }L,n\rightarrow+\infty\right); & =O\left(\left(\frac{p}{c\ln p}\right)^{n}p^{-cL/p}\right)
\end{align*}
Choosing: 
\begin{equation}
L\gg\max\left\{ 2,\frac{p}{c}\right\} n\label{eq:scaling for comparison lemma}
\end{equation}
then gives: 
\begin{align}
O\left(\left(1-\frac{1}{p}\right)^{L-n}\right) & =O\left(\left(1-\frac{1}{p}\right)^{c^{\prime}n}\right)\\
O\left(\left(\frac{p}{c\ln p}\right)^{n}p^{-cL/p}\right) & =O\left(p^{-c^{\prime\prime}n}\right)
\end{align}
and hence, that both terms are of $O\left(p^{-c^{\prime\prime\prime}n}\right)$
for some constant $c^{\prime\prime\prime}>0$.

Using these findings, along with (\ref{eq:break into I and II}),
and (\ref{eq:second step of final part of comparison lemma proof}),
we can write: 
\begin{align}
d_{\textrm{TV}}\left(\vec{a}^{\left(n\right)}\left(\mathcal{N}\right),\mathbf{Run}\left(p\right)^{n}\right) & \leq\sum_{\vec{\mathbf{v}}\in B^{n*}}\sum_{\begin{array}{c}
\mathbf{m}\in\mathbb{N}_{1}^{n}\\
\Sigma\left(\mathbf{m}\right)<L
\end{array}}\left|\mathbb{P}\left(\vec{a}^{\left(n\right)}\left(\mathcal{N}\right)=\mathbf{m}\vec{\mathbf{v}}\right)-\frac{p-1}{p}\frac{\left(1-\frac{1}{p}\right)^{\Sigma\left(\mathbf{m}\right)}}{\left(p-1\right)^{2n}}\right|\label{eq:almost done with geom comparison proof}\\
 & +O\left(p^{-c^{\prime\prime\prime}n}\right)\nonumber 
\end{align}
where the Big-Oh is for the sum with $\Sigma\left(\mathbf{m}\right)\geq L$.

Next, we get another instance of equation (\ref{eq:rough}) from the
proof of \textbf{Proposition \ref{prop:tail bound}}: 
\begin{equation}
H^{\circ N_{n}}\left(\mathcal{N}\right)=\Pi_{N_{n}}u_{\mathbf{j}_{N_{n}}\left(\mathcal{N}\right)}\mathcal{N}+\sum_{m=0}^{N_{n}-1}\Pi_{m}u_{\mathbf{j}_{m}\left(\mathcal{N}\right)}c_{j_{m}\left(\mathcal{N}\right)}\label{eq:rough-1}
\end{equation}
where: 
\begin{equation}
N_{n}=\sum_{m=1}^{n}\nu_{m}\left(\mathcal{N}\right)
\end{equation}
Proceeding by the same arguments, we then find that $\mathcal{N}$
must be confined to a single coset of $\Lambda^{N_{n}+1}$. Since:
\begin{equation}
N_{n}=\sum_{j=0}^{p-1}\#_{j}^{\left(n\right)}\left(\mathcal{N}\right)
\end{equation}
where $\#_{j}^{\left(n\right)}\left(\mathcal{N}\right)$ is the total
number of times the $j$th branch of $H$ was applied in the first
$n$ runs of $\mathcal{N}$, the event $\vec{a}^{\left(n\right)}\left(\mathcal{N}\right)=\mathbf{m}\vec{\mathbf{v}}$
tells us that: 
\begin{equation}
N_{n}=\sum_{j=0}^{p-1}\#_{j}^{\left(n\right)}\left(\mathcal{N}\right)=\Sigma\left(\mathbf{m}\right)
\end{equation}
As such, (\ref{eq:rough-1}) forces $\mathcal{N}$ to be confined
to a single coset of $\Lambda^{\Sigma\left(\mathbf{m}\right)+1}$.
Since $\Sigma\left(\mathbf{m}\right)<L$, by choosing $n^{\prime}$
large enough, the uniformity hypothesis (\ref{eq:uniformity hypothesis})
tells us that the probability of $\mathcal{N}$ occupying a single
coset of $\Lambda^{\Sigma\left(\mathbf{m}\right)+1}$ is $p^{-\Sigma\left(\mathbf{m}\right)-1}+O\left(p^{-n^{\prime}}\right)$,
where the error term reflects the fact that $\mathcal{N}$ is only
approximately uniform; the $O\left(p^{-n^{\prime}}\right)$ comes
from the approximate uniformity hypothesis and the triangle inequality.

With this, choosing $n$, $L$, and $n^{\prime}$ so that: 
\begin{equation}
n\max\left\{ 2,\frac{p}{c}\right\} \ll L<n^{\prime}
\end{equation}
the right-hand side of (\ref{eq:almost done with geom comparison proof})
is bounded from above by: 
\begin{align*}
 & \leq\sum_{\vec{\mathbf{v}}\in B^{*n}}\sum_{\begin{array}{c}
\mathbf{m}\in\mathbb{N}_{1}^{n}\\
\Sigma\left(\mathbf{m}\right)<L
\end{array}}\left|\mathbb{P}\left(\vec{a}^{\left(n\right)}\left(\mathcal{N}\right)=\mathbf{m}\vec{\mathbf{v}}\right)-\frac{p-1}{p}\frac{\left(1-\frac{1}{p}\right)^{\Sigma\left(\mathbf{m}\right)}}{\left(p-1\right)^{2n}}\right|+O\left(p^{-c^{\prime\prime\prime}n}\right)\\
\left(L<n^{\prime}\right); & \ll\sum_{\vec{\mathbf{v}}\in B^{*n}}\sum_{\begin{array}{c}
\mathbf{m}\in\mathbb{N}_{1}^{n}\\
\Sigma\left(\mathbf{m}\right)<n^{\prime}
\end{array}}\left(p^{-\Sigma\left(\mathbf{m}\right)-1}+O\left(p^{-n^{\prime}}\right)+\frac{p-1}{p}\frac{\left(1-\frac{1}{p}\right)^{\Sigma\left(\mathbf{m}\right)}}{\left(p-1\right)^{2n}}\right)+O\left(p^{-c^{\prime\prime\prime}n}\right)\\
 & \ll\binom{n^{\prime}-1}{n}p^{n-n^{\prime}}+\binom{n^{\prime}-1}{n}O\left(p^{n-n^{\prime}}\right)+\binom{n^{\prime}-1}{n}\frac{\left(1-\frac{1}{p}\right)^{n^{\prime}}}{\left(p-1\right)^{n}}+O\left(p^{-c^{\prime\prime\prime}n}\right)\\
 & =\binom{n^{\prime}-1}{n}O\left(p^{n-n^{\prime}}\right)+\binom{n^{\prime}-1}{n}\frac{\left(1-\frac{1}{p}\right)^{n^{\prime}}}{\left(p-1\right)^{n}}+O\left(p^{-c^{\prime\prime\prime}n}\right)\\
\left(n^{\prime}\geq n\max\left\{ 2,\frac{p}{c}\right\} \right); & =O\left(p^{-c_{1}n}\right)
\end{align*}
for some constant $c_{1}$.

Q.E.D.

\section{\label{sec:Proof-of-Lemma-1}Proof of Lemma \ref{lem:stabilization of first passage} }

To begin, recall the heuristic (\ref{eq:almost boundedness heuristic}):
\begin{equation}
\left\Vert H^{\circ n}\left(\mathcal{N}\right)\right\Vert _{\textrm{Mink}}\ll\left(\prod_{k=0}^{p-1}\left\Vert r_{k}\right\Vert _{\textrm{Mink}}\right)^{pn}\left\Vert \mathcal{N}\right\Vert _{\textrm{Mink}}
\end{equation}
which we obtained in \textbf{Section \ref{sec:Outline-of-the}}, assuming
that the applications of $H$'s branches behaved like sums of i.i.d.
geometric random variables. To that end, let $\mathcal{N}_{y}$ be
a random variable with distribution: 
\begin{equation}
\textrm{Log}\left(\left\{ z\in\mathcal{O}_{K}:y\leq\left\Vert z\right\Vert _{\textrm{Mink}}\leq y^{\alpha}\right\} \right)
\end{equation}
and assume MH1, MH2, and MH3 are satisfied.

\subsection{\label{subsec:Proof-of-Passage-1}Proof of Passage Time Estimate
(Equation \ref{eq:passing time estimate})}

\vphantom{} 
\begin{prop}
\label{prop:equidistribution proposition}Let everything be as in
\textbf{Assumption \ref{assu:initial parameter assumptions}}.\textbf{
}Then: 
\begin{equation}
d_{\textrm{TV}}\left(\mathcal{N}_{y}\textrm{ mod }\Lambda^{cn_{x}},\mathbf{Unif}\left(\mathcal{O}_{K}/\Lambda^{cn_{x}}\right)\right)\ll p^{-cn_{x}}\label{eq:confirmation of approx.  unif.  dist.  hypothesis}
\end{equation}
for some absolute constant of proportionality and some constant $c>0$
depending only on $H$. 
\end{prop}
Proof: For brevity, write: 
\begin{align}
\mathcal{U} & \overset{\textrm{def}}{=}\mathbf{Unif}\left(\mathcal{O}_{K}/\Lambda^{cn_{x}}\right)\\
G & \overset{\textrm{def}}{=}\mathcal{O}_{K}/\Lambda^{cn_{x}}\\
d_{\textrm{TV}} & \overset{\textrm{def}}{=}d_{\textrm{TV}}\left(\mathcal{N}_{y}\textrm{ mod }\Lambda^{cn_{x}},\mathbf{Unif}\left(\mathcal{O}_{K}/\Lambda^{cn_{x}}\right)\right)
\end{align}
Then: 
\begin{align}
d_{\textrm{TV}} & =\sum_{\lambda\in G}\left|\mathbb{P}\left(\mathcal{N}_{y}\overset{\Lambda^{cn_{x}}}{\equiv}\lambda\right)-\frac{1}{\left|G\right|}\right|
\end{align}
where the probability is of $\mathcal{N}_{y}$ being congruent to
$\lambda$ mod $\Lambda^{cn_{x}}$.

Here, we use the Fourier-analytic approach. Let $\hat{G}$ denote
the Pontryagin dual of $G$, written additively as the set of frequencies
$\xi$ (with $\xi$ being a group homomorphism from $G$ to $\mathbb{R}/\mathbb{Z}$),
so that every unitary character $\chi:G\rightarrow\mathbb{T}$ is
of the form: 
\begin{equation}
\chi\left(\lambda\right)=e^{2\pi i\xi\left(\lambda\right)}
\end{equation}
for some unique $\xi\in\hat{G}$, where we write $\xi=0$ to denote
the trivial frequency that sends all elements of $G$ to $0\in\mathbb{R}/\mathbb{Z}$.

Now, let $\widehat{\mathcal{N}_{y}},\hat{\mathcal{U}}:\hat{G}\rightarrow\mathbb{C}$
denote the Fourier transforms of the probability mass functions of
$\mathcal{N}_{y}$ and $\mathcal{U}$, respectively, with: 
\begin{align}
\widehat{\mathcal{N}_{y}}\left(\xi\right) & \overset{\textrm{def}}{=}\frac{1}{\left|G\right|}\sum_{\lambda\in G}\mathbb{P}\left(\mathcal{N}_{y}\overset{\Lambda^{cn_{x}}}{\equiv}\lambda\right)e^{-2\pi i\xi\left(\lambda\right)}\\
\hat{\mathcal{U}}\left(\xi\right) & \overset{\textrm{def}}{=}\frac{1}{\left|G\right|}\sum_{\lambda\in G}\mathbb{P}\left(\hat{\mathcal{U}}\overset{\Lambda^{cn_{x}}}{\equiv}\lambda\right)e^{-2\pi i\xi\left(\lambda\right)}
\end{align}
Since $G=\mathcal{O}_{K}/\Lambda^{cn_{x}}$ is a finite abelian group,
it is isomorphic to its own dual, the finiteness of which guarantees
that these two Fourier transforms are fully invertible, with: 
\begin{align}
\mathbb{P}\left(\mathcal{N}_{y}\overset{\Lambda^{cn_{x}}}{\equiv}\lambda\right) & =\sum_{\xi\in\hat{G}}\widehat{\mathcal{N}_{y}}\left(\xi\right)e^{2\pi i\xi\left(\lambda\right)}\\
\mathbb{P}\left(\mathcal{U}\overset{\Lambda^{cn_{x}}}{\equiv}\lambda\right) & =\sum_{\xi\in\hat{G}}\hat{\mathcal{U}}\left(\xi\right)e^{2\pi i\xi\left(\lambda\right)}
\end{align}
for all $\lambda\in\mathcal{O}_{K}/\Lambda^{cn_{x}}$. In particular,
since: 
\begin{equation}
\widehat{\mathcal{N}_{y}}\left(0\right)=\frac{1}{\left|G\right|}\underbrace{\sum_{\lambda\in G}\mathbb{P}\left(\mathcal{N}_{y}\overset{\Lambda^{cn_{x}}}{\equiv}\lambda\right)\overbrace{e^{-2\pi i0\left(\lambda\right)}}^{1}}_{1}=\frac{1}{\left|G\right|}
\end{equation}
we have, by Fourier inversion and Cauchy-Schwarz: 
\begin{align*}
d_{\textrm{TV}} & =\sum_{\lambda\in G}\left|\underbrace{\left(\frac{1}{\left|G\right|}+\sum_{\xi\neq0}\widehat{\mathcal{N}_{y}}\left(\xi\right)e^{2\pi i\xi\left(\lambda\right)}\right)}_{\mathbb{P}\left(\mathcal{N}_{y}\overset{\Lambda^{cn_{x}}}{\equiv}\lambda\right)}-\frac{1}{\left|G\right|}\right|=\sum_{\lambda\in G}\left|\sum_{\xi\neq0}\widehat{\mathcal{N}_{y}}\left(\xi\right)e^{2\pi i\xi\left(\lambda\right)}\right|\\
\left(\textrm{CSI}\right); & \leq\sqrt{\left|\hat{G}\right|}\sum_{\lambda\in G}\sqrt{\sum_{\xi\neq0}\left|\widehat{\mathcal{N}_{y}}\left(\xi\right)\right|^{2}}
\end{align*}
Since $\left|\hat{G}\right|=\left|G\right|$, we get: 
\begin{equation}
d_{\textrm{TV}}\leq\sqrt{\left|G\right|}\sqrt{\sum_{\xi\neq0}\left|\widehat{\mathcal{N}_{y}}\left(\xi\right)\right|^{2}}\label{eq:CSI bound on d_TV}
\end{equation}

Next, using the probability mass function of $\mathcal{N}_{y}$, we
can express $\mathcal{N}_{y}$'s Fourier transform as: 
\begin{equation}
\widehat{\mathcal{N}_{y}}\left(\xi\right)=\frac{1}{\left|G\right|}\sum_{\lambda\in G}\frac{\sum_{z\in\left(\lambda+\Lambda^{cn_{x}}\right)\cap U_{y}}\left\Vert z\right\Vert _{\textrm{Mink}}^{-d}}{\sum_{z\in U_{y}}\left\Vert z\right\Vert _{\textrm{Mink}}^{-d}}e^{-2\pi i\xi\left(\lambda\right)}=\frac{1}{\left|G\right|}\frac{\sum_{z\in U_{y}}\left\Vert z\right\Vert _{\textrm{Mink}}^{-d}e^{-2\pi i\xi\left(z\right)}}{\sum_{z\in U_{y}}\left\Vert z\right\Vert _{\textrm{Mink}}^{-d}}\label{eq:N_y-hat simplified}
\end{equation}
This re-write follows from the fact that for any function $f:G\times\mathcal{O}_{K}\rightarrow\mathbb{C}$:
\begin{equation}
\sum_{\lambda\in G}\sum_{z\in\left(\lambda+\Lambda^{cn_{x}}\right)\cap U_{y}}f\left(\lambda,z\right)=\sum_{\lambda\in G:z\in\lambda+\Lambda^{cn_{x}}}\sum_{z\in U_{y}}f\left(\lambda,z\right)=\sum_{z\in U_{y}}f\left(\left[z\right]_{\Lambda^{cn_{x}}},z\right)=\sum_{z\in U_{y}}f\left(z,z\right)
\end{equation}
As such, (\ref{eq:CSI bound on d_TV}) becomes: 
\begin{equation}
d_{\textrm{TV}}\leq\sqrt{\left|G\right|}\sqrt{\sum_{\xi\neq0}\left|\widehat{\mathcal{N}_{y}}\left(\xi\right)\right|^{2}}
\end{equation}

Here, we have: 
\begin{equation}
\sum_{\xi\neq0}\left|\widehat{\mathcal{N}_{y}}\left(\xi\right)\right|^{2}=\sum_{\xi\neq0}\left|\frac{\sum_{z\in U_{y}}\left\Vert z\right\Vert _{\textrm{Mink}}^{-d}e^{-2\pi i\xi\left(z\right)}}{\sum_{z\in U_{y}}\left\Vert z\right\Vert _{\textrm{Mink}}^{-d}}\right|^{2}=\frac{\sum_{\xi\neq0}\left|\sum_{z\in U_{y}}\left\Vert z\right\Vert _{\textrm{Mink}}^{-d}e^{-2\pi i\xi\left(z\right)}\right|^{2}}{\left(\sum_{z\in U_{y}}\left\Vert z\right\Vert _{\textrm{Mink}}^{-d}\right)^{2}}
\end{equation}
Expanding out the numerator and summing over $\xi$, we get: 
\begin{align}
\sum_{\xi\neq0}\left|\sum_{z\in U_{y}}\left\Vert z\right\Vert _{\textrm{Mink}}^{-d}e^{-2\pi i\xi\left(z\right)}\right|^{2} & \leq\sum_{z\in U_{y}}\sum_{w\in U_{y}}\left\Vert z\right\Vert _{\textrm{Mink}}^{-d}\left\Vert w\right\Vert _{\textrm{Mink}}^{-d}\sum_{\xi\neq0}e^{2\pi i\xi\left(w-z\right)}\\
 & =\left|G\right|\sum_{z\in U_{y}}\left\Vert z\right\Vert _{\textrm{Mink}}^{-2d}-\left(\sum_{z\in U_{y}}\left\Vert z\right\Vert _{\textrm{Mink}}^{-d}\right)^{2}
\end{align}
where the passage to the second line follows by \textbf{character
orthogonality}: 
\begin{equation}
\sum_{\xi\neq0}e^{2\pi i\xi\left(w-z\right)}=\left|\hat{G}\right|\left[w\overset{\Lambda^{cn_{x}}}{\equiv}z\right]-1,\textrm{ }\forall z,w\in G
\end{equation}
where we are using Iverson bracket notation: 
\begin{equation}
\left[w\overset{\Lambda^{cn_{x}}}{\equiv}z\right]=\begin{cases}
1 & \textrm{if }w\overset{\Lambda^{cn_{x}}}{\equiv}z\\
0 & \textrm{else}
\end{cases}
\end{equation}
As such 
\begin{equation}
\sum_{\xi\neq0}\left|\widehat{\mathcal{N}_{y}}\left(\xi\right)\right|^{2}=\frac{\left|G\right|\sum_{z\in U_{y}}\left\Vert z\right\Vert _{\textrm{Mink}}^{-2d}-\left(\sum_{z\in U_{y}}\left\Vert z\right\Vert _{\textrm{Mink}}^{-d}\right)^{2}}{\left(\sum_{z\in U_{y}}\left\Vert z\right\Vert _{\textrm{Mink}}^{-d}\right)^{2}}\leq\left|G\right|\frac{\sum_{z\in U_{y}}\left\Vert z\right\Vert _{\textrm{Mink}}^{-2d}}{\left(\sum_{z\in U_{y}}\left\Vert z\right\Vert _{\textrm{Mink}}^{-d}\right)^{2}}
\end{equation}
and so: 
\begin{equation}
d_{\textrm{TV}}\leq\sqrt{\left|G\right|}\sqrt{\sum_{\xi\neq0}\left|\widehat{\mathcal{N}_{y}}\left(\xi\right)\right|^{2}}\leq\left|G\right|\frac{\sum_{z\in U_{y}}\left\Vert z\right\Vert _{\textrm{Mink}}^{-2d}}{\left(\sum_{z\in U_{y}}\left\Vert z\right\Vert _{\textrm{Mink}}^{-d}\right)^{2}}
\end{equation}
Here, \textbf{Lemma \ref{lem:lattice sum asymptotic}} gives us: 
\begin{equation}
\sum_{z\in U_{y}}\left\Vert z\right\Vert _{\textrm{Mink}}^{-d}\sim C_{K}\left(\alpha-1\right)d\ln y\textrm{ as }y\rightarrow\infty\label{eq:lattice sum lower bound}
\end{equation}
and: 
\begin{equation}
\sum_{z\in U_{y}}\left|z\right|_{\infty}^{-2d}\sim C_{K}\left(y^{-d}-y^{-\alpha d}\right)\textrm{ as }y\rightarrow\infty
\end{equation}
for some constant $C_{K}$ depending only on $K$, and hence, that:
\begin{equation}
d_{\textrm{TV}}\leq\left|G\right|\frac{\sum_{z\in U_{y}}\left\Vert z\right\Vert _{\textrm{Mink}}^{-2d}}{\left(\sum_{z\in U_{y}}\left\Vert z\right\Vert _{\textrm{Mink}}^{-d}\right)^{2}}\ll\frac{C_{K}\frac{\left(y^{-d}-y^{-\alpha d}\right)}{d}}{\left(C_{K}\left(\alpha-1\right)d\ln y\right)^{2}}\ll y^{-d}
\end{equation}
which tends to $0$ as $y\rightarrow\infty$.

For $y=x^{\alpha^{k}}$ for $k\in\left\{ 1,2\right\} $, we have:
\begin{equation}
d_{\textrm{TV}}\ll y^{-\alpha d}=x^{-\alpha^{2}d}
\end{equation}
Since $\alpha\in\left(1,2\right)$, and since\textbf{ Assumption \ref{assu:initial parameter assumptions}}
gives: 
\begin{equation}
n_{x}\asymp\frac{\ln x}{\ln A_{0}}
\end{equation}
we have: 
\begin{equation}
x^{-\alpha^{2}d}\asymp A_{0}^{-\alpha^{2}dn_{x}}\ll p^{-cn_{x}}
\end{equation}
for some constant $c>0$ depending only on $H$, $\alpha$, and $A_{0}$.

Q.E.D.

\vphantom{}

\textbf{Proof of (\ref{eq:passing time estimate})}:

Since (\ref{eq:confirmation of approx.  unif.  dist.  hypothesis})
is satisfied, we can apply \textbf{Lemma \ref{lem:Geometric comparison}}
to conclude that: 
\begin{equation}
d_{\textrm{TV}}\left(\vec{a}^{\left(n_{x}\right)}\left(\mathcal{N}_{y}\right),\mathbf{Run}\left(p\right)^{n_{x}}\right)\ll p^{-cn_{x}}\label{eq:geometric comparison conclusion}
\end{equation}
Next, recall that the $n$-path: 
\begin{equation}
\vec{a}^{\left(n\right)}\left(\mathcal{N}_{y}\right)=\left(\nu_{1}\left(\mathcal{N}_{y}\right)\mathbf{e}_{\textrm{Stay}_{1}\left(\mathcal{N}_{y}\right)},\ldots,\nu_{n}\left(\mathcal{N}_{y}\right)\mathbf{e}_{\textrm{Stay}_{n}\left(\mathcal{N}_{y}\right)}\right)
\end{equation}
is the length of the first $n$ runs of branch applications specified
by $H$ to $\mathcal{N}_{y}$. So, given $i\in\left\{ 0,\ldots,p-1\right\} $,
let us write: 
\begin{equation}
\vec{a}_{i}^{\left(n\right)}\left(\mathcal{N}_{y}\right)\overset{\textrm{def}}{=}\left(\nu_{1}\left(\mathcal{N}_{y}\right)\mathbf{e}_{\textrm{Stay}_{1}\left(\mathcal{N}_{y}\right)}\cdot\mathbf{e}_{i},\ldots,\nu_{n}\left(\mathcal{N}_{y}\right)\mathbf{e}_{\textrm{Stay}_{n}\left(\mathcal{N}_{y}\right)}\cdot\mathbf{e}_{i}\right)\in\mathbb{Z}^{n}\label{eq: nth run vector with ith component}
\end{equation}
Note that a given entry of $\vec{a}_{i}^{\left(n\right)}\left(\mathcal{N}_{y}\right)$
is non-zero if and only if the run associated to the corresponding
entry of $\vec{a}^{\left(n\right)}\left(\mathcal{N}_{y}\right)$ is
a run of $i$s. Indeed, recall the identity: 
\begin{equation}
\#_{i}^{\left(n\right)}\left(\mathcal{N}_{y}\right)=\Sigma\left(\vec{a}_{i}^{\left(n\right)}\left(\mathcal{N}_{y}\right)\right)
\end{equation}
Using this, we have that:

\begin{equation}
\mathbb{P}\left(\Sigma\left(\vec{a}^{\left(n_{x}\right)}\left(\mathcal{N}_{y}\right)\right)\leq C_{0}n_{x}\right)\leq\mathbb{P}\left(\Sigma\left(\mathbf{Run}\left(p\right)^{n_{x}}\right)\leq C_{0}n_{x}\right)+O\left(p^{-cn_{x}}\right)
\end{equation}
Here, note that: 
\begin{equation}
\Sigma\left(\vec{a}^{\left(n_{x}\right)}\left(\mathcal{N}_{y}\right)\right)=\vec{\#}^{\left(n_{x}\right)}\left(\mathcal{N}_{y}\right)
\end{equation}
and that $\Sigma\left(\mathbf{Run}\left(p\right)^{n_{x}}\right)$
is a random variable taking values in $\mathbb{Z}^{n_{x}}$. As such:
\begin{align*}
\mathbb{P}\left(\Sigma\left(\mathbf{Run}\left(p\right)^{n_{x}}\right)\leq C_{0}n_{x}\right) & \leq\mathbb{P}\left(\Sigma\left(\Sigma\left(\mathbf{Run}\left(p\right)^{n_{x}}\right)\right)\leq C_{0}n_{x}\right)\\
\left(\mathbf{Lemma\textrm{ }}\mathbf{\ref{lem:Chernoff Bound}}\right); & \ll p^{-cn_{x}}
\end{align*}
where we used our Chernoff bound to deal with the sum of the entries
in $\Sigma\left(\mathbf{Run}\left(p\right)^{n_{x}}\right)$. Because
of this, we have: 
\begin{equation}
\mathbb{P}\left(\vec{\#}^{\left(n_{x}\right)}\left(\mathcal{N}_{y}\right)\leq C_{0}n_{x}\right)\ll p^{-cn_{x}}\ll x^{-c}\label{eq:number vector x to the -c bound}
\end{equation}
where we allow $c$ to vary as we move from left to right along the
inequalities in (\ref{eq:number vector x to the -c bound}). From
this, it follows that:

\begin{equation}
\max_{0\leq j<p}\mathbb{P}\left(\#_{j}^{\left(n_{x}\right)}\left(\mathcal{N}_{y}\right)\leq C_{0}n_{x}\right)\ll p^{-cn_{x}}\ll x^{-c}\label{eq:head bound}
\end{equation}

Next, letting $N_{n_{x}}$ denote $\sum_{k=1}^{n_{x}}\nu_{k}\left(\mathcal{N}_{y}\right)$
(this is the sum of the lengths of the runs), we have: 
\begin{equation}
H^{\circ N_{n_{x}}}\left(\mathcal{N}_{y}\right)=r_{\mathbf{j}_{N_{n_{x}}}\left(\mathcal{N}_{y}\right)}\mathcal{N}_{y}+\sum_{m=0}^{N_{n_{x}}-1}r_{\mathbf{j}_{m}\left(\mathcal{N}_{y}\right)}c_{j_{m}\left(\mathcal{N}_{y}\right)}\label{eq:nu_n_0 affine H of N}
\end{equation}
Here: 
\begin{equation}
\left\Vert r_{\mathbf{j}_{N_{n_{x}}}\left(\mathcal{N}_{y}\right)}z\right\Vert _{\textrm{Mink}}\leq\left(\prod_{k=0}^{p-1}\left\Vert r_{k}\right\Vert _{\textrm{Mink}}^{\#_{k}^{\left(n_{x}\right)}\left(\mathcal{N}_{y}\right)}\right)\left\Vert z\right\Vert _{\textrm{Mink}},\textrm{ }\forall z\in K
\end{equation}
To deal with the $m$-sum on the right-hand side of (\ref{eq:nu_n_0 affine H of N}),
observe that: 
\begin{equation}
\left\Vert r_{\mathbf{j}_{m}\left(\mathcal{N}_{y}\right)}\right\Vert _{\textrm{Mink}}\leq\prod_{j:\left\Vert r_{j}\right\Vert _{\textrm{Mink}}>1}\left\Vert r_{j}\right\Vert _{\textrm{Mink}}^{\#_{j}\left(\mathbf{j}_{m}\left(\mathcal{N}_{y}\right)\right)}\leq R_{H}^{\sum_{j=0}^{p-1}\left[\left\Vert r_{j}\right\Vert _{\textrm{Mink}}>1\right]\#_{j}\left(\mathbf{j}_{m}\left(\mathcal{N}_{y}\right)\right)}
\end{equation}
where the exponent uses Iverson bracket notation. Since $\#_{j}$
counts the numbers of $j$s in $\mathbf{j}_{m}\left(\mathcal{N}_{y}\right)$,
we have that for all $m\leq N_{n_{x}}$, $\#_{j}\left(\mathbf{j}_{m}\left(\mathcal{N}_{y}\right)\right)\leq\#_{j}^{\left(n_{x}\right)}\left(\mathcal{N}_{y}\right)$,
and so:
\begin{align*}
\left\Vert \sum_{m=0}^{N_{n_{x}}-1}r_{\mathbf{j}_{m}\left(\mathcal{N}_{y}\right)}c_{j_{m+1}\left(\mathcal{N}_{y}\right)}\right\Vert _{\textrm{Mink}} & \ll\sum_{m=0}^{N_{n_{x}}-1}R_{H}^{\sum_{j=0}^{p-1}\left[\left\Vert r_{j}\right\Vert _{\textrm{Mink}}>1\right]\#_{j}\left(\mathbf{j}_{m}\left(\mathcal{N}_{y}\right)\right)}\\
 & \leq\sum_{m=0}^{N_{n_{x}}-1}R_{H}^{\sum_{j=0}^{p-1}\left[\left\Vert r_{j}\right\Vert _{\textrm{Mink}}>1\right]\#_{j}^{\left(n_{x}\right)}\left(\mathcal{N}_{y}\right)}\\
 & =N_{n_{x}}R_{H}^{\sum_{j=0}^{p-1}\left[\left\Vert r_{j}\right\Vert _{\textrm{Mink}}>1\right]\#_{j}^{\left(n_{x}\right)}\left(\mathcal{N}_{y}\right)}
\end{align*}
Hence: 
\begin{equation}
\left\Vert H^{\circ N_{n_{x}}}\left(\mathcal{N}_{y}\right)\right\Vert _{\textrm{Mink}}\ll\left(\prod_{k=0}^{p-1}\left\Vert r_{k}\right\Vert _{\textrm{Mink}}^{\#_{k}^{\left(n_{x}\right)}\left(\mathcal{N}_{y}\right)}\right)\left\Vert \mathcal{N}_{y}\right\Vert _{\textrm{Mink}}+N_{n_{x}}R_{H}^{\sum_{j=0}^{p-1}\left[\left\Vert r_{j}\right\Vert _{\textrm{Mink}}>1\right]\#_{j}^{\left(n_{x}\right)}\left(\mathcal{N}_{y}\right)}\label{eq:H nu_n absolute value estimate}
\end{equation}
Now, suppose that: 
\begin{align}
\#_{j}^{\left(n_{x}\right)}\left(\mathcal{N}_{y}\right) & \geq C_{0}n_{x},\textrm{ }\forall j:\left\Vert r_{j}\right\Vert _{\textrm{Mink}}<1\label{eq:contracting branch count}\\
\#_{j}^{\left(n_{x}\right)}\left(\mathcal{N}_{y}\right) & \leq C_{0}n_{x},\textrm{ }\forall j:\left\Vert r_{j}\right\Vert _{\textrm{Mink}}\geq1\label{eq:expanding branch count}
\end{align}
Then, we have: 
\begin{equation}
\prod_{k=0}^{p-1}\left\Vert r_{k}\right\Vert _{\textrm{Mink}}^{\#_{k}^{\left(n_{x}\right)}\left(\mathcal{N}_{y}\right)}\leq\prod_{k=0}^{p-1}\left\Vert r_{k}\right\Vert _{\textrm{Mink}}^{C_{0}n_{x}}=\rho_{H}^{C_{0}n_{x}}
\end{equation}
and: 
\begin{equation}
R_{H}^{\sum_{j=0}^{p-1}\left[\left\Vert r_{j}\right\Vert _{\textrm{Mink}}>1\right]\#_{j}^{\left(n_{x}\right)}\left(\mathcal{N}_{y}\right)}\leq R_{H}^{pC_{0}n_{x}}
\end{equation}
and so: 
\begin{equation}
\left\Vert H^{\circ N_{n_{x}}}\left(\mathcal{N}_{y}\right)\right\Vert _{\textrm{Mink}}\ll\rho_{H}^{C_{0}n_{x}}\left\Vert \mathcal{N}_{y}\right\Vert _{\textrm{Mink}}+N_{n_{x}}R_{H}^{pC_{0}n_{x}}\label{eq:Need N_n_0 bound}
\end{equation}
Here, recall that $N_{n_{x}}=N_{n_{x}}\left(\mathcal{N}_{y}\right)$
is a random variable: it is the total number of iterates of $H$ needed
in order for $\mathcal{N}_{y}$ to go through $n_{x}$ distinct runs
of consecutively applied branches of $H$. In fact, we have: 
\begin{equation}
N_{n_{x}}\left(\mathcal{N}_{y}\right)=\Sigma\left(\vec{\#}^{\left(n_{x}\right)}\left(\mathcal{N}_{y}\right)\right)=\Sigma\left(\Sigma\left(\vec{a}^{\left(n_{x}\right)}\left(\mathcal{N}_{y}\right)\right)\right)
\end{equation}
Consequently, by (\ref{eq:geometric comparison conclusion}), we have
that: 
\begin{equation}
p^{-cn_{x}}\gg d_{\textrm{TV}}\left(N_{n_{x}}\left(\mathcal{N}_{y}\right),\Sigma\left(\Sigma\left(\mathbf{Run}\left(p\right)^{n_{x}}\right)\right)\right)=d_{\textrm{TV}}\left(N_{n_{x}}\left(\mathcal{N}_{y}\right),\Sigma\left(\mathbf{Geom}\left(p\right)^{n_{x}}\right)\right)
\end{equation}
Our Chernoff Bound (\textbf{Lemma \ref{lem:Chernoff Bound}}) gives
us: 
\begin{equation}
\mathbb{P}\left(\Sigma\left(\mathbf{Geom}\left(p\right)^{n_{x}}\right)\geq R_{H}^{pC_{0}n_{x}}\right)\ll p^{-c^{\prime}n_{x}}\textrm{ as }n_{x}\rightarrow\infty
\end{equation}
for some constant $c^{\prime}>0$, and hence: 
\begin{align*}
\mathbb{P}\left(\nu_{n_{x}}\left(\mathcal{N}_{y}\right)\geq R_{H}^{pC_{0}n_{x}}\right) & \leq d_{\textrm{TV}}\left(N_{n_{x}}\left(\mathcal{N}_{y}\right),\Sigma\left(\Sigma\left(\mathbf{Run}\left(p\right)^{n_{x}}\right)\right)\right)+\mathbb{P}\left(\Sigma\left(\mathbf{Geom}\left(p\right)^{n_{x}}\right)\geq R_{H}^{pC_{0}n_{x}}\right)\\
 & \ll p^{-cn_{x}}+p^{-c^{\prime}n_{x}}\\
 & =O\left(p^{-c^{\prime\prime}n_{x}}\right)
\end{align*}
for some constant $c^{\prime\prime\prime}>0$.

Thus, as $x\rightarrow\infty$, $\nu_{n_{x}}\left(\mathcal{N}_{y}\right)$
will be $O\left(R_{H}^{pC_{0}n_{x}}\right)$, and so (\ref{eq:Need N_n_0 bound})
becomes: 
\begin{align*}
\left\Vert H^{\circ N_{n_{x}}}\left(\mathcal{N}_{y}\right)\right\Vert _{\textrm{Mink}} & \ll\rho_{H}^{C_{0}n_{x}}\left\Vert \mathcal{N}_{y}\right\Vert _{\textrm{Mink}}+N_{n_{x}}R_{H}^{pC_{0}n_{x}}=\rho_{H}^{C_{0}n_{x}}\left\Vert \mathcal{N}_{y}\right\Vert _{\textrm{Mink}}+O\left(R_{H}^{pC_{0}n_{x}}\right)\\
\left(\left|\mathcal{N}_{y}\right|_{\infty}\leq y^{\alpha}\leq x^{\alpha^{3}}\right); & \leq\rho_{H}^{C_{0}n_{x}}x^{\alpha^{3}}+O\left(R_{H}^{pC_{0}n_{x}}\right)\\
\left(\textrm{use }(\ref{eq:def of n_0})\right); & \ll x^{\alpha^{3}-C_{0}\frac{\ln\rho_{H}^{-1}}{\ln A_{0}}}+x^{C_{0}\frac{p\ln R_{H}}{\ln A_{0}}}
\end{align*}

By (\ref{eq:C_0 A_0 inequality}), we have: 
\begin{equation}
\max\left\{ \alpha^{3}-C_{0}\frac{\ln\rho_{H}^{-1}}{\ln A_{0}},C_{0}\frac{p\ln R_{H}}{\ln A_{0}}\right\} <1
\end{equation}
and hence: 
\begin{equation}
\left\Vert H^{\circ N_{n_{x}}}\left(\mathcal{N}_{y}\right)\right\Vert _{\textrm{Mink}}\ll x\label{eq:almost boundedness}
\end{equation}
The head bound (\ref{eq:head bound}) then implies that (\ref{eq:contracting branch count})
occurs as $x\rightarrow\infty$, while our tail bound from \textbf{Proposition
\ref{prop:tail bound}} implies that (\ref{eq:expanding branch count})
occurs as $x\rightarrow\infty$ and hence, that: 
\begin{equation}
T_{x}\left(\mathcal{N}_{y}\right)\leq N_{n_{x}}<\infty
\end{equation}
for all sufficiently large $x$. This completes the proof of \textbf{(\ref{eq:passing time estimate})}
from \textbf{Lemma \ref{lem:stabilization of first passage}}.

Q.E.D.

\vphantom{}

\subsection{\label{subsec:Proof-of-Passage}Proof of Passage Location Comparison
(Equation \ref{eq:passage location comparison})}

In order to show that the rest of \textbf{Lemma \ref{lem:stabilization of first passage}}---that
is, \textbf{(\ref{eq:passage location comparison})}---follows from
the Fourier-analytic properties of $X_{H}$, we begin by establishing
an approximate formula for $\mathbb{P}\left(\textrm{Pass}_{x}\left(\mathcal{N}_{y}\right)\in S\right)$,
to generalize \textbf{Proposition 5.2} from Tao's paper.

\vphantom{} 
\begin{lem}[Approximate First Passage Formula]
\label{lem:approximate formula for passage location}Let: 
\begin{equation}
U^{x}\overset{\textrm{def}}{=}\left\{ z\in\mathcal{O}_{K}:1\leq\left\Vert z\right\Vert _{\textrm{Mink}}\leq x\right\} 
\end{equation}
let $S\subseteq U^{x}$, and let $y$ be either $x^{\alpha}$ or $x^{\alpha^{2}}$.
Then, there is a constant $c>0$ so that: 
\begin{equation}
\mathbb{P}\left(\textrm{Pass}_{x}\left(\mathcal{N}_{y}\right)\in S\right)=\sum_{L\in I_{y}}\sum_{\vec{\mathbf{V}}\in\mathcal{A}^{\left(L-m_{x}\right)}}\sum_{z\in S^{\prime}}\mathbb{P}\left(H_{\vec{\mathbf{V}}}\left(\mathcal{N}_{y}\right)=z\right)+O\left(\ln^{-c}x\right)\textrm{ as }x\rightarrow\infty\label{eq:approximate formula for passage location}
\end{equation}
where $S^{\prime}$ is the set of all $z\in\mathcal{O}_{K}$ so that:

i. $T_{x}\left(z\right)=m_{x}$.

ii. $\textrm{Pass}_{x}\left(z\right)\in S$ and: 
\begin{equation}
e^{-\ln^{0.7}x}\rho_{H}^{-m_{x}}x\leq\left\Vert z\right\Vert _{\textrm{Mink}}\leq e^{\ln^{0.7}x}\rho_{H}^{-m_{x}}x\label{eq:restrictions on z}
\end{equation}
\end{lem}
\begin{rem}
The inequalities in \textbf{Assumption \ref{assu:initial parameter assumptions}}
ensure that $I_{y}\subset\left[m_{x},n_{x}\right]$ for all sufficiently
large $x$, regardless of whether $y=x^{\alpha}$ or $y=x^{\alpha^{2}}$. 
\end{rem}
Proof: Fix $S$. From (\ref{eq:geometric comparison conclusion})
and \textbf{Lemma \ref{lem:Chernoff Bound}}, we have: 
\begin{equation}
\max_{0\leq j<p}\mathbb{P}\left(\left|\#_{j}^{\left(n\right)}\left(\mathcal{N}\right)-pn\right|\geq\ln^{0.6}x\right)\ll e^{-c\ln^{0.2}x}\label{eq:consequence of chernoff bound}
\end{equation}
for all $n\in\left\{ 0,\ldots,n_{x}\right\} $. Note that: 
\begin{align*}
\mathbb{P}\left(\vec{a}^{\left(n_{x}\right)}\left(\mathcal{N}_{y}\right)\notin\mathcal{A}^{\left(n_{x}\right)}\right) & \leq\sum_{n=0}^{n_{x}}\mathbb{P}\left(\vec{a}^{\left(n\right)}\left(\mathcal{N}_{y}\right)\notin\mathcal{A}^{\left(n\right)}\right)\\
 & \leq\sum_{j=0}^{p-1}\sum_{n=0}^{n_{x}}\mathbb{P}\left(\left|\#_{j}^{\left(n\right)}\left(\mathcal{N}\right)-pn\right|\geq\ln^{0.6}x\right)\\
\left(\textrm{Use }(\ref{eq:consequence of chernoff bound})\right); & \ll\sum_{j=0}^{p-1}\sum_{n=0}^{n_{x}}e^{-c\ln^{0.2}x}\\
 & =pn_{x}e^{-c\ln^{0.2}x}\\
 & \ll e^{-c\ln^{0.2}x}\ln x
\end{align*}
Replacing $x$ with $\exp\left(x^{10}\right)$, the upper bound becomes
$x^{10}e^{-cx^{2}}$, which is $O\left(e^{-x}\right)$ as $x\rightarrow\infty$,
giving us:

\begin{equation}
\mathbb{P}\left(\vec{a}^{\left(n_{x}\right)}\left(\mathcal{N}_{y}\right)\notin\mathcal{A}^{\left(n_{x}\right)}\right)\ll e^{-c\ln^{0.2}x}\ln x\ll e^{-\ln^{0.1}x}\ll\ln^{-c^{\prime}}x\label{eq:Need 3}
\end{equation}
for some constant $c^{\prime}>0$. Thus, writing: 
\begin{align*}
\mathbb{P}\left(\textrm{Pass}_{x}\left(\mathcal{N}_{y}\right)\in S\right) & =\mathbb{P}\left(\textrm{Pass}_{x}\left(\mathcal{N}_{y}\right)\in S\textrm{ \& }\vec{a}^{\left(n_{x}\right)}\left(\mathcal{N}_{y}\right)\in\mathcal{A}^{\left(n_{x}\right)}\right)\\
 & +\underbrace{\mathbb{P}\left(\textrm{Pass}_{x}\left(\mathcal{N}_{y}\right)\in S\textrm{ \& }\vec{a}^{\left(n_{x}\right)}\left(\mathcal{N}_{y}\right)\notin\mathcal{A}^{\left(n_{x}\right)}\right)}_{\leq\mathbb{P}\left(\vec{a}^{\left(n_{x}\right)}\left(\mathcal{N}_{y}\right)\notin\mathcal{A}^{\left(n_{x}\right)}\right)}
\end{align*}
we get:

\begin{equation}
\mathbb{P}\left(\textrm{Pass}_{x}\left(\mathcal{N}_{y}\right)\in S\right)=\mathbb{P}\left(\textrm{Pass}_{x}\left(\mathcal{N}_{y}\right)\in S\textrm{ \& }\vec{a}^{\left(n_{x}\right)}\left(\mathcal{N}_{y}\right)\in\mathcal{A}^{\left(n_{x}\right)}\right)+O\left(\ln^{-c}x\right)\label{eq:Need 4}
\end{equation}
as $x\rightarrow\infty$, for some constant $c>0$.

For brevity, let $\vec{\mathbf{V}}^{\left(n\right)}$ denote $\vec{a}^{\left(n\right)}\left(\mathcal{N}_{y}\right)$,
so that $\Sigma\left(\vec{\mathbf{V}}^{\left(n\right)}\right)=\Sigma\left(\vec{a}^{\left(n\right)}\left(\mathcal{N}_{y}\right)\right)=\vec{\#}^{\left(n\right)}\left(\mathcal{N}_{y}\right)$.
Now, suppose the event $\vec{\mathbf{V}}^{\left(n_{x}\right)}\in\mathcal{A}^{\left(n_{x}\right)}$
occurs, so that (\ref{eq:approximate uniform behavior}) happens.
Here, note that $\Sigma\left(\vec{\#}^{\left(n\right)}\left(\mathcal{N}_{y}\right)\right)=\Sigma\left(\Sigma\left(\vec{\mathbf{V}}^{\left(n\right)}\right)\right)$
implies: 
\begin{equation}
\textrm{Syr}_{H}^{\circ n}\left(\mathcal{N}_{y}\right)=H^{\circ\Sigma\left(\Sigma\left(\vec{\mathbf{V}}^{\left(n\right)}\right)\right)}\left(\mathcal{N}_{y}\right)
\end{equation}
Thus, we have:
\begin{equation}
\textrm{Syr}_{H}^{\circ n}\left(\mathcal{N}_{y}\right)=M_{H}\left(\vec{\mathbf{V}}^{\left(n\right)}\right)\mathcal{N}_{y}+X_{H}\left(\vec{\mathbf{V}}^{\left(n\right)}\right),\textrm{ }\forall n\in\left\{ 0,\ldots,n_{x}\right\} 
\end{equation}
Applying the estimates from \textbf{Proposition \ref{prop:elementary properties of X_H of v-arrow}}
yields: 
\begin{align*}
\left\Vert \textrm{Syr}_{H}^{\circ n}\left(\mathcal{N}_{y}\right)\right\Vert _{\textrm{Mink}} & \ll\overline{\rho}_{H}^{\frac{1}{p}\ln^{0.6}x}\rho_{H}^{n}\left\Vert \mathcal{N}_{y}\right\Vert _{\textrm{Mink}}+\left(R_{H}\overline{\rho}_{H}^{1/p}\right)^{\ln^{0.6}x}R_{H}^{pn}\\
 & =\left(1+R_{H}^{\ln^{0.6}x}\left(R_{H}^{p}\rho_{H}^{-1}\right)^{n}\left\Vert \mathcal{N}_{y}\right\Vert _{\textrm{Mink}}^{-1}\right)\overline{\rho}_{H}^{\frac{1}{p}\ln^{0.6}x}\rho_{H}^{n}\left\Vert \mathcal{N}_{y}\right\Vert _{\textrm{Mink}}
\end{align*}
Here, since $R_{H}^{p}\rho_{H}^{-1}>1$, $n\leq n_{x}$ (where $n_{x}\asymp\frac{\ln x}{\ln A_{0}}$
as $x\rightarrow\infty$), and $\left\Vert \mathcal{N}_{y}\right\Vert _{\textrm{Mink}}\geq y\geq x^{\alpha}$,
we have: 
\begin{equation}
\left\Vert \textrm{Syr}_{H}^{\circ n}\left(\mathcal{N}_{y}\right)\right\Vert _{\textrm{Mink}}\ll\left(1+R_{H}^{\ln^{0.6}x}x^{\frac{\ln\left(R_{H}^{p}\rho_{H}^{-1}\right)}{\ln A_{0}}-\alpha}\right)\overline{\rho}_{H}^{\frac{1}{p}\ln^{0.6}x}\rho_{H}^{n}\left\Vert \mathcal{N}_{y}\right\Vert _{\textrm{Mink}}\label{eq:syr estimate (rough)}
\end{equation}
By (\ref{eq:alpha A_0 inequality}), the exponent of $x$ in (\ref{eq:syr estimate (rough)})
is negative, which gives: 
\begin{equation}
\left\Vert \textrm{Syr}_{H}^{\circ n}\left(\mathcal{N}_{y}\right)\right\Vert _{\textrm{Mink}}\ll\left(1+O\left(x^{-c}\right)\right)\overline{\rho}_{H}^{\frac{1}{p}\ln^{0.6}x}\rho_{H}^{n}\left\Vert \mathcal{N}_{y}\right\Vert _{\textrm{Mink}}
\end{equation}
and so: 
\begin{equation}
\left\Vert \textrm{Syr}_{H}^{\circ n}\left(\mathcal{N}_{y}\right)\right\Vert _{\textrm{Mink}}\ll e^{\frac{1}{p}\left(\ln\overline{\rho}_{H}\right)\ln^{0.6}x}\rho_{H}^{n}\left\Vert \mathcal{N}_{y}\right\Vert _{\textrm{Mink}},\textrm{ }\forall n\in\left\{ 0,\ldots,n_{x}\right\} \label{eq:Tao 5.14 analogue}
\end{equation}
Equation (\ref{eq:Tao 5.14 analogue}), note, is the analogue of Equation
5.14 from Tao's paper.

Recalling that $T_{x}\left(\mathcal{N}_{y}\right)$ is the smallest
$n\geq1$ for which $\left\Vert \textrm{Syr}_{H}^{\circ n}\left(\mathcal{N}_{y}\right)\right\Vert _{\textrm{Mink}}\leq x$,
(\ref{eq:Tao 5.14 analogue}) gives the estimates: 
\begin{equation}
e^{\frac{1}{p}\left(\ln\overline{\rho}_{H}\right)\ln^{0.6}x}\rho_{H}^{T_{x}\left(\mathcal{N}_{y}\right)}\left\Vert \mathcal{N}_{y}\right\Vert _{\textrm{Mink}}\leq x<e^{\frac{1}{p}\left(\ln\overline{\rho}_{H}\right)\ln^{0.6}x}\rho_{H}^{T_{x}\left(\mathcal{N}_{y}\right)-1}\left\Vert \mathcal{N}_{y}\right\Vert _{\textrm{Mink}}\label{eq:upper and lower bound on N_y}
\end{equation}
and hence: 
\begin{equation}
\frac{1}{p}\frac{\ln\overline{\rho}_{H}}{\ln\rho_{H}^{-1}}\ln^{0.6}x+\frac{\ln\frac{\left\Vert \mathcal{N}_{y}\right\Vert _{\textrm{Mink}}}{x}}{\ln\rho_{H}^{-1}}\leq T_{x}\left(\mathcal{N}_{y}\right)<\frac{1}{p}\frac{\ln\overline{\rho}_{H}}{\ln\rho_{H}^{-1}}\ln^{0.6}x+\frac{\ln\frac{\left\Vert \mathcal{N}_{y}\right\Vert _{\textrm{Mink}}}{x}}{\ln\rho_{H}^{-1}}+1\label{eq:Tao 5.15 analogue}
\end{equation}
i.e. 
\begin{equation}
T_{x}\left(\mathcal{N}_{y}\right)=\frac{\ln\frac{\left\Vert \mathcal{N}_{y}\right\Vert _{\textrm{Mink}}}{x}}{\ln\rho_{H}^{-1}}+O\left(\ln^{0.6}x\right)\label{eq:Tao 5.15 analogue restated}
\end{equation}
(This is the analogue of equation 5.15 from Tao's paper.) Here, we
note that: 
\begin{equation}
0<\frac{1}{p}\frac{\ln\overline{\rho}_{H}}{\ln\rho_{H}^{-1}}\ln^{0.6}x+\frac{\ln\frac{\left\Vert \mathcal{N}_{y}\right\Vert _{\textrm{Mink}}}{x}}{\ln\rho_{H}^{-1}}+1<n_{x}<\frac{\ln x}{\ln A_{0}}
\end{equation}
occurs for all sufficiently large $x$, so that the far-right side
of (\ref{eq:Tao 5.15 analogue}) is between $0$ and $n_{x}$ for
all sufficiently large $x$. Since $n_{x}$ is larger than every element
of $I_{y}$, this gives $T_{x}\left(\mathcal{N}_{y}\right)\in\left[0,n_{x}\right]$.
Consequently, given the event $\vec{\mathbf{v}}^{\left(n_{x}\right)}\in\mathcal{A}^{\left(n_{x}\right)}$,
in order for $T_{x}\left(\mathcal{N}_{y}\right)\in I_{y}$, it suffices
that the lower and upper bounds of (\ref{eq:Tao 5.15 analogue}) are
greater than and less than the upper and lower endpoints of $I_{y}$,
respectively; these being: 
\begin{equation}
\frac{\ln\left(y/x\right)}{\ln A_{2}}+\ln^{0.8}x<\frac{1}{p}\frac{\ln\overline{\rho}_{H}}{\ln\rho_{H}^{-1}}\ln^{0.6}x+\frac{\ln\frac{\left|\mathcal{N}_{y}\right|_{\infty}}{x}}{\ln\rho_{H}^{-1}}
\end{equation}
and: 
\begin{equation}
\frac{1}{p}\frac{\ln\overline{\rho}_{H}}{\ln\rho_{H}^{-1}}\ln^{0.6}x+\frac{\ln\frac{\left|\mathcal{N}_{y}\right|_{\infty}}{x}}{\ln\rho_{H}^{-1}}+1<\frac{\ln\left(y^{\alpha}/x\right)}{\ln\rho_{H}^{-1}}-\ln^{0.8}x
\end{equation}
Solving these for an inequality for $\left\Vert \mathcal{N}_{y}\right\Vert _{\textrm{Mink}}$
yields: 
\begin{equation}
ye^{\left(\ln\rho_{H}^{-1}\right)\ln^{0.8}x-\frac{1}{p}\left(\ln\overline{\rho}_{H}\right)\ln^{0.6}x}<\left\Vert \mathcal{N}_{y}\right\Vert _{\textrm{Mink}}<\frac{y^{\alpha}}{\rho_{H}}e^{-\left(\ln\rho_{H}^{-1}\right)\ln^{0.8}x-\frac{1}{p}\left(\ln\overline{\rho}_{H}\right)\ln^{0.6}x}
\end{equation}
From this, given the event $\vec{\mathbf{V}}^{\left(n_{x}\right)}\in\mathcal{A}^{\left(n_{x}\right)}$,
we see that $T_{x}\left(\mathcal{N}_{y}\right)\in I_{y}$ will occur
whenever: 
\begin{equation}
\left\Vert \mathcal{N}_{y}\right\Vert _{\textrm{Mink}}\in I_{x}^{\prime}
\end{equation}
where: 
\begin{equation}
I_{x}^{\prime}\overset{\textrm{def}}{=}\left(yx^{c_{1}\left(x\right)},\rho_{H}^{-p}y^{\alpha}e^{-c_{2}\left(x\right)}\right)
\end{equation}
and where: 
\begin{align}
c_{1}\left(x\right) & \overset{\textrm{def}}{=}\frac{p\ln\rho_{H}^{-1}}{\ln^{0.2}x}-\frac{1}{p}\frac{\ln\overline{\rho}_{H}}{\ln^{0.4}x}\\
c_{2}\left(x\right) & \overset{\textrm{def}}{=}\left(p\ln\rho_{H}^{-1}\right)\ln^{0.8}x+\frac{1}{p}\left(\ln\overline{\rho}_{H}\right)\ln^{0.6}x
\end{align}
Thus: 
\begin{equation}
\mathbb{P}\left(T_{x}\left(\mathcal{N}_{y}\right)\in I_{y}\right)\leq\mathbb{P}\left(\left\Vert \mathcal{N}_{y}\right\Vert _{\textrm{Mink}}\in I_{x}^{\prime}\right)\label{eq:first passage time estimate}
\end{equation}

\begin{claim}
We have: 
\begin{equation}
\mathbb{P}\left(\left\Vert \mathcal{N}_{y}\right\Vert _{\textrm{Mink}}\in I_{x}^{\prime}\right)=1-O\left(\ln^{-c}x\right)\textrm{ as }x\rightarrow\infty\label{eq:first passage time estimate refined}
\end{equation}
for some constant $c>0$.

Proof of Claim: In order for the upper limit of $I_{x}^{\prime}$
to be less than the upper limit of $I_{y}$, we need: 
\begin{equation}
\rho_{H}^{-p}y^{\alpha}e^{-c_{2}\left(x\right)}<y^{\alpha}
\end{equation}
which is true for all sufficiently large $x$, since $c_{2}\left(x\right)\rightarrow\infty$
as $x\rightarrow\infty$. On the other hand, $c_{1}\left(x\right)>0$
for all sufficiently large $x$, giving us: 
\begin{equation}
yx^{c_{1}\left(x\right)}>y
\end{equation}
for all sufficiently large $x$. As such, $I_{x}^{\prime}\cap\left[y,y^{\alpha}\right]=I_{x}^{\prime}$
as $x\rightarrow\infty$, and so, using Iverson bracket notation:
\begin{equation}
\mathbb{P}\left(\left\Vert \mathcal{N}_{y}\right\Vert _{\textrm{Mink}}\in I_{x}^{\prime}\right)=\frac{\sum_{z\in\mathcal{O}_{K}}\overbrace{\left[\left\Vert z\right\Vert _{\textrm{Mink}}\in I_{x}^{\prime}\cap\left[y,y^{\alpha}\right]\right]}^{\left[\left\Vert z\right\Vert _{\textrm{Mink}}\in I_{x}^{\prime}\right]}\left\Vert z\right\Vert _{\textrm{Mink}}^{-d}}{\sum_{z\in\mathcal{O}_{K}}\left[\left\Vert z\right\Vert _{\textrm{Mink}}\in\left[y,y^{\alpha}\right]\right]\left\Vert z\right\Vert _{\textrm{Mink}}^{-d}}
\end{equation}
Here, \textbf{Lemma \ref{lem:lattice sum asymptotic}} yields: 
\begin{equation}
\sum_{z\in\mathcal{O}_{K}}\frac{\left[\left\Vert z\right\Vert _{\textrm{Mink}}\in\left[y,y^{\alpha}\right]\right]}{\left\Vert z\right\Vert _{\textrm{Mink}}^{d}}\sim C_{K}\left(\alpha-1\right)d\ln y\textrm{ as }x\rightarrow\infty
\end{equation}
and: 
\begin{align*}
\sum_{z\in\mathcal{O}_{K}}\frac{\left[\left\Vert z\right\Vert _{\textrm{Mink}}\in I_{x}^{\prime}\right]}{\left\Vert z\right\Vert _{\textrm{Mink}}^{d}} & =\sum_{z\in\mathcal{O}_{K}}\frac{\left[yx^{c_{1}\left(x\right)}\leq\left\Vert z\right\Vert _{\textrm{Mink}}\leq\rho_{H}^{-1/p}y^{\alpha}e^{-c_{2}\left(x\right)}\right]}{\left\Vert z\right\Vert _{\textrm{Mink}}^{d}}\\
 & \sim C_{K}d\left(\ln\left(\rho_{H}^{-p}y^{\alpha}e^{-c_{2}\left(x\right)}\right)-\ln\left(yx^{c_{1}\left(x\right)}\right)\right)\\
 & \sim C_{K}d\left(\ln\rho_{H}^{-p}+\left(\alpha-1\right)\ln y-c_{2}\left(x\right)-c_{1}\left(x\right)\ln x\right)
\end{align*}
So, as $x\rightarrow\infty$: 
\begin{align*}
\mathbb{P}\left(\left\Vert \mathcal{N}_{y}\right\Vert _{\textrm{Mink}}\in I_{x}^{\prime}\right) & \sim\frac{C_{K}d\left(\ln\rho_{H}^{-p}+\left(\alpha-1\right)\ln y-c_{2}\left(x\right)-c_{1}\left(x\right)\ln x\right)}{C_{K}\left(\alpha-1\right)d\ln y}\\
 & =\frac{\ln\rho_{H}^{-p}}{\left(\alpha-1\right)\ln y}+1-\frac{c_{2}\left(x\right)}{\left(\alpha-1\right)\ln y}-\frac{c_{1}\left(x\right)\ln x}{\left(\alpha-1\right)\ln y}
\end{align*}
Since $c_{1}\left(x\right)=O\left(\ln^{-0.2}x\right)$ and $c_{2}\left(x\right)=O\left(\ln^{0.8}x\right)$
as $x\rightarrow\infty$, while $\ln y=O\left(\ln x\right)$ as $x\rightarrow\infty$,
the $c_{1}\left(x\right)$ and $c_{2}\left(x\right)$ terms in this
asymptotic vanish as $x\rightarrow\infty$, giving us: 
\begin{equation}
\mathbb{P}\left(\left\Vert \mathcal{N}_{y}\right\Vert _{\textrm{Mink}}\in I_{x}^{\prime}\right)\sim\frac{p\ln\rho_{H}^{-1}}{\left(\alpha-1\right)\ln y}+1=1+O\left(\ln^{-c}x\right)
\end{equation}
for $c=1$.

This proves the claim. 
\end{claim}
By (\ref{eq:first passage time estimate}), we see that with probability
$1-O\left(\ln^{-c}x\right)$, there is some $n\in I_{y}$ so that
$T_{x}\left(\mathcal{N}_{y}\right)=n$. Because of this, (\ref{eq:Need 4})
above becomes: 
\begin{equation}
\mathbb{P}\left(\textrm{Pass}_{x}\left(\mathcal{N}_{y}\right)\in S\right)=\sum_{n\in I_{y}}\mathbb{P}\left(\left\{ \textrm{Pass}_{x}\left(\mathcal{N}_{y}\right)\in S\right\} \cap\left\{ \vec{\mathbf{V}}^{\left(n_{x}\right)}\in\mathcal{A}^{\left(n_{x}\right)}\right\} \cap\left\{ T_{x}\left(\mathcal{N}_{y}\right)=n\right\} \right)+O\left(\ln^{-c}x\right)\label{eq:Need 4 refined}
\end{equation}
With the crossing time written out exactly, we are now in a position
to relate it to the crossing time of the other iterates. 
\begin{claim}
For each $n\in I_{y}$, the event: 
\begin{equation}
\left\{ \textrm{Pass}_{x}\left(\mathcal{N}_{y}\right)\in S\right\} \cap\left\{ \vec{\mathbf{V}}^{\left(n_{x}\right)}\in\mathcal{A}^{\left(n_{x}\right)}\right\} \cap\left\{ T_{x}\left(\mathcal{N}_{y}\right)=n\right\} \label{eq:first event}
\end{equation}
is equal to the event: 
\begin{equation}
B_{n,y}\overset{\textrm{def}}{=}\left\{ T_{x}\left(\textrm{Syr}_{H}^{\circ n-m_{x}}\left(\mathcal{N}_{y}\right)\right)=m_{x}\right\} \cap\left\{ \textrm{Pass}_{x}\left(\textrm{Syr}_{H}^{\circ n-m_{x}}\left(\mathcal{N}_{y}\right)\right)\in S\right\} \cap\left\{ \vec{\mathbf{V}}^{\left(n_{x}\right)}\in\mathcal{A}^{\left(n_{x}\right)}\right\} \label{eq:def of B_n,y}
\end{equation}

Proof of Claim: Fix $z\in\mathcal{O}_{K}$, and suppose $\mathcal{N}_{y}=z$.

I. Suppose the event $\mathcal{N}_{y}=z$ is contained in (\ref{eq:first event}).
If $T_{x}\left(z\right)=n$, then $\left\Vert \textrm{Syr}_{H}^{\circ n}\left(z\right)\right\Vert _{\textrm{Mink}}\leq\left\Vert z\right\Vert _{\textrm{Mink}}$,
with $\left\Vert \textrm{Syr}_{H}^{\circ m}\left(z\right)\right\Vert _{\textrm{Mink}}>\left\Vert z\right\Vert _{\textrm{Mink}}$
for all $m<n$. This implies that $T_{x}\left(\textrm{Syr}_{H}^{\circ n-m_{x}}\left(z\right)\right)=m_{x}$.
Similarly, since $H$ eventually iterates $\textrm{Syr}_{H}^{\circ n-m_{x}}\left(z\right)$
to $\textrm{Syr}_{H}^{\circ n}\left(z\right)\in S$, it follows that
$\textrm{Pass}_{x}\left(\textrm{Syr}_{H}^{\circ n-m_{x}}\left(\mathcal{N}_{y}\right)\right)\in S$.
Indeed, this first passage cannot occur anywhere else, as that would
imply the first passage \emph{time }was strictly less than $n$. Finally,
since $\vec{\mathbf{V}}^{\left(n_{x}\right)}$ was given to be in
$\mathcal{A}^{\left(n_{x}\right)}$ by (\ref{eq:first event}), we
see that the event $\mathcal{N}_{y}=z$ is contained in $B_{n,y}$.
Since $z$ was arbitrary, we conclude that all events in (\ref{eq:first event})
are contained in $B_{n,y}$.

II. Suppose the event $\mathcal{N}_{y}=z$ is contained in $B_{n,y}$.
Since $T_{x}\left(\textrm{Syr}_{H}^{\circ n-m_{x}}\left(z\right)\right)=m_{x}$
and $\vec{\mathbf{V}}^{\left(n_{x}\right)}\in\mathcal{A}^{\left(n_{x}\right)}$,
we have: 
\begin{equation}
\left\Vert \textrm{Syr}_{H}^{\circ n}\left(z\right)\right\Vert _{\textrm{Mink}}\leq x<\left\Vert \textrm{Syr}_{H}^{\circ n-1}\left(z\right)\right\Vert _{\textrm{Mink}}
\end{equation}
which, by (\ref{eq:Tao 5.14 analogue}), forces: 
\begin{equation}
x<e^{\frac{1}{p}\left(\ln\overline{\rho}_{H}\right)\ln^{0.6}x}\rho_{H}^{n-1}\left\Vert z\right\Vert _{\textrm{Mink}}
\end{equation}
and so: 
\begin{equation}
n<\frac{\ln\frac{\left\Vert z\right\Vert _{\textrm{Mink}}}{x}}{\ln\rho_{H}^{-1}}+\frac{1}{p}\frac{\ln\overline{\rho}_{H}}{\ln\rho_{H}^{-1}}\ln^{0.6}x+1=\frac{\ln\frac{\left\Vert z\right\Vert _{\textrm{Mink}}}{x}}{\ln\rho_{H}^{-1}}+O\left(\ln^{0.6}x\right)
\end{equation}
By (\ref{eq:Tao 5.15 analogue restated}), we have: 
\begin{equation}
T_{x}\left(z\right)=\frac{\ln\frac{\left\Vert z\right\Vert _{\textrm{Mink}}}{x}}{\ln\rho_{H}^{-1}}+O\left(\ln^{0.6}x\right)
\end{equation}
and so, $T_{x}\left(z\right)\geq n$. Thus, $T_{x}\left(z\right)\geq n-m_{x}$,
and hence: 
\begin{equation}
T_{x}\left(z\right)=n-m_{x}+T_{x}\left(\textrm{Syr}_{H}^{\circ n-m_{x}}\left(z\right)\right)=n
\end{equation}
From this, it follows that $\textrm{Pass}_{x}\left(z\right)\in S$,
and thus, that the event $\mathcal{N}_{y}=z$ being in $B_{n,y}$
implies the event $\mathcal{N}_{y}=z$ is in (\ref{eq:first event}).
Since $z$ was arbitrary, we conclude $B_{n,y}$ is equivalent to
(\ref{eq:first event}). This proves the claim. 
\end{claim}
With this claim, (\ref{eq:first passage time estimate refined}) implies
that (\ref{eq:Need 4 refined}) can be written as: 
\begin{equation}
\mathbb{P}\left(\textrm{Pass}_{x}\left(\mathcal{N}_{y}\right)\in S\right)=\sum_{n\in I_{y}}\mathbb{P}\left(B_{n,y}\right)+O\left(\ln^{-c}x\right)
\end{equation}
Using the set $S^{\prime}$, we now have several implications. 
\begin{itemize}
\item If $B_{n,y}$ occurs, then, by (\ref{eq:Tao 5.14 analogue}) and (\ref{eq:Tao 5.15 analogue}),
we have: 
\begin{equation}
\left\Vert \textrm{Syr}_{H}^{\circ n-m_{x}}\left(\mathcal{N}_{y}\right)\right\Vert _{\textrm{Mink}}=e^{O\left(\ln^{0.6}x\right)}\rho_{H}^{T_{x}\left(\mathcal{N}_{y}\right)-m_{x}}\left\Vert \mathcal{N}_{y}\right\Vert _{\textrm{Mink}}=e^{O\left(\ln^{0.6}x\right)}\rho_{H}^{-m_{x}}x
\end{equation}
and hence, the event: 
\begin{equation}
\left(\textrm{Syr}_{H}^{\circ n-m_{x}}\left(\mathcal{N}_{y}\right)\in E^{\prime}\right)\cap\left(\vec{\mathbf{V}}^{\left(n_{x}\right)}\in\mathcal{A}^{\left(n_{x}\right)}\right)\label{eq:Our analogue of Tao 5.17}
\end{equation}
occurs. 
\item Conversely, if (\ref{eq:Our analogue of Tao 5.17}) holds, then from
(\ref{eq:Tao 5.14 analogue}), we have: 
\begin{equation}
\left\Vert \textrm{Syr}_{H}^{\circ n^{\prime}}\left(\mathcal{N}_{y}\right)\right\Vert _{\textrm{Mink}}=e^{O\left(\ln^{0.6}x\right)}\rho_{H}^{n-m_{x}-n^{\prime}}\left\Vert \textrm{Syr}_{H}^{\circ n-m_{x}}\left(\mathcal{N}_{y}\right)\right\Vert _{\textrm{Mink}}\geq e^{O\left(\ln^{0.6}x\right)}\left\Vert \textrm{Syr}_{H}^{\circ n-m_{x}}\left(\mathcal{N}_{y}\right)\right\Vert _{\textrm{Mink}}
\end{equation}
for all $n^{\prime}$ with $0\leq n^{\prime}\leq n-m_{x}$. Thus,
by (\ref{eq:restrictions on z}): 
\begin{equation}
\left\Vert \textrm{Syr}_{H}^{\circ n^{\prime}}\left(\mathcal{N}_{y}\right)\right\Vert _{\textrm{Mink}}>x
\end{equation}
for all $0\leq n^{\prime}\leq n-m_{x}$. We conclude that: 
\begin{equation}
\textrm{Pass}_{x}\left(\mathcal{N}_{y}\right)=n-m_{x}+\textrm{Pass}_{x}\left(\textrm{Syr}_{H}^{\circ n-m_{x}}\left(\mathcal{N}_{y}\right)\right)=n
\end{equation}
thanks to the definition of $S^{\prime}$, and hence also that: 
\begin{equation}
T_{x}\left(\mathcal{N}_{y}\right)=T_{x}\left(\textrm{Syr}_{H}^{\circ n-m_{x}}\left(\mathcal{N}_{y}\right)\right)\in E
\end{equation}
and thus, that the event $B_{n,y}$ occurs. 
\end{itemize}
From these, we conclude that: 
\begin{equation}
B_{n,y}=\left(\textrm{Syr}_{H}^{\circ n-m_{x}}\left(\mathcal{N}_{y}\right)\in E^{\prime}\right)\cap\left(\vec{\mathbf{V}}^{\left(n_{x}\right)}\in\mathcal{A}^{\left(n_{x}\right)}\right)
\end{equation}
Since the event $\vec{\mathbf{V}}^{\left(n_{x}\right)}\in\mathcal{A}^{\left(n_{x}\right)}$
is contained in the event $\vec{\mathbf{V}}^{\left(n-m_{x}\right)}\in\mathcal{A}^{\left(n-m_{x}\right)}$,
we conclude from (\ref{eq:Need 3}) that: 
\begin{align}
\mathbb{P}\left(\textrm{Pass}_{x}\left(\mathcal{N}_{y}\right)\in E\right) & =\sum_{n\in I_{y}}\mathbb{P}\left(\left(\textrm{Syr}_{H}^{\circ n-m_{x}}\left(\mathcal{N}_{y}\right)\in E^{\prime}\right)\cap\left(\vec{\mathbf{V}}^{\left(n-m_{x}\right)}\in\mathcal{A}^{\left(n-m_{x}\right)}\right)\right)\\
 & +O\left(\ln^{-c}x\right)\nonumber 
\end{align}
Finally, let $\vec{\mathbf{W}}$ be a tuple in $\mathcal{A}^{\left(n-m_{x}\right)}$
and let $z\in S^{\prime}$. Then, $\left(\textrm{Syr}_{H}^{\circ n-m_{x}}\left(\mathcal{N}_{y}\right)=z\right)\cap\left(\vec{\mathbf{V}}^{\left(n-m_{x}\right)}=\vec{\mathbf{W}}\right)$
occurs if and only if $H_{\vec{\mathbf{V}}}\left(\mathcal{N}_{y}\right)\in S^{\prime}$,
and we are done.

Q.E.D.

\vphantom{}Just like in Tao's paper, we can demonstrate \textbf{(\ref{eq:passage location comparison})}
by establishing that for any set $S$ with: 
\begin{equation}
S\subseteq\left\{ z\in\mathcal{O}_{K}:1\leq\left\Vert z\right\Vert _{\textrm{Mink}}\leq x\right\} 
\end{equation}
that there is a quantity $Q$ (where $Q$ is allowed to depend on
$x$, $\alpha$, and $S$, but is independent of whether $y$ is $x^{\alpha}$
or $x^{\alpha^{2}}$) so that: 
\begin{align}
\mathbb{P}\left(\textrm{Pass}_{x}\left(\mathcal{N}_{y}\right)\in S\right) & =\left(1+O\left(\ln^{-c}x\right)\right)Q+O\left(\ln^{-c}x\right)\label{eq:Q expression for passage location}
\end{align}
In order to get the $\left(1+O\left(\ln^{-c}x\right)\right)Q$ term,
we will rewrite \textbf{Proposition \ref{prop:approximate formula for passage,  archimedean case}}'s
bound for $\mathbb{P}\left(\textrm{Pass}_{x}\left(\mathcal{N}_{y}\right)\in S\right)$
in terms of $\mathbf{Syrac}_{H}$. At this point in his paper, Tao
re-writes what we expressed as: 
\begin{equation}
\sum_{n\in I_{y}}\sum_{\vec{\mathbf{V}}\in\mathcal{A}^{\left(n-m_{x}\right)}}\sum_{z\in S^{\prime}}\mathbb{P}\left(H_{\vec{\mathbf{V}}}\left(\mathcal{N}_{y}\right)=z\right)
\end{equation}
from the right-hand side of (\ref{eq:approximate formula for passage location})
in terms of the probability mass function of the Syracuse random variables.
In our language, this would mean considering the values of $X_{H}$.
But here, a major distinction arises between Tao's approach and our
own. It is worth discussing this in detail, as it our argument's greatest
point of departure from Tao's.

\vphantom{} 
\begin{example}
In Tao's case, the numen $X_{H}$ would be the function the first
author calls $\chi_{3}$ \cite{my_disseration,my_first_blog_paper},
treated here as a function from $\mathbb{Z}_{2}$ to $\mathbb{Z}_{3}$,
with the defining limit: 
\begin{equation}
\chi_{3}\left(\mathfrak{z}\right)\overset{\textrm{def}}{=}\lim_{n\rightarrow\infty}\chi_{3}\left(\left[\mathfrak{z}\right]_{2^{n}}\right)
\end{equation}
converging pointwise in $\mathbb{Z}_{3}$ for all $\mathfrak{z}\in\mathbb{Z}_{2}$,
where, recall, $\left[\mathfrak{z}\right]_{2^{n}}$ is the unique
integer in $\left\{ 0,\ldots,2^{n}-1\right\} $ which is congruent
to $\mathfrak{z}$ mod $2^{n}$. Recall also that Tao's $n$th Syracuse
Random Variable $\mathbf{Syrac}\left(\mathbb{Z}/3^{n}\mathbb{Z}\right)$
is precisely the random variable $\left[\chi_{3}\right]_{3^{n}}$
obtained by projecting the values of $\chi_{3}$ mod $3^{n}$: 
\begin{equation}
\mathbb{P}\left(\mathbf{Syrac}\left(\mathbb{Z}/3^{n}\mathbb{Z}\right)\overset{3^{n}}{\equiv}k\right)=\mathbb{P}\left(\chi_{3}\overset{3^{n}}{\equiv}k\right)=\textrm{meas}_{\textrm{Haar}}\left(\chi_{3}^{-1}\left(k+3^{n}\mathbb{Z}_{3}\right)\right)
\end{equation}
where the expression on the right is the $2$-adic Haar probability
measure of the pre-image of $k+3^{n}\mathbb{Z}_{3}$ under $\chi_{3}$.

Letting $K=\mathbb{Q}$, so that $\mathcal{N}_{y}$ takes values in
$\mathbb{Z}$, let $H$ be the Shortened Collatz map $H:\mathbb{Z}\rightarrow\mathbb{Z}$:
\begin{equation}
H\left(m\right)=\begin{cases}
\frac{m}{2} & \textrm{if }m\overset{2}{\equiv}0\\
\frac{3m+1}{2} & \textrm{if }m\overset{2}{\equiv}1
\end{cases}
\end{equation}
The congruence: 
\begin{equation}
M=F_{n-m_{x}}\left(\vec{a}\right)\mod3^{n-m_{x}}\label{eq:Tao 5.18}
\end{equation}
in equation (5.18) from Tao's paper arises when Tao chooses $\vec{a}\in\mathcal{A}^{\left(n-m_{x}\right)}$,
an integer $M$, and considers the event he denotes as $\textrm{Aff}_{\vec{a}}\left(\mathcal{N}_{y}\right)=M$.
In our language, $\textrm{Aff}_{\vec{a}}\left(\mathcal{N}_{y}\right)=M$
is the event: 
\begin{equation}
H_{\mathbf{j}}\left(\mathcal{N}_{y}\right)=M\label{eq:Aff a N_y equals M event,  in our language}
\end{equation}
where $\mathbf{j}\in\Sigma_{2}^{*}$ is a finite string of $0$s and
$1$s containing at least $n-m_{x}$ $1$s ($\#_{1}\left(\mathbf{j}\right)\geq n-m_{x}$),
as well as the other conditions required in order to correspond to
$\vec{a}\in\mathcal{A}^{\left(n-m_{x}\right)}$. As: 
\begin{equation}
H_{\mathbf{j}}\left(\mathcal{N}_{y}\right)=\frac{3^{\#_{1}\left(\mathbf{j}\right)}}{2^{\left|\mathbf{j}\right|}}\mathcal{N}_{y}+\chi_{3}\left(\mathbf{j}\right)\label{eq:H_j of N_y for Collatz example}
\end{equation}
$\chi_{3}\left(\mathbf{j}\right)$ is a rational number whose numerator
in simplest form is co-prime to $3$, upon taking congruences mod
$3^{n-m_{x}}$, we see that the event $H_{\mathbf{j}}\left(\mathcal{N}_{y}\right)=M$
forces the congruence $\chi_{3}\left(\mathbf{j}\right)\overset{3^{n-m_{x}}}{\equiv}M$
to hold. This then takes us to considering the $3$-adic distribution
of $\chi_{3}$.

However, we will \emph{not }take the non-archimedean route. Instead,
using (\ref{eq:Aff a N_y equals M event,  in our language}) and (\ref{eq:H_j of N_y for Collatz example}),
we arrive at an archimedean analogue for the congruence condition,
namely: 
\begin{equation}
M-\chi_{3}\left(\mathbf{j}\right)=\frac{3^{\#_{1}\left(\mathbf{j}\right)}}{2^{\left|\mathbf{j}\right|}}\mathcal{N}_{y}
\end{equation}
Note that Tao's congruence is just a statement that $\chi_{3}\left(\mathbf{j}\right)$
is $3$-adically close to $M$: 
\begin{align*}
\chi_{3}\left(\mathbf{j}\right) & \overset{3^{n-m_{x}}}{\equiv}M\\
 & \Updownarrow\\
\left|\chi_{3}\left(\mathbf{j}\right)-M\right|_{3} & \leq\frac{1}{3^{n-m_{x}}}
\end{align*}
and that the $3$-adic distance between $\chi_{3}\left(\mathbf{j}\right)$
and $M$ tends to $0$ as $n\rightarrow\infty$, and hence, as $x\rightarrow\infty$.
This set-up tells us precisely what we need to do: we will change
Tao's $3$-adic closeness condition into an archimedean closeness
condition.
\end{example}
\vphantom{}

So, considering our approximation formula (\ref{eq:approximate formula for passage location}),
let $L\in I_{y}$ (where $A$ is to be determined), let $\vec{\mathbf{V}}\in\mathcal{A}^{\left(L-m_{x}\right)}$,
and let $z\in S^{\prime}$. The event $H_{\vec{\mathbf{V}}}\left(\mathcal{N}_{y}\right)=z$
is equivalent to:

\begin{equation}
z=M_{H}\left(\vec{\mathbf{V}}\right)\mathcal{N}_{y}+\chi_{H}\left(\vec{\mathbf{V}}\right)\label{eq:our aff equals M event}
\end{equation}
where we are using $z$ where Tao used $M$. We now proceed like so:

\vphantom{} 
\begin{prop}
\label{prop:approximate formula for passage,  archimedean case}Let
everything be as given in \textbf{Assumption \ref{assu:initial parameter assumptions}}
and \textbf{Lemma \ref{lem:approximate formula for passage location}}.
Then, there exists a choice of $\alpha$ and parameters $A_{0},A_{1}>1$
satisfying the inequalities of \textbf{Proposition \ref{prop:gamma inequality proposition}}
so that:

\begin{equation}
\mathbb{P}\left(\textrm{Pass}_{x}\left(\mathcal{N}_{y}\right)\in S\right)=\frac{\left(1+O\left(x^{-c}\right)\right)\overline{\rho}_{H}^{\frac{d}{p}\ln^{0.6}x}}{C_{K}\left(\alpha-1\right)d\ln y}\sum_{L\in I_{y}}\rho_{H}^{d\left(L-m_{x}\right)}\sum_{\vec{\mathbf{V}}\in\mathcal{A}^{\left(L-m_{x}\right)}}\sum_{\begin{array}{c}
z\in S^{\prime}\\
\left\Vert X_{H}\left(\vec{\mathbf{V}}\right)-z\right\Vert _{\textrm{Mink}}\leq x^{\textrm{c}_{\alpha}}
\end{array}}\left\Vert z\right\Vert _{\textrm{Mink}}^{-d}\label{eq:approximate formula for passage, ready for Fourier transferrence}
\end{equation}
as $x\rightarrow\infty$, where: 
\begin{equation}
\textrm{c}_{\alpha}\overset{\textrm{def}}{=}\alpha^{3}-\left(\alpha-1\right)\left(1-\frac{\ln\rho_{H}^{-1}}{\ln A_{1}}\right)
\end{equation}
is slightly larger than $1$, where $\frac{\ln\rho_{H}^{-1}}{\ln A_{1}}<1$,
by \textbf{Proposition \ref{prop:gamma inequality proposition}}. 
\end{prop}
Proof: Let everything be as given, and suppose the event $H_{\vec{\mathbf{V}}}\left(\mathcal{N}_{y}\right)=z$
occurs. Since: 
\begin{align*}
H_{\vec{\mathbf{V}}}\left(\mathcal{N}_{y}\right) & =z\\
 & \Updownarrow\\
M_{H}\left(\vec{\mathbf{V}}\right)\mathcal{N}_{y} & =z-X_{H}\left(\vec{\mathbf{V}}\right)
\end{align*}
By (\ref{eq:selective estimate for M_H of v arrow}), since $\vec{\mathbf{V}}\in\mathcal{A}^{\left(L-m_{x}\right)}$,
we have that: 
\begin{equation}
\left\Vert M_{H}\left(\vec{\mathbf{V}}\right)\right\Vert _{\textrm{Mink}}\leq\overline{\rho}_{H}^{\frac{1}{p}\ln^{0.6}x}\rho_{H}^{L-m_{x}}\leq\overline{\rho}_{H}^{\frac{1}{p}\ln^{0.6}x}x^{-\left(\alpha-1\right)\left(1-\frac{\ln\rho_{H}^{-1}}{\ln A_{1}}\right)-\frac{\ln\rho_{H}^{-1}}{\ln^{0.2}x}}
\end{equation}
Combining this with the fact that: 
\begin{equation}
\left\Vert \mathcal{N}_{y}\right\Vert _{\textrm{Mink}}\leq y^{\alpha}\leq x^{\alpha^{3}}
\end{equation}
we obtain the estimate: 
\begin{equation}
\left\Vert X_{H}\left(\vec{\mathbf{V}}\right)-z\right\Vert _{\textrm{Mink}}=\left\Vert M_{H}\left(\vec{\mathbf{V}}\right)\mathcal{N}_{y}\right\Vert _{\textrm{Mink}}\leq\overline{\rho}_{H}^{\frac{1}{p}\ln^{0.6}x}x^{\overbrace{\alpha^{3}-\left(\alpha-1\right)\left(1-\frac{\ln\rho_{H}^{-1}}{\ln A_{1}}\right)}^{\textrm{c}_{\alpha}}}\textrm{ as }x\rightarrow\infty
\end{equation}
The exponent of $x$ on the right is a positive real number ever-so-slightly
larger than $1$. As such, we have that $H_{\vec{\mathbf{V}}}\left(\mathcal{N}_{y}\right)=z$
implies: 
\begin{equation}
\left\Vert X_{H}\left(\vec{\mathbf{V}}\right)-z\right\Vert _{\textrm{Mink}}\leq x^{\textrm{c}_{\alpha}}
\end{equation}
for all sufficiently large $x$. Now, using the logarithmic distribution
of $\mathcal{N}_{y}$, we have: 
\begin{equation}
\mathbb{P}\left(H_{\vec{\mathbf{V}}}\left(\mathcal{N}_{y}\right)=z\right)=\mathbb{P}\left(\mathcal{N}_{y}=M_{H}\left(\vec{\mathbf{V}}\right)^{-1}\left(X_{H}\left(\vec{\mathbf{V}}\right)-z\right)\right)=\frac{\sum_{w\in U_{y}\cap\left\{ M_{H}\left(\vec{\mathbf{V}}\right)^{-1}\left(X_{H}\left(\vec{\mathbf{V}}\right)-z\right)\right\} }\left\Vert w\right\Vert _{\textrm{Mink}}^{-d}}{\sum_{w\in U_{y}}\left\Vert w\right\Vert _{\textrm{Mink}}^{-d}}\label{eq:affine probability}
\end{equation}
As $y\rightarrow\infty$, the denominator is given by our formula
from \textbf{Proposition \ref{lem:lattice sum asymptotic}}. As for
the numerator, since $\mathcal{N}_{y}$ takes values in $\mathcal{O}_{K}$,
the event $\mathcal{N}_{y}=M_{H}\left(\vec{\mathbf{V}}\right)^{-1}\left(X_{H}\left(\vec{\mathbf{V}}\right)-z\right)$
forces $M_{H}\left(\vec{\mathbf{V}}\right)^{-1}\left(X_{H}\left(\vec{\mathbf{V}}\right)-z\right)$
to be in $\mathcal{O}_{K}$, which gives us: 
\begin{equation}
\sum_{w\in U_{y}\cap\left\{ M_{H}\left(\vec{\mathbf{V}}\right)^{-1}\left(X_{H}\left(\vec{\mathbf{V}}\right)-z\right)\right\} }\left\Vert w\right\Vert _{\textrm{Mink}}^{-d}=\frac{\left[M_{H}\left(\vec{\mathbf{V}}\right)^{-1}\left(X_{H}\left(\vec{\mathbf{V}}\right)-z\right)\in U_{y}\right]}{\left\Vert M_{H}\left(\vec{\mathbf{V}}\right)^{-1}\left(X_{H}\left(\vec{\mathbf{V}}\right)-z\right)\right\Vert _{\textrm{Mink}}^{d}}
\end{equation}
As we saw, $M_{H}\left(\vec{\mathbf{V}}\right)^{-1}\left(X_{H}\left(\vec{\mathbf{V}}\right)-z\right)\in U_{y}$
implies $\left\Vert X_{H}\left(\vec{\mathbf{V}}\right)-z\right\Vert _{\textrm{Mink}}\leq x$.
Hence: 
\begin{align*}
\mathbb{P}\left(H_{\vec{\mathbf{V}}}\left(\mathcal{N}_{y}\right)=z\right) & \leq\left[M_{H}\left(\vec{\mathbf{V}}\right)^{-1}\left(X_{H}\left(\vec{\mathbf{V}}\right)-z\right)\in U_{y}\right]\frac{\left\Vert M_{H}\left(\vec{\mathbf{V}}\right)^{-1}\left(X_{H}\left(\vec{\mathbf{V}}\right)-z\right)\right\Vert _{\textrm{Mink}}^{-d}}{C_{K}\left(\alpha-1\right)d\ln y}\\
 & \leq\frac{\left\Vert M_{H}\left(\vec{\mathbf{V}}\right)\right\Vert _{\infty}^{d}}{C_{K}\left(\alpha-1\right)d\ln y}\frac{\left[M_{H}\left(\vec{\mathbf{V}}\right)^{-1}\left(X_{H}\left(\vec{\mathbf{V}}\right)-z\right)\in U_{y}\right]}{\left\Vert X_{H}\left(\vec{\mathbf{V}}\right)-z\right\Vert _{\textrm{Mink}}^{d}}\\
 & \leq\frac{\left\Vert M_{H}\left(\vec{\mathbf{V}}\right)\right\Vert _{\textrm{Mink}}^{d}}{C_{K}\left(\alpha-1\right)d\ln y}\frac{\left[\left\Vert X_{H}\left(\vec{\mathbf{V}}\right)-z\right\Vert _{\textrm{Mink}}\leq x\right]}{\left\Vert X_{H}\left(\vec{\mathbf{V}}\right)-z\right\Vert _{\textrm{Mink}}^{d}}
\end{align*}
Next, as: 
\begin{equation}
\left\Vert M_{H}\left(\vec{\mathbf{V}}\right)\right\Vert _{\textrm{Mink}}^{d}\leq\prod_{j=0}^{p-1}\left\Vert r_{j}\right\Vert _{\textrm{Mink}}^{d\Sigma\left(\vec{\mathbf{V}}\right)\cdot\mathbf{e}_{j}}\leq\overline{\rho}_{H}^{\frac{d}{p}\ln^{0.6}x}\rho_{H}^{d\left(L-m_{x}\right)}
\end{equation}
we have: 
\begin{equation}
\mathbb{P}\left(H_{\vec{\mathbf{V}}}\left(\mathcal{N}_{y}\right)=z\right)\leq\frac{\overline{\rho}_{H}^{\frac{d}{p}\ln^{0.6}x}\rho_{H}^{d\left(L-m_{x}\right)}}{C_{K}\left(\alpha-1\right)d\ln y}\frac{\left[\left\Vert X_{H}\left(\vec{\mathbf{V}}\right)-z\right\Vert _{\textrm{Mink}}\leq x^{\textrm{c}_{\alpha}}\right]}{\left\Vert X_{H}\left(\vec{\mathbf{V}}\right)-z\right\Vert _{\textrm{Mink}}^{d}}\textrm{ as }x\rightarrow\infty\label{eq:halfway done with refining the approximate formula}
\end{equation}

Thus: 
\begin{align*}
\mathbb{P}\left(\textrm{Pass}_{x}\left(\mathcal{N}_{y}\right)\in S\right) & \leq\sum_{L\in I_{y}}\sum_{\vec{\mathbf{V}}\in\mathcal{A}^{\left(L-m_{x}\right)}}\sum_{z\in S^{\prime}}\frac{\overline{\rho}_{H}^{\frac{d}{p}\ln^{0.6}x}\rho_{H}^{d\left(L-m_{x}\right)}}{C_{K}\left(\alpha-1\right)d\ln y}\frac{\left[\left\Vert X_{H}\left(\vec{\mathbf{V}}\right)-z\right\Vert _{\textrm{Mink}}\leq x^{\textrm{c}_{\alpha}}\right]}{\left\Vert X_{H}\left(\vec{\mathbf{V}}\right)-z\right\Vert _{\textrm{Mink}}^{d}}\\
 & =\sum_{L\in I_{y}}\frac{\overline{\rho}_{H}^{\frac{d}{p}\ln^{0.6}x}\rho_{H}^{d\left(L-m_{x}\right)}}{C_{K}\left(\alpha-1\right)d\ln y}\sum_{\vec{\mathbf{V}}\in\mathcal{A}^{\left(L-m_{x}\right)}}\sum_{\begin{array}{c}
z\in S^{\prime}\\
\left\Vert X_{H}\left(\vec{\mathbf{V}}\right)-z\right\Vert _{\textrm{Mink}}\leq x^{\textrm{c}_{\alpha}}
\end{array}}\left\Vert X_{H}\left(\vec{\mathbf{V}}\right)-z\right\Vert _{\textrm{Mink}}^{-d}
\end{align*}
Meanwhile, using (\ref{eq:restrictions on z}), we get: 
\begin{equation}
\left\Vert z\right\Vert _{\textrm{Mink}}=O\left(x^{1+\frac{\ln\rho_{H}^{-1}}{\ln A_{1}}}\right)
\end{equation}
Combining this with (\ref{eq:selective estimate for X_H of v arrow}),
we get: 
\begin{align}
\left\Vert \frac{X_{H}\left(\vec{\mathbf{V}}\right)}{z}\right\Vert _{\textrm{Mink}} & \ll\frac{\left(R_{H}^{2}\overline{\rho}_{H}^{1/p}\right)^{\ln^{0.6}x}}{x^{1+\frac{\ln\rho_{H}^{-1}}{\ln A_{1}}}}
\end{align}
which will decay to $0$ as $x\rightarrow\infty$. This gives us:
\begin{equation}
\left\Vert X_{H}\left(\vec{\mathbf{V}}\right)-z\right\Vert _{\textrm{Mink}}^{-d}=\left\Vert z\right\Vert _{\textrm{Mink}}^{d}\left(1+O\left(x^{-1}\right)\right)\textrm{ as }x\rightarrow\infty\label{eq:X - z equals z times 1 plus O}
\end{equation}
from which (\ref{eq:halfway done with refining the approximate formula})
becomes: 
\begin{equation}
\mathbb{P}\left(H_{\vec{\mathbf{V}}}\left(\mathcal{N}_{y}\right)=z\right)\ll\left[\left\Vert X_{H}\left(\vec{\mathbf{V}}\right)-z\right\Vert _{\textrm{Mink}}\leq x^{\textrm{c}_{\alpha}}\right]\frac{1+O\left(x^{-c}\right)}{C_{K}\left(\alpha-1\right)d\ln y}\frac{\overline{\rho}_{H}^{\frac{d}{p}\ln^{0.6}x}\rho_{H}^{d\left(L-m_{x}\right)}}{\left\Vert z\right\Vert _{\textrm{Mink}}^{d}}\textrm{ as }y\rightarrow\infty\label{eq:affine probability asymptotic}
\end{equation}
Using this estimate on the right-hand side of our approximate formula
(\ref{eq:approximate formula for passage location}) then yields (\ref{eq:approximate formula for passage, ready for Fourier transferrence}).

Q.E.D.

\vphantom{}The key ramification of our archimedean argument is that
the analogue of Tao's \textbf{Lemma 5.3} ends up holding without any
assumptions on oscillation or Fourier decay. Next, we define the analogue
of the function Tao calls $c_{n}\left(X\right)$.

\vphantom{} 
\begin{defn}
Let everything be as given in \textbf{Proposition \ref{prop:approximate formula for passage,  archimedean case}}.
Let $f_{x,L}:K\rightarrow\left[0,\infty\right)$ be defined by: 
\begin{equation}
f_{x,L}\left(w\right)\overset{\textrm{def}}{=}\sum_{\begin{array}{c}
z\in S^{\prime}\\
\left\Vert z-w\right\Vert _{\textrm{Mink}}\leq x^{\textrm{c}_{\alpha}}
\end{array}}\frac{\omega_{x}\left(L\right)}{\left\Vert z\right\Vert _{\textrm{Mink}}^{d}},\textrm{ }\forall w\in K\label{eq:def of f_x}
\end{equation}
where: 
\begin{equation}
\omega_{x}\left(L\right)\overset{\textrm{def}}{=}\frac{p}{p-1}\left(\frac{p}{p-1}\right)^{p+\ln^{0.6}x}\left(\rho_{H}^{d}\left(p-1\right)^{2}\right)^{L-m_{x}}\overline{\rho}_{H}^{\frac{d}{p}\ln^{0.6}x}\label{eq:def of omega_x weight}
\end{equation}

\vphantom{} 
\end{defn}
Next, we have the following identities which we will use to facilitate
the analogue of expression 5.22 from Tao's paper.

\vphantom{} 
\begin{prop}
Let everything be as given in \textbf{Proposition \ref{prop:approximate formula for passage,  archimedean case}}.
Then, for any $L\geq1$ and any $\mathbf{n}\in\mathbb{N}_{0}^{p}$:
\begin{equation}
\mathbb{E}\left(f_{x,L}\left(\mathbf{S}_{H}\left(L-m_{x}\right)\right)\right)=\sum_{\vec{\mathbf{V}}\in\left(\mathbb{N}_{1}B\right)^{L-m_{x}}}\mathbb{P}\left(\mathbf{Run}\left(p\right)^{L-m_{x}}=\vec{\mathbf{V}}\right)f_{x,L}\left(X_{H}\left(\vec{\mathbf{V}}\right)\right)\label{eq:expectation identity}
\end{equation}
\begin{equation}
\mathbb{E}\left(f_{x,L}\left(\mathbf{S}_{H}\left(L-m_{x}\right)\right)\right)=\sum_{z\in S^{\prime}}\frac{\omega_{x}\left(L\right)}{\left\Vert z\right\Vert _{\textrm{Mink}}^{d}}\mathbb{P}\left(\left\Vert \mathbf{S}_{H}\left(L-m_{x}\right)-z\right\Vert _{\textrm{Mink}}\leq x^{\textrm{c}_{\alpha}}\right)\label{eq:Tao page 28 transformation}
\end{equation}
\end{prop}
Proof: To prove (\ref{eq:Tao page 28 transformation}), we write:
\begin{align*}
\mathbb{E}\left(f_{x,L}\left(\mathbf{S}_{H}\left(L-m_{x}\right)\right)\right) & =\sum_{w\in K}f_{x,L}\left(w\right)\mathbb{P}\left(\mathbf{S}_{H}\left(L-m_{x}\right)=w\right)\\
 & =\sum_{w\in K}\left(\sum_{\begin{array}{c}
z\in S^{\prime}\\
\left|w-z\right|_{\infty}\leq x
\end{array}}\frac{\omega_{x}\left(L\right)}{\left\Vert z\right\Vert _{\textrm{Mink}}^{d}}\right)\mathbb{P}\left(\mathbf{S}_{H}\left(L-m_{x}\right)=w\right)\\
\left(\textrm{swap sums}\right); & =\sum_{z\in S^{\prime}}\frac{\omega_{x}\left(L\right)}{\left\Vert z\right\Vert _{\textrm{Mink}}^{d}}\sum_{w\in K}\left[\left\Vert w-z\right\Vert _{\textrm{Mink}}\leq x^{\textrm{c}_{\alpha}}\right]\mathbb{P}\left(\mathbf{S}_{H}\left(L-m_{x}\right)=w\right)\\
 & =\sum_{z\in S^{\prime}}\frac{\omega_{x}\left(L\right)}{\left\Vert z\right\Vert _{\textrm{Mink}}^{d}}\mathbb{P}\left(\left\Vert \mathbf{S}_{H}\left(L-m_{x}\right)-z\right\Vert _{\textrm{Mink}}\leq x^{\textrm{c}_{\alpha}}\right)
\end{align*}
To prove (\ref{eq:expectation identity}), observe that: 
\begin{eqnarray*}
 & \sum_{z\in S^{\prime}}\frac{\omega_{x}\left(L\right)}{\left\Vert z\right\Vert _{\textrm{Mink}}^{d}}\mathbb{P}\left(\left\Vert \mathbf{S}_{H}\left(L-m_{x}\right)-z\right\Vert _{\textrm{Mink}}\leq x^{\textrm{c}_{\alpha}}\right)\\
 & =\\
 & \sum_{z\in S^{\prime}}\frac{\omega_{x}\left(L\right)}{\left\Vert z\right\Vert _{\textrm{Mink}}^{d}}\mathbb{P}\left(\left\Vert X_{H}\left(\mathbf{Run}\left(p\right)^{L-m_{x}}\right)-z\right\Vert _{\textrm{Mink}}\leq x^{\textrm{c}_{\alpha}}\right)\\
 & =\\
 & \sum_{\vec{\mathbf{V}}\in\left(\mathbb{N}_{1}B\right)^{L-m_{x}}}\mathbb{P}\left(\mathbf{Run}\left(p\right)^{L-m_{x}}=\vec{\mathbf{V}}\right)\sum_{\begin{array}{c}
z\in S^{\prime}\\
\left\Vert X_{H}\left(\vec{\mathbf{V}}\right)-z\right\Vert _{\textrm{Mink}}\leq x^{\textrm{c}_{\alpha}}
\end{array}}\frac{\omega_{x}\left(L\right)}{\left\Vert z\right\Vert _{\textrm{Mink}}^{d}}\\
 & =\\
 & \sum_{\vec{\mathbf{V}}\in\left(\mathbb{N}_{1}B\right)^{L-m_{x}}}\mathbb{P}\left(\mathbf{Run}\left(p\right)^{L-m_{x}}=\vec{\mathbf{V}}\right)f_{x,L}\left(X_{H}\left(\vec{\mathbf{V}}\right)\right)
\end{eqnarray*}

Q.E.D.

\vphantom{}

Our next proposition gives our analogue of expression 5.22 from Tao's
paper:

\vphantom{} 
\begin{prop}
\label{prop:Fourier transference step 1}Let everything be as given
in \textbf{Proposition \ref{prop:approximate formula for passage,  archimedean case}}.
Let $\vec{\mathcal{V}}=\left(\mathcal{V}_{1},\ldots,\mathcal{V}_{L-m_{x}}\right)$
be a random variable of type $\mathbf{Run}\left(p\right)^{L-m_{x}}$,
where $L\in I_{y}$. Then, there is a constant $c>0$ so that: 
\begin{equation}
\sum_{\vec{\mathbf{V}}\in\mathcal{A}^{\left(L-m_{x}\right)}}\sum_{\begin{array}{c}
z\in S^{\prime}\\
\left\Vert X_{H}\left(\vec{\mathbf{V}}\right)-z\right\Vert _{\textrm{Mink}}\leq x^{\textrm{c}_{\alpha}}
\end{array}}\frac{\overline{\rho}_{H}^{\frac{d}{p}\ln^{0.6}x}\rho_{H}^{d\left(L-m_{x}\right)}}{\left\Vert z\right\Vert _{\textrm{Mink}}^{d}}=\mathbb{E}\left(\left[\vec{\mathcal{V}}\in\mathcal{A}^{\left(L-m_{x}\right)}\right]f_{x,L}\left(X_{H}\left(\vec{\mathcal{V}}\right)\right)\right)+O\left(x^{-c}\right)\label{eq:Fourier transference step 1}
\end{equation}
as $x\rightarrow\infty$.
\end{prop}
Proof: Like with (\ref{eq:expectation identity}), the right-hand
side of (\ref{eq:Fourier transference step 1}) can be written as:

\begin{equation}
\sum_{\vec{\mathbf{V}}\in\mathcal{A}^{\left(L-m_{x}\right)}}\mathbb{P}\left(\mathbf{Run}\left(p\right)^{L-m_{x}}=\vec{\mathbf{V}}\right)\rho_{H}^{d\left(L-m_{x}\right)}\sum_{\begin{array}{c}
z\in S^{\prime}\\
\left\Vert X_{H}\left(\vec{\mathbf{V}}\right)-z\right\Vert _{\textrm{Mink}}\leq x^{\textrm{c}_{\alpha}}
\end{array}}\left\Vert z\right\Vert _{\textrm{Mink}}^{-d}\label{eq:expectation to be transformed}
\end{equation}
Now, let $\Delta$ denote the real absolute value of the difference
of (\ref{eq:expectation to be transformed}) and the left-hand side
of (\ref{eq:Fourier transference step 1}). Then: 
\begin{equation}
\Delta\leq\sum_{\vec{\mathbf{V}}\in\mathcal{A}^{\left(L-m_{x}\right)}}\sum_{\begin{array}{c}
z\in S^{\prime}\\
\left\Vert X_{H}\left(\vec{\mathbf{V}}\right)-z\right\Vert _{\textrm{Mink}}\leq x^{\textrm{c}_{\alpha}}
\end{array}}\frac{\left|\omega_{x}\left(L\right)\mathbb{P}\left(\mathbf{Run}\left(p\right)^{L-m_{x}}=\vec{\mathbf{V}}\right)-\overline{\rho}_{H}^{\frac{d}{p}\ln^{0.6}x}\rho_{H}^{d\left(L-m_{x}\right)}\right|}{\left\Vert z\right\Vert _{\textrm{Mink}}^{d}}\label{eq:big delta upper bound}
\end{equation}
Here: 
\begin{equation}
\mathbb{P}\left(\mathbf{Run}\left(p\right)^{L-m_{x}}=\vec{\mathbf{V}}\right)=\frac{p-1}{p}\frac{\left(1-\frac{1}{p}\right)^{\Sigma\left(\Sigma\left(\vec{\mathbf{V}}\right)\right)}}{\left(p-1\right)^{2\left(L-m_{x}\right)}}\prod_{k=2}^{L-m_{x}}\left[\mathbf{v}_{k}\neq\mathbf{v}_{k-1}\right]
\end{equation}
where, recall, $\vec{\mathbf{V}}$ is of the form $\mathbf{m}\vec{\mathbf{v}}$,
for $\vec{\mathbf{v}}=\left(\mathbf{v}_{1},\ldots,\mathbf{v}_{L-m_{x}}\right)$.
As such: 
\begin{equation}
\omega_{x}\left(L\right)\mathbb{P}\left(\mathbf{Run}\left(p\right)^{L-m_{x}}=\vec{\mathbf{V}}\right)-\overline{\rho}_{H}^{\frac{d}{p}\ln^{0.6}x}\rho_{H}^{d\left(L-m_{x}\right)}
\end{equation}
becomes: 
\begin{equation}
\left(\left(\frac{p}{p-1}\right)^{p-1+\ln^{0.6}x}\left(1-\frac{1}{p}\right)^{\Sigma\left(\Sigma\left(\vec{\mathbf{V}}\right)\right)}\prod_{k=2}^{L-m_{x}}\left[\mathbf{v}_{k}\neq\mathbf{v}_{k-1}\right]-1\right)\overline{\rho}_{H}^{\frac{d}{p}\ln^{0.6}x}\rho_{H}^{d\left(L-m_{x}\right)}
\end{equation}
If $\prod_{k=2}^{L-m_{x}}\left[\mathbf{v}_{k}\neq\mathbf{v}_{k-1}\right]=0$,
then the above goes to $0$ as $x\rightarrow\infty$, because $\rho_{H}^{d}<1$.
So, suppose $\prod_{k=2}^{L-m_{x}}\left[\mathbf{v}_{k}\neq\mathbf{v}_{k-1}\right]=1$.
By (\ref{eq:condition on sum of n-sum of v-arrow}), we have: 
\begin{equation}
p\left(L-m_{x}\right)-\ln^{0.6}x<\Sigma\left(\Sigma\left(\vec{\mathbf{V}}\right)\right)<p\left(L-m_{x}\right)+\ln^{0.6}x
\end{equation}
and so: 
\begin{equation}
\left(\frac{p}{p-1}\right)^{p-1}\left(\frac{p-1}{p}\right)^{p\left(L-m_{x}\right)}<\left(\frac{p}{p-1}\right)^{p-1+\ln^{0.6}x}\left(1-\frac{1}{p}\right)^{\Sigma\left(\Sigma\left(\vec{\mathbf{V}}\right)\right)}<\left(\frac{p}{p-1}\right)^{p-1-\ln^{0.6}x}\left(\frac{p-1}{p}\right)^{p\left(L-m_{x}\right)}\label{eq:sandwich}
\end{equation}
Since: 
\begin{equation}
\sup_{L\in I_{y}}\left(\frac{p-1}{p}\right)^{p\left(L-m_{x}\right)}=O\left(x^{-c}\right)\textrm{ as }x\rightarrow\infty
\end{equation}
for some $c>0$, we see that the central expression of (\ref{eq:sandwich})
tends to $0$ as $x\rightarrow\infty$, giving us: 
\begin{equation}
\left|\omega_{x}\left(L\right)\mathbb{P}\left(\mathbf{Run}\left(p\right)^{L-m_{x}}=\vec{\mathbf{V}}\right)-\overline{\rho}_{H}^{\frac{d}{p}\ln^{0.6}x}\rho_{H}^{d\left(L-m_{x}\right)}\right|=\underbrace{\overline{\rho}_{H}^{\frac{d}{p}\ln^{0.6}x}\rho_{H}^{d\left(L-m_{x}\right)}}_{O\left(x^{-c}\right)}\textrm{ as }x\rightarrow\infty
\end{equation}
for some $c>0$. Hence, $\Delta=0$, from which (\ref{eq:Fourier transference step 1})
follows.

Q.E.D.

\vphantom{}

Finally, we have our analogue of Tao's \textbf{Lemma 5.3}:

\vphantom{} 
\begin{prop}
\label{prop:boundedness of the omega_x sum}If: 
\begin{equation}
\left(p-1\right)^{2}\rho_{H}^{d}<1
\end{equation}
then: 
\begin{equation}
\sup_{L\in I_{y}}\sum_{z\in S^{\prime}}\frac{\omega_{x}\left(L\right)}{\left\Vert z\right\Vert _{\textrm{Mink}}^{d}}\ll x^{-\left(\alpha-1\right)c}\label{eq:Tao Lemma 5.3 analogue}
\end{equation}
for both $y=x^{\alpha}$ and $y=x^{\alpha^{2}}$. 
\end{prop}
Proof: We begin by writing: 
\begin{align*}
\sum_{z\in S^{\prime}}\frac{\omega_{x}\left(L\right)}{\left\Vert z\right\Vert _{\textrm{Mink}}^{d}} & =\left(\frac{p}{p-1}\right)^{p=1+\ln^{0.6}x}\left(\left(p-1\right)^{2}\rho_{H}^{d}\right)^{L-m_{x}}\overline{\rho}_{H}^{\frac{d}{p}\ln^{0.6}x}\sum_{z\in S^{\prime}}\left\Vert z\right\Vert _{\textrm{Mink}}^{-d}\\
 & \ll\left(\frac{p}{p-1}\overline{\rho}_{H}^{d/p}\right)^{\ln^{0.6}x}\left(\left(p-1\right)^{2}\rho_{H}^{d}\right)^{L-m_{x}}\ln^{0.7}x
\end{align*}
where we used \textbf{Proposition \ref{lem:lattice sum asymptotic}
}to bound the $z$ sum by $\ln^{0.7}x$, using the bounds on $z$
from (\ref{eq:restrictions on z}). If $\left(p-1\right)^{2}\rho_{H}^{d}<1$,
then: 
\begin{equation}
\left(\left(p-1\right)^{2}\rho_{H}^{d}\right)^{L-m_{x}}\ll\left(\left(p-1\right)^{2}\rho_{H}^{d}\right)^{\left(\alpha-1\right)\left(\frac{1}{\ln\rho_{H}^{-1}}-\frac{1}{\ln A_{1}}\right)\ln x+\ln^{0.8}x}\ll x^{-\left(\alpha-1\right)c}\textrm{ as }x\rightarrow\infty
\end{equation}
for some $c>0$, and we get: 
\begin{equation}
\sum_{z\in S^{\prime}}\frac{\omega_{x}\left(L\right)}{\left\Vert z\right\Vert _{\textrm{Mink}}^{d}}\ll x^{-\left(\alpha-1\right)c}
\end{equation}

Q.E.D.

\vphantom{}

With this bound in hand, we can do as Tao does on page 28 immediately
after his proof of \textbf{Lemma 5.3} and rewrite expression 5.22.

\vphantom{} 
\begin{prop}
\label{prop:removing the indicator function from the expectation}If
$\left(p-1\right)^{2}\rho_{H}^{d}<1$, then, for all $L\in I_{y}$,
we have: 
\begin{equation}
\mathbb{E}\left(\left[\vec{\mathcal{V}}\in\mathcal{A}^{\left(L-m_{x}\right)}\right]f_{x,L}\left(X_{H}\left(\vec{\mathcal{V}}\right)\right)\right)=\mathbb{E}\left(f_{x,L}\left(X_{H}\left(\vec{\mathcal{V}}\right)\right)\right)+O\left(x^{-\left(\alpha-1\right)c}\right)\label{eq:removing the indicator function from the expectation}
\end{equation}
as $x\rightarrow\infty$, for some $c>0$. 
\end{prop}
Proof: By (\ref{eq:Need 3}), we know that: 
\begin{equation}
\mathbb{P}\left(\mathbf{Run}\left(p\right)^{L-m_{x}}\notin\mathcal{A}^{\left(L-m_{x}\right)}\right)\ll\ln^{-10}x\textrm{ as }x\rightarrow\infty
\end{equation}
Now, write: 
\begin{equation}
\mathbb{E}\left(\left[\vec{\mathcal{V}}\in\mathcal{A}^{\left(L-m_{x}\right)}\right]f_{x,L}\left(X_{H}\left(\vec{\mathcal{V}}\right)\right)\right)=\mathbb{E}\left(f_{x,L}\left(X_{H}\left(\vec{\mathcal{V}}\right)\right)\right)-\mathbb{E}\left(\left[\vec{\mathcal{V}}\notin\mathcal{A}^{\left(L-m_{x}\right)}\right]f_{x,L}\left(X_{H}\left(\vec{\mathcal{V}}\right)\right)\right)
\end{equation}
The term on the far right is: 
\begin{equation}
\mathbb{E}\left(\left[\vec{\mathcal{V}}\notin\mathcal{A}^{\left(L-m_{x}\right)}\right]f_{x,L}\left(X_{H}\left(\vec{\mathcal{V}}\right)\right)\right)=\sum_{\vec{\mathbf{V}}\in\left(\mathbb{N}_{1}B\right)^{L-m_{x}}\backslash\mathcal{A}^{\left(L-m_{x}\right)}}\mathbb{P}\left(\mathbf{Run}\left(p\right)^{L-m_{x}}=\vec{\mathbf{V}}\right)f_{x,L}\left(X_{H}\left(\vec{\mathbf{V}}\right)\right)
\end{equation}
Here, we use \textbf{Proposition \ref{prop:boundedness of the omega_x sum}}
to get the bound: 
\begin{equation}
f_{x,L}\left(X_{H}\left(\vec{\mathbf{V}}\right)\right)=\sum_{\begin{array}{c}
z\in S^{\prime}\\
\left\Vert X_{H}\left(\vec{\mathbf{V}}\right)-z\right\Vert _{\textrm{Mink}}\leq x^{\textrm{c}_{\alpha}}
\end{array}}\frac{\omega_{x}\left(L\right)}{\left\Vert z\right\Vert _{\textrm{Mink}}^{d}}\leq\sum_{z\in S^{\prime}}\frac{\omega_{x}\left(L\right)}{\left\Vert z\right\Vert _{\textrm{Mink}}^{d}}\ll x^{-\left(\alpha-1\right)c}\textrm{ as }x\rightarrow\infty
\end{equation}
as such: 
\begin{align*}
\mathbb{E}\left(\left[\vec{\mathcal{V}}\notin\mathcal{A}^{\left(L-m_{x}\right)}\right]f_{x,L}\left(X_{H}\left(\vec{\mathcal{V}}\right)\right)\right) & \ll x^{-\left(\alpha-1\right)c}\sum_{\vec{\mathbf{V}}\in\left(\mathbb{N}_{1}B\right)^{L-m_{x}}\backslash\mathcal{A}^{\left(L-m_{x}\right)}}\mathbb{P}\left(\mathbf{Run}\left(p\right)^{L-m_{x}}=\vec{\mathbf{V}}\right)\\
 & =x^{-\left(\alpha-1\right)c}\mathbb{P}\left(\mathbf{Run}\left(p\right)^{L-m_{x}}\notin\mathcal{A}^{\left(L-m_{x}\right)}\right)\\
 & \ll\frac{\ln^{-10}x}{x^{\left(\alpha-1\right)c}}
\end{align*}

Q.E.D.

\vphantom{}We now have everything needed to complete the proof of
the second half of \textbf{Lemma \ref{lem:stabilization of first passage}}.

\vphantom{}\textbf{Proof of} \textbf{Equation \ref{eq:passage location comparison}}:\textbf{
}Using the parameters as specified by \textbf{Assumption \ref{assu:initial parameter assumptions}},
\textbf{Proposition \ref{prop:approximate formula for passage,  archimedean case}
}yields: 
\begin{align*}
\mathbb{P}\left(\textrm{Pass}_{x}\left(\mathcal{N}_{y}\right)\in S\right) & =\frac{\left(1+O\left(x^{-c}\right)\right)\overline{\rho}_{H}^{\frac{d}{p}\ln^{0.6}x}}{C_{K}\left(\alpha-1\right)d\ln y}\sum_{L\in I_{y}}\rho_{H}^{d\left(L-m_{x}\right)}\sum_{\vec{\mathbf{V}}\in\mathcal{A}^{\left(L-m_{x}\right)}}\sum_{\begin{array}{c}
z\in S^{\prime}\\
\left\Vert X_{H}\left(\vec{\mathbf{V}}\right)-z\right\Vert _{\textrm{Mink}}\leq x^{\textrm{c}_{\alpha}}
\end{array}}\left\Vert z\right\Vert _{\textrm{Mink}}^{-d}\\
\left(\mathbf{Prop.\textrm{ }\ref{prop:Fourier transference step 1}}\right); & =\frac{1+O\left(x^{-c}\right)}{C_{K}\left(\alpha-1\right)d\ln y}\sum_{L\in I_{y}}\left(\mathbb{E}\left(\left[\vec{\mathcal{V}}\in\mathcal{A}^{\left(L-m_{x}\right)}\right]f_{x,L}\left(X_{H}\left(\vec{\mathcal{V}}\right)\right)\right)+O\left(x^{-c^{\prime}}\right)\right)\\
\left(\mathbf{Prop.\textrm{ }\ref{prop:removing the indicator function from the expectation}}\right); & =\frac{1+O\left(x^{-c}\right)}{C_{K}\left(\alpha-1\right)d\ln y}\left(\sum_{L\in I_{y}}\mathbb{E}\left(f_{x,L}\left(X_{H}\left(\vec{\mathcal{V}}\right)\right)\right)+O\left(\left|I_{y}\right|x^{-c^{\prime}}\right)\right)
\end{align*}
where the application of \textbf{Proposition \ref{prop:removing the indicator function from the expectation}}
is allowed due to the dimension constraint MH3 from \textbf{Theorem
\ref{thm:almost boundedness}}. Here: 
\begin{equation}
\left|I_{y}\right|\leq\left(\alpha-1\right)\frac{\ln y}{\ln\rho_{H}^{-1}}-2\ln^{0.8}x\ll\frac{\alpha-1}{\ln\rho_{H}^{-1}}\ln y\textrm{ as }x\rightarrow\infty
\end{equation}
By (\ref{eq:Tao page 28 transformation}), we can write:

\begin{equation}
\mathbb{E}\left(f_{x,L}\left(X_{H}\left(\vec{\mathcal{V}}\right)\right)\right)=\sum_{z\in S^{\prime}}\frac{\omega_{x}\left(L\right)}{\left\Vert z\right\Vert _{\textrm{Mink}}^{d}}\mathbb{P}\left(\left\Vert \mathbf{S}_{H}\left(L-m_{x}\right)-z\right\Vert _{\textrm{Mink}}\leq x^{\textrm{c}_{\alpha}}\right)
\end{equation}
Hence: 
\begin{align*}
\mathbb{P}\left(\textrm{Pass}_{x}\left(\mathcal{N}_{y}\right)\in S\right) & =\frac{1+O\left(x^{-c}\right)}{C_{K}\left(\alpha-1\right)d\ln y}\left(\sum_{L\in I_{y}}\sum_{z\in S^{\prime}}\frac{\omega_{x}\left(L\right)}{\left\Vert z\right\Vert _{\textrm{Mink}}^{d}}\mathbb{P}\left(\left\Vert \mathbf{S}_{H}\left(L-m_{x}\right)-z\right\Vert _{\textrm{Mink}}\leq x^{\textrm{c}_{\alpha}}\right)+O\left(\left|I_{y}\right|x^{-c^{\prime}}\right)\right)\\
 & \leq\frac{1+O\left(x^{-c}\right)}{C_{K}\left(\alpha-1\right)d\ln y}\left(\left|I_{y}\right|\sup_{L\in I_{y}}\sum_{z\in S^{\prime}}\frac{\omega_{x}\left(L\right)}{\left\Vert z\right\Vert _{\textrm{Mink}}^{d}}+O\left(\left|I_{y}\right|x^{-c^{\prime}}\right)\right)\\
\left(\left|I_{y}\right|\ll\frac{\alpha-1}{\ln\rho_{H}^{-1}}\ln y\right); & \ll\frac{1+O\left(x^{-c}\right)}{C_{K}d\ln\rho_{H}^{-1}}\left(\sup_{L\in I_{y}}\sum_{z\in S^{\prime}}\frac{\omega_{x}\left(L\right)}{\left\Vert z\right\Vert _{\textrm{Mink}}^{d}}+O\left(x^{-c^{\prime}}\right)\right)\\
\left(\left(p-1\right)^{2}\rho_{H}^{d}<1\right); & =\frac{1+O\left(x^{-c}\right)}{C_{K}d\ln\rho_{H}^{-1}}\left(O\left(x^{-\left(\alpha-1\right)c^{\prime\prime}}\right)+O\left(x^{-c^{\prime}}\right)\right)\\
 & \ll x^{-\left(\alpha-1\right)c}
\end{align*}
for some constant $c$. Thus: 
\begin{equation}
\max\left\{ \mathbb{P}\left(\textrm{Pass}_{x}\left(\mathcal{N}_{x^{\alpha}}\right)\in S\right),\mathbb{P}\left(\textrm{Pass}_{x}\left(\mathcal{N}_{x^{\alpha^{2}}}\right)\in S\right)\right\} \ll x^{-\left(\alpha-1\right)c}
\end{equation}
Since $S\subseteq\left\{ z\in\mathcal{O}_{K}:\left|z\right|_{\infty}\leq x\right\} $
was arbitrary, we conclude that: 
\begin{align*}
d_{\textrm{TV}}\left(\textrm{Pass}_{x}\left(\mathcal{N}_{x^{\alpha}}\right),\textrm{Pass}_{x}\left(\mathcal{N}_{x^{\alpha^{2}}}\right)\right) & \leq\sup_{S}\left|\mathbb{P}\left(\textrm{Pass}_{x}\left(\mathcal{N}_{x^{\alpha}}\right)\in S\right)-\mathbb{P}\left(\textrm{Pass}_{x}\left(\mathcal{N}_{x^{\alpha^{2}}}\right)\in S\right)\right|\\
 & \ll x^{-\left(\alpha-1\right)c}
\end{align*}
which proves Equation (\ref{eq:passage location comparison}), and
with it, the rest of \textbf{Lemma \ref{lem:stabilization of first passage}}.

Q.E.D.

\section{Conclusion}

Whether together or independently, both authors intend to further
explore the material presented here. In that respect, the framework
we present for generalizing Tao's method to number fields will be
foundational for these future inquiries. Topics the first author plans
on exploring include generalizing Tao's non-archimedean arguments
to produce yet more proofs of almost-boundedness, as well as exploring
how Tao's methods may be applied to generalized Hydra maps whose branches
are fractional linear maps, rather than affine linear maps; the latter
direction would seem natural, given the great importance that fractional
linear transformations have played in the context of Fourier decay
for self-similar measures, such as in the Bourgain-Dyatlov paper \cite{Bourgain Dyatlov};
see also \cite{Sahlsten survey paper}.

\end{document}